\documentclass[12pt]{amsart}

\usepackage{amssymb}
\usepackage{times}
\usepackage{amsfonts}
\usepackage{color}

  \usepackage[pdftex]{graphicx}
  \usepackage[pdftex]{hyperref}
  \hypersetup{
        unicode=false,          
        pdftoolbar=true,        
        pdfmenubar=true,        
        pdffitwindow=false,     
        pdfstartview={FitH},    
        pdfnewwindow=true,      
        colorlinks=true,        
        linkcolor=red,          
        citecolor=blue,         
        filecolor=magenta,      
        urlcolor=cyan           
    }

\newtheorem{theorem}{Theorem}[section]

\newtheorem{lemma}[theorem]{Lemma}
\newtheorem{proposition}[theorem]{Proposition}
\newtheorem{definition}[theorem]{Definition}

\newtheorem{remark}[theorem]{Remark}

\newtheorem{example}[theorem]{Example}

\numberwithin{equation}{section}

\def\beq{\begin{equation}}
\def\eeq{\end{equation}}

\def\ben{\begin{enumerate}}
\def\een{\end{enumerate}}

\def\ss{\smallskip}

\def\cI{{\mathcal I}}

\def\cQ{{\mathcal Q}}

\def\cP{{\mathcal P}}

\def\cU{{\mathcal U}}

\def\bbP{\mathbb P}

\def\be{\mathbf{e}}
\def\wtilde{\widetilde}
\def\what{\widehat}

\def\cW{\mathcal W}

\def\vtr{\vartriangleright}

\def\ec{equality check}

\newcommand{\RR}{{\mathbb R}}
\newcommand{\PP}{\mathbb{P}}

\def\RR{\mathbb R}

\newcommand{\df}[1]{{\bf{#1}}{\index{#1}}}

\makeindex

\begin{document}

\title[Non-commutative Representations]{Non-commutative Representations of Families of $k^2$ Commutative Polynomials
in $2k^2$ Commuting Variables  }

\date{\today}

\author[H. Dym]{Harry Dym}
\email{dym@wisdom.weizmann.ac.il}

\author[B. Helton]{J. W. Helton}
\email{helton@math.ucsd.edu}

\author[C. Meier]{Caleb Meier}
\email{c1meier@math.ucsd.edu}

\thanks{Helton and Meier were partially funded by
NSF Research supported by NSF grants
DMS-0700758, DMS-0757212, and the Ford Motor Co.
We thank Klep and Vinnikov for valuable conversations.
}

\begin{abstract}
Given a collection $\cP =  \{p_1(x_1,\ldots,x_{2k^2}),\ldots,  p_{k^2}(x_1,\ldots,x_{2k^2})\}$
of $k^2$ commutative
polynomials in $2k^2$ variables,
the objective is to find a condensed representation for
these polynomials in terms of a
single non-commutative polynomial
$p(X,Y)$ in two $k\times k$ matrix variables $X$ and $Y$.
Algorithms that will generically determine whether
the given family $\cP$ has a
non-commutative representation and that will produce such a
representation if they exist are developed. These
algorithms will determine a non-commutative representation
for families $\cP$ that admit a non-commutative representation in an
 open, dense subset of the
vector space of non-commutative polynomials in two variables.
\end{abstract}

\maketitle

\section{Introduction}
\label{sec:nov25a11}

This paper addresses a new type of  problem concerning
a condensed description of a collection of polynomials.

\subsection{Problem statement}
\label{sec:Theprob}
The problem is to  represent
a family  $\cP$ of $k^{2}$ polynomials $p_{1},....,p_{k^{2}}$
in $2k^{2}$
commuting variables $x_{1},....,x_{2k^{2}}$ as an {\bf nc} (non-commutative)
polynomial $p=p(X,Y)$ in two $k\times k$ matrices  $X$
and  $Y$ whose entries are the variables $x_j$ (without repetition).
For example,
 it is readily checked that if
\begin{eqnarray*}
p_1(x_1,\ldots,x_8)&=&x_1^2+x_2x_3+x_1x_5+x_2x_7\\
p_2(x_1,\ldots,x_8)&=&x_1x_2+x_2x_4+x_1x_6+x_2x_8\\
p_3(x_1,\ldots,x_8)&=&x_1x_3+x_3x_4+x_3x_5+x_4x_7\\
p_4(x_1,\ldots,x_8)&=&x_2x_3+x_4^2+x_3x_6+x_4x_8,
\end{eqnarray*}
then
$$
\begin{pmatrix}p_1&p_2\\ p_3&p_4\end{pmatrix}=X^2+XY\quad\textrm{with}\
X=\begin{pmatrix}x_1&x_2\\ x_3&x_4\end{pmatrix}\ \textrm{and}\
Y=\begin{pmatrix}x_5&x_6\\ x_7&x_8\end{pmatrix}.
$$

The {\bf main objectives} of this paper are to:
\begin{enumerate}
\item[(1)] {\it Present a number of
conditions that a given
set of polynomials
$$
p_1(x_1,\ldots,x_{2k^2}),\ldots,   p_{k^2}(x_1,\ldots,x_{2k^2})
$$
in $2k^2$ commuting variables $x_1,\ldots,x_{2k^2}$ must  satisfy
in order for it to admit an nc representation $p(X,Y)$.}
\vspace{2mm}
\item[(2)] {\it Present a number of procedures
 for recovering such representations, when they exist}.
\end{enumerate}

To formally describe the problem
we shall say that a
family $\cP$ of $k^2$ polynomials
$$
p_1=p_1(x_1,\ldots,x_{2k^2}),\ldots, p_k=p_{k^2}(x_1,\ldots,x_{2k^2}),
$$
in $2k^2$
commutative variables
$x_1,\cdots,x_{2k^2}$, admits a \df{ nc representation}
if there exists a pair of $k\times k$  matrices $X$ and $Y$ and an nc
polynomial $p$
in two nc variables such that

\begin{align}\label{eq1:16jan12}
&X=\left(
                    \begin{array}{ccccc}
                      x_{\sigma(1)} & x_{\sigma(2)} & . & . & x_{\sigma(k)} \\
                      . & . &  &  & . \\
                      . & . & . &  & . \\
                      . & . & . & . & . \\
                      x_{\sigma(k(k-1)+1)} & x_{\sigma(k(k-1)+2)} & . & .
& x_{\sigma(k^{2})} \\
                    \end{array}\right), \\
             & Y=\left(
                    \begin{array}{ccccc}
                      x_{\sigma(k^{2}+1)} & x_{\sigma(k^{2}+2)} & . & .
& x_{\sigma(k^{2}+k)} \\
                      . & . &  &  & . \\
                      . & . & . &  & . \\
                      . & . & . & . & . \\
                      x_{\sigma(k^{2}+k(k-1)+1)}
& x_{\sigma(k^{2}+k(k-1)+2)} & . & . &
                      x_{\sigma(2k^{2})}\end{array} \right)
\end{align}

\noindent
and
\begin{equation}
\label{eq:oct18h11}
                  p(X,Y)=
                   \left(
                    \begin{array}{ccccc}
                      p_{\lambda(1)} & p_{\lambda(2)} & . & .
& p_{\lambda(k)} \\
                      . & . &  &  & . \\
                      . & . & . &  & . \\
                      . & . & . & . & . \\
                      p_{\lambda(k(k-1)+1)}& p_{\lambda(k(k-1)+2)} & .
& . & p_{\lambda(k^{2})} \\
                    \end{array}
                  \right),
\end{equation}
where $\sigma$  is a permutation of the set of integers
$\{1,\ldots,2k^2\}$  and $\lambda$ is a permutation
of the set of integers $\{1,\ldots,k^2\}$.

\subsection{Our Algorithms}

The main contribution of this paper is to introduce a
collection of algorithms for solving the nc polynomial representation problem.
Since they are long, full descriptions are postponed to
the body of the paper. However, we shall try to present their
flavor in this subsection.

The algorithms
are based on the analysis
of the patterns of one letter words, two letter words and some
three letter words in the given family of polynomials. Thus, for example,
if the one letter word
$7x_5^n$ occurs in one of the polynomials, and the family admits an
nc representation $p(X,Y)$, then $x_5$ must be a diagonal entry
of either $X$ or $Y$ and $k-1$ of the other polynomials will contain either
exactly one or two one letter words of the same degree with the
coefficient  $7$. If there are no one letter words, then the analysis is
more delicate; see Sections \ref{sec:nov25e11} and \ref{sec:25f11}.

If the diagonal variables are determined successfully, then subsequent
algorithms serve to \df{partition} the remaining $2k^2-2k$ variables between
$X$ and $Y$ and then to \df{position} them within these matrices.
En route the $k^2$ polynomials are arranged in an appropriate order
in a $k\times k$ array.
The final step is to obtain
the nc polynomial; this is done by matching coefficients as in Example \ref{ex1:sep4:2011}
below.

Most families of polynomials containing $k^2$ polynomials
in $2k^2$ variables will not have nc representations.  Either
there will be no way of partitioning and positioning the variables that is
consistent with the family or there will be no choice of coefficients
that work.

\subsubsection{Examples}

 The next examples serve to illustrate
some of the structure one sees in families of polynomials $\cP$ that
admit an nc representation and how it corresponds to
diagonal determination, positioning and partitioning.

\begin{example}\label{ex1:sep4:2011}
{\it
The family of polynomials
\begin{align}
&p_1= 3x_2x_4+3x_4x_8+x_2x_3+x_4x_7+6x_1x_3+6x_3x_7+x_1x_4+x_3x_8 \nonumber\\
&p_2=3x_2x_6+ 3x_6x_8+x_1x_6+ x_5x_8+6x_1x_5+6x_5x_7+x_2x_5+x_6x_7 \nonumber\\
&p_3=3x_2^2+2x_1x_2+3x_4x_6+x_4x_5+x_3x_6+6x_1^2+6x_3x_5\nonumber \\
&p_4= 3x_4x_6+3x_8^2+2x_7x_8+x_3x_6+6x_3x_5+6x_7^2+x_4x_5 \nonumber
\end{align}
admits an nc representation.}
\end{example}
\bigskip

\noindent
{\bf Discussion:}
If the given family of polynomials admits
an nc representation $p(X,Y)$, then it must admit at least one
representation of the form
\begin{equation}
\label{eq:oct18c11}
p(X,Y)=aX^2+bXY+cYX+dY^2
\end{equation}
for some choice of $a,b,c,d\in\RR$, since $p_1,\ldots,p_4$ are
homogeneous of degree two. Moreover, as
$$
p_3=3x_2^2+6x_1^2+2x_1x_2+\cdots\quad\textrm{and}\quad
p_4=3x_8^2+6x_7^2+2x_7x_8+\cdots
$$
and there are no other one letter words in the family $\cP$, it is not hard
to see (as we shall clarify later in more detail in \S \ref{sec:nov25b11}) that
$x_2$ and $x_8$ are diagonal entries in one of the matrices, say $X$, and
correspondingly $x_1$ and $x_7$ are diagonal entries in the other matrix, $Y$.
Moreover, since $x_2$ and $x_6$ are in $p_3$, whereas $x_7$ and $x_8$
are in $p_7$, it follows that if an nc representation exists,
then, either
\begin{equation}
\label{eq:oct18a11}
X=\begin{pmatrix}x_2&?\\?&x_8\end{pmatrix},\quad
Y=\begin{pmatrix}x_1&?\\?&x_7\end{pmatrix}\quad\textrm{and}\quad
p(X,Y)=\begin{pmatrix}p_3&?\\?&p_4\end{pmatrix},
\end{equation}
or
\begin{equation}
\label{eq:oct18b11}
X=\begin{pmatrix}x_8&?\\?&x_2\end{pmatrix},\quad
Y=\begin{pmatrix}x_7&?\\?&x_1\end{pmatrix}\quad\textrm{and}\quad
p(X,Y)=\begin{pmatrix}p_4&?\\?&p_3\end{pmatrix},
\end{equation}
and, in (\ref{eq:oct18c11})  we must have
$$
a=3,\quad (b+c)=2\quad\textrm{and}\quad d=6.
$$
We shall assume that (\ref{eq:oct18a11}) holds; the other possibility
may be treated similarly.

The next step is to try to {\bf partition} the remaining variables between $X$
and $Y$.  Towards this end
it is useful to note that if
$$
X=\begin{pmatrix}x_2&x_a\\x_b&x_8\end{pmatrix}\quad \textrm{then}\quad
X^2=\begin{pmatrix}x_2^2+x_ax_b& \cdot\\\cdot&\cdot\end{pmatrix}
$$
and hence (since we are assuming that (\ref{eq:oct18a11}) is in force)
that $p_3$ must contain a term of the form $3x_ax_b$. Comparison
with the given polynomial $p_3$ leads to the conclusion that
$x_4$ and $x_6$ belong to $X$. Let us arbitrarily {\bf position} $x_6$ as the
$12$ entry of $X$ and $x_4$ as the $21$ entry of $X$. Then
$$
X=\begin{pmatrix}x_2&x_6\\x_4&x_8\end{pmatrix},\quad \textrm{and}\quad
X^2=\begin{pmatrix}x_2^2+x_6x_4& x_2x_6+x_6x_8\\ \cdot&\cdot\end{pmatrix}.
$$
The $11$ entry of $X^2$ provides no new information, but comparison of the
$12$ entry with the given polynomials leads to the conclusion that if the
given family admits an nc representation, then
$p_2$ must sit in the $12$ position in $p(X,Y)$ and hence
$$
p(X,Y)=\begin{pmatrix}p_3&p_2\\p_1&p_4\end{pmatrix}.
$$
Similarly,
$$
Y=\begin{pmatrix}x_1&x_c\\x_d&x_7\end{pmatrix}\Longrightarrow
Y^2= \begin{pmatrix}\cdot&x_1x_c+x_cx_7\\\cdot&\cdot\end{pmatrix},
$$
which upon comparison with the entries in $p_2$ leads to the conclusion that
$x_c=x_5$. Therefore, $x_d=x_3$. Comparison of
$$
3X^2+bXY+cYX+6Y^2\quad\textrm{with}\quad \begin{bmatrix}p_3&p_2\\ p_1&p_4
\end{bmatrix}
$$
implies further that equality will prevail
if and only $b=c=1$ (i.e., $a=3$, $b=c=1$ and
$d=6$ in (\ref{eq:oct18c11})).

\medskip

The example illustrates the strategy of first determining
which variables occur on either the  diagonal of $X$ or of  $Y$.
In general, if the given family $\cP$ admits an nc representation
$p(X,Y)$ of degree $d$ and if
$$
p(X,Y)=aX^n+bY^n+\cdots \quad\textrm{with $\vert a\vert+\vert b\vert>0$}
$$
(and no other nonzero multiples of $X^n$ and $Y^n$) for some positive integer
$n\ge 2$, then:
\begin{itemize}
\item[\rm(1)] if $b=0$ (resp., $a=0$), there will be
exactly $k$ one letter words of degree $n$ with coefficient $a$ (resp., $b$);
\vspace{2mm}
\item[\rm(2)] if $ab\ne 0$ and $a\ne b$, there will be exactly $k$ one
letter words of degree $n$ with coefficient $a$
and exactly $k$ one letter words of degree $n$ with coefficient $b$;
\vspace{2mm}
\item[\rm(3)] if $a=b$, there will be exactly $2k$ one letter words of
degree $n$ with coefficient $a$
\end{itemize}
Example \ref{ex1:sep4:2011} fits into setting (2).

So far we have focused on how we can use one letter words
occurring in polynomials in $\cP$.
A substantial part of this paper is also devoted to
developing
procedures for finding the diagonal variables
that are based on patterns in  two and (some) three letter words. The latter come into play
if there are no one letter words
to partition the diagonal variables between $X$ and $Y$.

\subsection{Effectiveness of our algorithms}

Our algorithms depend upon the existence of appropriate patterns of one, two
and some three letter words in the $k^2$ polynomials in the given family.
We shall show that these algorithms are effective generically, i.e.,
they are effective on an open dense set of the set of polynomials that
admit nc representations.

\subsubsection{General Results}

Let $\cW$ be the space of nc polynomials in two variables of degree $d$.
We say that a subspace $\mathcal{U}$ of $\cW$ is of degree $d$
if the maximum degree of the basis elements
of $\cU$ is $d$.

\begin{theorem}\label{genTheorem}
Let $p_1,\ldots,p_{k^2}$ be a family $\cP$ of polynomials
in $2k^2$ commuting variables $x_1,\ldots,x_{2k^2}$ of
degree $d>3$ and let $\mathcal{U}$ be a subspace of $\cW$
of degree $d$.
Then there exists an open dense subset $\mathcal{S}$
of $\mathcal{U}$ for which the
algorithms developed in this paper determine an
nc representation $p \in \mathcal{S}$ for $\cP$
if and only if $\cP$ has an nc representation $p \in \mathcal{S}$.
If such a representation exists, then  these  algorithms may be used
to construct it.
\end{theorem}

\begin{proof}
The proof is postponed until Section~\ref{sec4:29nov11}.
\end{proof}

Example \ref{ex:oct19a11}, below,
exhibits a family $\cP$ with an nc  representation $p$
for which the algorithms do not work.
However, perturbing $p$ gives $\tilde \cP$ for which
they do.

\begin{example}\label{ex:oct19a11}
Suppose that we are given the list of polynomials
\begin{align}
&p_1= x_2x_4+x_4x_8+x_2x_3+x_4x_7+x_1x_3+x_3x_7+x_1x_4+x_3x_8 \nonumber\\
&p_2=x_2x_6+ x_6x_8+x_1x_6+ x_5x_8+x_1x_5+x_5x_7+x_2x_5+x_6x_7 \nonumber\\
&p_3=x_2^2+2x_1x_2+x_4x_6+x_4x_5+x_3x_6+x_1^2+x_3x_5\nonumber \\
&p_4= x_4x_6+x_8^2+2x_7x_8+x_3x_6+x_3x_5+x_7^2+x_4x_5 \nonumber
\end{align}
\end{example}
\bigskip

\noindent  {\bf Discussion}
If the given family admits an nc representation $p(X,Y)$,
then it is readily seen from the one letter words in the family $\cP$ that
it must be of the form (\ref{eq:oct18c11}) with $a=d=1$ and that $x_1$,
$x_2$, $x_7$ and $x_8$ are diagonal variables.

However, it is impossible to
decide on the basis of one letter words which of these variables belong to $X$
and which belong to $Y$. The most that we can say so far is that
$x_1$, $x_2$ and $p_3$ must lie in the same diagonal
position, and hence $x_7$, $x_8$ and $p_4$ are in the other diagonal
position. Thus, we may assume that $x_2$ is in the $11$ position of $X$,
$x_1$ is in the $11$ position of $Y$. But it is still not clear how to
allocate $x_7$ and $x_8$.

It is readily seen that $p(X,Y)=X^2+XY+YX+Y^2$ is an
nc representation for the family of polynomials in Example \ref{ex:oct19a11}.
Thus, $p(X,Y)$ is contained in the subspace
$\mathcal{U}$ of nc polynomials defined by
$$
\cU = \{ aX^2+bXY+cYX+dY^2:\, a,b,c,d\in\RR\},
$$
and, although our algorithms are not effective on
the entire subspace $\mathcal{U}$, they do work on the
(open dense)  subset  $\mathcal{S}$ of
$\cU$ consisting of polynomials in $\cU$ for which
 $a\ne d$ and $2a \ne b$.
\qed

\subsubsection{More detailed statements}
\label{subs:mdr}

Our main theorems on algorithm effectiveness  are more detailed than
Theorem \ref{genTheorem}.
 These theorems and
a number of our algorithms
depend in part on the coefficients of the terms in the
\df{commutative collapse}
$\what p$ of an nc polynomial $p(X,Y)$, which is defined as the
commutative polynomial
$$ \what p(x,y) = p(xI, yI).$$
In particular, if
$\varphi(i,j)$ is the sum of the coefficients of the terms in the
nc polynomial
$p(X,Y)$ of degree $i$ in $X$ and degree $j$ in $Y$,
then $\varphi(i,j)$ is
the coefficient of $x^iy^j$ in $\what{p}(x,y)$.

We shall also need
the following more refined quantities:
\medskip

\noindent
$\varphi(i,j;X)$ (resp. $\varphi(i,j;Y)$) denotes the sum of the coefficients of the terms in the
nc polynomial
$p(X,Y)$ of degree $i$ in $X$ and degree $j$ in $Y$ that end in $X$ (resp.
end in  $Y$);
\ss

\noindent
$\varphi(X;i,j)$ (resp. $\varphi(Y;i,j)$) denotes the sum of the coefficients of the terms in the
nc polynomial
$p(X,Y)$ of degree $i$ in $X$ and degree $j$ in $Y$ that begin with $X$ (resp. begin with $Y$).

Thus, for example, if
$$
p(X,Y)=aX^2YXY+bXYXYX+cYXYX^2+dYX^3Y,
$$
then
$$
\varphi(X;3,2)= a+b,\quad        \varphi(Y;3,2)=c+d,\quad
\varphi(3,2;X)= b+c\quad\textrm{and}\quad\varphi(3,2;Y)=a+d.
$$
Clearly
$$
\varphi(i,j)=\varphi(X;i,j)+\varphi(Y;i,j)=\varphi(i,j;X)+\varphi(i,j;Y).
$$

The next theorem provides some insight into our one letter algorithms.
There is an analogous result for two letter words:
Theorem \ref{thm:2varMain}, which will be presented in Section
\ref{sec4:29nov11}

\begin{theorem}
\label{thm:nov25a11}
Let $p_1,\ldots,p_{k^2}$ be a family $\cP$ of polynomials in $2k^2$ commuting
variables $x_1,\ldots,x_{2k^2}$ and let $\mathcal{W}$ denote the set of
nc polynomials $p(X,Y)$ of degree $d>1$ such that there exists
an integer $t \ge 2$ for which
 $$\varphi(t,0) \ne 0, \quad  \varphi(0,t) \ne 0, \quad
  \varphi(t,0) \ne \varphi(0,t)$$
and that additionally satisfies one of the following properties:
\begin{itemize}
\item[(1)] $t\varphi(t,0) \ne \varphi(t-1,1)$
\item[(2)] $t\varphi(0,t) \ne \varphi(1,t-1)$
\item[(3)] $\varphi(t-1,1;Y) \ne 0$ and $\varphi(t,0) \ne \varphi(t-1,1;Y)$
\item[(4)] $\varphi(Y; t-1,1) \ne 0$ and $\varphi(t,0) \ne \varphi(Y; t-1,1)$
\item[(5)] $\varphi(1,t-1;X) \ne 0$ and $\varphi(0,t) \ne \varphi(1,t-1;X)$
\item[(6)] $\varphi(X; 1,t-1) \ne 0$ and $\varphi(0,t) \ne \varphi(X; 1,t-1).$
\end{itemize}

\noindent
Then the one letter algorithms stated in Section~\ref{finalresultone} determine that $\cP$
admits an
nc representation $p(X,Y)$ in the class $\cW$ if and only if
$\cP$ has a representation in this class.
If such a representation exists, then these algorithms can  be used to
construct it.
\end{theorem}

\begin{proof}
See Theorems \ref{thm:nov6a11}, \ref{thm:nov6b11} and Remark
\ref{rem:nov10a11}.
\end{proof}

The long list of caveats looks confining,
but they are all strict inequality constraints
and so {\it hold generically}.

\subsection{Uniqueness }

The issue of uniqueness of an nc representation
is of interest in its own right.
We shall see is that while the
representation $p(X,Y)$ is  highly non-unique,
the  arrangement of commutative variables
$x_j$ in the matrices $X$ and $Y$ is determined
up to permutations, transpositions and interchanges of $X$
and $Y$.

\subsubsection{Polynomial identities and non-uniqueness of $p$}
\label{sec1:22jan12}
A basic theorem in the
theory of rings with polynomial identities implies that if $\Sigma_{2k}$
denotes the set of all permutations of the
set $\{1,\ldots,2k\}$ for each positive
integer $k$, then the polynomial
\begin{equation}
\label{eq:oct27z11}
q(X_1,\ldots,X_{2k})=
\sum_{\sigma\in\Sigma_{2k}}sgn( \sigma)
X_{\sigma(1)}\cdots X_{\sigma(2k)}=0
\end{equation}
for every choice of the $k\times k$ matrices $X_1,\ldots,X_{2k}$ in
$\RR^{k\times k}$. Thus, if $X$ and $Y$
are arbitrary real $k\times k$ matrices and if
\begin{equation}
\label{eq:oct27x11}
p(X,Y)\stackrel{def}{=}q(X_1,\ldots,X_{2k})
\end{equation}
with
\begin{equation}
\label{eq:oct27y11}
X_j=[X^j,Y]\quad\textrm{for $j=1,\ldots,2k$},
\end{equation}
then $p(X,Y)=0$; see \cite{l50} and \cite{al50} for additional information.

Any other replacement of $X_j$ in (\ref{eq:oct27z11}) by a polynomial in
$X$ and $Y$ will also yield a polynomial $p(X,Y)=0$. However, the
choice in (\ref{eq:oct27y11}) will have nonzero coefficients.
In particular this means that if a given family $\cP$ of polynomials admits
an nc representation, then it admits infinitely many
nc representations.

If $k=2$, for example, the nc polynomials
\beq
\label{eq:introex2}
YXY^{2}X+Y^{2}X^{2}Y+YX^{2}YX+XY^{2}X^{2}
+XYXY^{2}+X^{2}YXY
\end{equation}
and
\beq
\label{eq:oct26a12}
Y^{2}XYX+YX^{2}Y^{2}+XY^{2}XY+YXYX^{2}+XYX^{2}Y+X^{2}Y^{2}X
\end{equation}
generate the same family regardless of how
the commutative variables $x_1,\cdots, x_{8}$ are partitioned between
$X$ and $Y$ and positioned. 

This stems from the fact that the difference between the nc polynomial in (\ref{eq:introex2}) and the nc polynomial in (\ref{eq:oct26a12}) is equal to
the commutator
$$
[Y-X, \; (XY-YX)^2]=(Y-X)(XY-YX)^2-(XY-YX)^2(Y-X)=0,
$$
since for $2\times 2$ matrices 
$$
X =\left(\begin{array}{cc}
x_{1}& x_{2} \\
x_{3} & x_{4} \\
\end{array}
\right)\quad\textrm{and}\quad
Y=\left(\begin{array}{cc}
x_{5}& x_{6} \\
x_{7} & x_{8} \\ \end{array}\right),
$$
the polynomial $(XY-YX)^2$ has the special form 
\begin{equation}
\label{eq:oct26b12}
p(X,Y)=[X,Y]^2=(XY-YX)^2=
\begin{pmatrix}p(x)&0\\0&p(x)\end{pmatrix}.
\end{equation}
This is well known by the experts in matrix identities, and it is easily verified by direct calculation that 
\begin{multline*}
p(x)=x_{2}^{2}x_{7}^{2}-x_{2}\{x_{3}[2x_{6}x_{7}+(x_{5}-x_{8})^{2}]-
(x_{1}-x_{4})x_{7}(x_{5}-x_{8})\}\\
+x_{6}\{x_{1}x_{3}x_{5}-x_{3}x_{4}x_{5}+x_{3}^{2}x_{6}
-x_{1}^{2}x_{7}+2x_{1}x_{4}x_{7}-x_{4}^{2}x_{7}+x_{3}(-x_{1}+x_{4})x_{8}\},
\end{multline*}
Thus, the polynomial $p(X,Y)$ in (\ref{eq:oct26b12}) is an example of a homogeneous nc polynomial that
 produces a family $\cP$   
with some of the polynomials in
$\cP$  equal to zero
and some not. We shall see later that the fact that the degrees of the polynomials in $\cP$ are either $4$ or $0$ is  consistent with Lemma \ref{lem:monHomog}.

If $p(X,Y)=(X+Y)^n$ for some positive integer $n$, then it is
impossible to determine which variables belong to $X$ and which variables
belong to $Y$.

The theorems presented later in the paper that validate
our algorithms, e.g., Theorem \ref{thm:nov25a11},
 have hypotheses
that exclude cases like \eqref{eq:introex2}.

\subsubsection{Uniqueness of $X,Y$}

We just saw
 that an nc representation for $\cP$ is highly non-unique,
however, the pair $X,Y$ in such representations
 is generically very tightly determined.  This is
indicated by the following theorem.

\begin{theorem}
\label{cor:uniqIntro}
If $\cP$ is a family of polynomials $p_1, \cdots, p_{k^2}$ in the
commutative variables $x_1,\ldots, x_{2k^2}$  that admits two
nc representations $p(X,Y)$ and ${\wtilde p}(\wtilde{X},\wtilde{Y})$ that
satisfy the conditions of
Theorem~\ref{thm:nov25a11},
then there exists a permutation matrix $\Pi$ such that one of the following
must hold:\\
\begin{eqnarray}\label{eq:feb23a12}
\begin{tabular}{ll}
${\bf (1)}\quad X =\Pi^T\wtilde{X}\Pi, $&$Y = \Pi^T\wtilde{Y}\Pi,$\\
${\bf (2)}\quad X= \Pi^T\wtilde{X}^T\Pi,$&$Y=\Pi^T\wtilde{Y}^T\Pi, $ \\
${\bf (3)}\quad X= \Pi^T\wtilde{Y}\Pi,$ &$Y=\Pi^T\wtilde{X}\Pi, $\\
${\bf (4)}\quad X= \Pi^T\wtilde{Y}^T\Pi, $&$Y= \Pi^T\wtilde{X}^T\Pi. $\\
\end{tabular}
\end{eqnarray}
\end{theorem}

\def\tX{{\wtilde{X}}}
\def\tY{{\wtilde{Y}}}

\begin{proof}
The proof is postponed until Section~\ref{sec:variations}.
\end{proof}

We shall say that the pairs $X,Y$ and $\tX,\tY$ are
{\bf permutation equivalent} if they are  related by any of
the four choices in
\eqref{eq:feb23a12}.

\subsection{Motivation}

The problem we study in the paper is undertaken primarily
for its own sake, however,
the original  motivation arose from the observation that
the running time for algebraic calculations on a large family of commutative
polynomials $\cP$ can be much longer than the corresponding calculation
on a small
family of nc polynomials representing $\cP$.
Such calculations can be done
using nc  computer algebra, for example NCAlgebra
or NCGB \cite{NCAlgebra}, which runs under Mathematica.

As an example, consider  computing Gr\"obner Bases,
a powerful but time consuming algebraic construction.
The reader does not need to know anything about them to
get the thrust of this example.
We have a list $P$ of nc polynomials
and run an nc Gr\'obner Basis algorithm on
\begin{align}
P= \{a^T  m + m^T a + m^T m,~
 a  w + w^T  w + w^T  a^T,~ m^T a  m,& \\
 m^T  a  w,~ m^T  a^T  m,~ w^T  a  w,~
 w^T  a^T  m,~ \ w^T  a^T w& \}.\nonumber
\end{align}
Using NCGB on a Macbook Pro it finished in  .007 seconds.
Now substitute two by two matrices
$$ a \to
\left(
\begin{array}{ccc}
 a_{11} &   a_{12}   \\
 a_{21} &   a_{22}
\end{array}
\right),
\qquad
w \to
\left(
\begin{array}{ccc}
 w_{11} &   w_{12}   \\
 w_{21} &   w_{22}
\end{array}
\right),
\qquad
m \to
\left(
\begin{array}{ccc}
 m_{11} &   m_{12}   \\
 m_{21} &   m_{22}
\end{array}
\right)
$$
with commuting entries
for the variables and run the  ordinary
Mathematica Gr\"obner Basis Command.
 The run took 161 seconds in  the most favorable
  monomial order that we tried.
 The  corresponding $3 \times 3$ matrix substitution
 yielded a Gr\"obner Basis computation which
 did not finish in  1 hour.
Calculation with
 higher order matrix substitutions  would be prohibitive.

 The
 NC Gr\"obner Basis and the commutative one contain different information.
 Namely, the NCGB determines membership in the two sided ideal
 $\cI_P$  generated by $P$ while
 the  commutative GB obtained from  ``generic'' $n\times n$
 matrix substitution determines membership in the ideal
 generated by $\cI_P + R_n$ where $R_n$ is the ideal
 of all nc polynomials which vanish on the $n \times n$ matrices.
 We would assert that the NCGB contains very valuable
 information (possibly more than in the $\cI_P + R_n$ case)
 and is readily obtained.
 In fact what brought us to the nc polynomial representation
question
 was the reverse side of this.
 To speed up nc GB runs we tried symbolic matrix substitutions
 in the hope that the commutative GBs would go quickly and as
 $n$ got bigger guide us toward an NCGB.
 This approach seems hopeless because of prohibitively long run times.

In special circumstances nc representations could
 have a significant advantage for numerical computation.
 In particular,
 the numerics for solving the second order polynomial (in matrices) equation,
called a Riccati equation, is highly developed.
Consequently, it would be very useful to be able to  replace a collection of conventional
polynomial equations by an nc representation.

Finally we mention that there is a burgeoning area devoted to extending
real  (and some complex) algebraic  geometry to free algebras.
Here one analyzes non-commutative polynomials
with properties determined by substituting in  square matrices of arbitrary size.
See the recent references \cite{BB08, BK10, dOHMVP, KV09, KS10}
and their extensive bibliographies.

\subsection{Computational Cost}
\label{sec:introcost}

The problem considered here can be attacked by ``brute force"
rather than by the methods developed in this paper. 
There are $(2k^2)!$ arrangements  
of the variables $x_1,\cdots,x_{2k^2}$ in $X$ and $Y$ 
and $(k^2)!$  arrangements 
of the polynomials $p_1,\cdots,p_{k^2}$ in $\cP$.
For a given arrangement $\sigma$ of the variables in $X$ and $Y$,
one obtains a matrix of commutative polynomials by
forming a general nc polynomial 
$p(X,Y)={\sum_{\alpha,\beta} c_{\alpha\beta}m_{\alpha\beta}}(X,Y)$ of degree $d$ as 
in (\ref{eq:sep19c10}) with undetermined
coefficients $c_{\alpha\beta}$ that are chosen to match the array determined by the arrangement 
$\lambda$ of $k^2$ polynomials, if possible. 
For each pair of arrangements $\sigma$ and $\lambda$, one attempts to 
solve for the coefficients $c_{\alpha\beta}$ to obtain an nc representation.
We will refer to this approach as the \df{Brute Force Method}.
Because there are $(k^2)!(2k^2)!$ possible systems,
the cost of this approach is very high.
Also, to rule out the existence of an  nc representation  this way it is necessary to  
{\it check all of these cases and to verify that they fail.} 

Much to the contrary, the procedures we introduce are
likely to determine non-existence  
in the first few step. Even when me must run through all the steps, 
we find that the implementation of the algorithm that we call Algorithm 2 
requires on the order of 
$$ 
 10 \left( k^7 + 3d k^5 +d^3 k^3 \right) + \sum_{i=2}^d \frac{2^{3i+1}}{3}
 $$
operations, which is much less than the 
$$
(2k^2)!(k^2)!\left(\sum_{i=2}^d \frac{2^{3i+1}}{3}\right) 
$$
operations required by brute force; see  
\S \ref{sec:cost} for details.

\section{One and two letter monomials in the $k^2$ polynomials:
Determining the diagonal variables}
\label{sec:nov25b11}

In this section we enumerate the one and two letter monomials that appear
in the $k\times k$ array of commutative polynomials corresponding to the nc
polynomial $p(X,Y)$.

\subsection{Preliminary calculations}
\label{subsec:precalc}This subsection is devoted to notation and a couple of
definitions that will be useful in the main developments.

Let $e_1,\ldots,e_k$ denote the standard basis for $\mathbb{R}^k$
and let $E_{st}$ denote the $k\times k$ matrix with a $1$ in the $st$
position and $0$'s elsewhere. Then, since
$$
E_{st}=e_se_t^T,
$$
it is readily seen that
$$
    E_{st}E_{uv}=e_s(e_t^Te_u)e_v^T=\left\{\begin{array}{l} 0\quad
\textrm{if}\
t\ne u\\
E_{sv}\quad\textrm{if}\ t=u\end{array}\right.
$$
and hence that
$$
(E_{st})^2=\left\{\begin{array}{l} 0\quad\textrm{if}\
s\ne t\\
E_{st}\quad\textrm{if}\ s=t\end{array}\right.
$$

Let $\alpha=(\alpha_1,\ldots,\alpha_\ell)$, $\beta=(\beta_1,\ldots,\beta_\ell)$ be
multi-indices with
\begin{equation}
\label{eq:aug27a10}
\textrm{positive integer entries, except for $\beta_\ell$,  which may also be zero}
\end{equation}
and suppose further that
\begin{equation}
\label{eq:aug25a10}
\alpha_1+\cdots+\alpha_\ell=s,\quad \beta_1+\cdots+\beta_\ell=t,
\end{equation}
and let
\begin{equation}
\label{eq:sep19c10}
m_{\alpha,\beta}(X,Y)=X^{\alpha_1}Y^{\beta_1}\cdots X^{\alpha_\ell}
Y^{\beta_\ell}.
\end{equation}
Then, since $s\ge \ell$ and $t\ge \ell-1+\beta_\ell$, it follows that
$$
\ell \le s\quad\textrm{and}\quad \ell \le t+1-\beta_\ell.
$$
The proof of the next two lemmas will rest heavily on the following observations:\\
\medskip

if $m$, $r$ and $n$ are nonnegative integers such that
$m+r\ge2$ and $n\ge2$, then
\begin{align}
\label{eq:aug28a10}
(E_{cd})^m(E_{ab})^n(E_{cd})^r
&=\left\{\begin{array}{ll}E_{aa}&\ \textrm{if $ a=b=c=d$}\\ \vspace{6pt}
0&\ \textrm{otherwise}\end{array}\right. .\\
\label{eq:aug28b10}
E_{aa}E_{cd}E_{aa}&=\left\{\begin{array}{ll}E_{aa}&\quad\textrm{if $c=d=a$}\\
\vspace{6pt}
0&\quad\textrm{otherwise}\end{array}\right. .\\
\label{eq:aug28c10}
E_{ab}E_{cc}E_{ab}&=\left\{\begin{array}{ll}E_{aa}&\quad\textrm{if $c=d=a$}\\
\vspace{6pt}
0&\quad\textrm{otherwise}\end{array}\right.
\end{align}

\begin{remark}
It is also useful to note that if the constraints (\ref{eq:aug27a10})
and (\ref{eq:aug25a10}) are in force and
\begin{equation}
\label{eq:sep16a10}
X=x_iE_{ab}+\cdots \quad\textrm{and}\quad Y=x_jE_{cd}+\cdots,
\end{equation}
then
\begin{equation}
\label{eq:sep16b10}
m_{\alpha,\beta}(X,Y)=x_i^sx_j^tm_{\alpha,\beta}(E_{ab},E_{cd})+\cdots.
\end{equation}
\end{remark}

\begin{lemma}
\label{lem:aug28a10}
Assume that the multi-indices $\alpha=(\alpha_1,\ldots,\alpha_\ell)$ and $\beta=(\beta_1,\ldots,\beta_\ell)$ are subject to the constraints  (\ref{eq:aug27a10}) and (\ref{eq:aug25a10}). Suppose further that $s\ge 2$, $t\ge 2$, and  that
\begin{equation}
\label{eq:aug28d10}
\max\{\alpha_1,\ldots,\alpha_\ell;\,\beta_1,\ldots,\beta_\ell\}\ge 2.
\end{equation}
Then
\begin{equation}
\label{eq:aug28e10}
m_{\alpha,\beta}(x_iE_{ab},x_jE_{cd})=\left\{\begin{array}{lr} x_i^sx_j^tE_{aa}&\quad if a=b=c=d\\
0&\quad otherwise
\end{array}\right. .
\end{equation}
In other words, $m_{\alpha,\beta}(x_iE_{ab},x_jE_{cd})\ne 0$ if and only if $x_i$ and $x_j$ are diagonal pairs in the same position.
\end{lemma}

\begin{proof} The proof is divided into cases.
\bigskip

\noindent
{\bf 1.} If $\ell=1$, then 
$$
m_{\alpha,\beta}(x_iE_{ab},x_jE_{cd})=x_i^sx_j^t(E_{ab})^s(E_{cd})^t
$$
and the asserted conclusion (\ref{eq:aug28e10}) follows from (\ref{eq:aug28a10}) with $m=s$ and $n=t$, since $s\ge 2$ and $t\ge 2$, by assumption.
\bigskip

\noindent
{\bf 2.}  If $\ell>1$ and $\alpha_r\ge 2$ for some $r\in\{1,\ldots,k\}$, then
$$
(x_iE_{ab}+\cdots)^r=x_i^r(E_{ab})^r+\cdots =\left\{\begin{array}{ll}
x_i^rE_{aa}+\cdots &\quad\textrm{if $b=a$}\\
\vspace{6pt}
0&\quad\textrm{otherwise}\end{array}\right. .
$$
But if $b=a$, then
$$
m_{\alpha,\beta}(x_iE_{ab},x_jE_{cd})=x_i^sx_j^t(E_{aa})^{\alpha_1}(E_{cd})^{\beta_1}E_{aa})^{\alpha_2}\cdots
$$
and (\ref{eq:aug28e10})  follows from (\ref{eq:aug28b10}).
\bigskip

\noindent
{\bf 3.}  If $\ell>1$ and $\beta_r\ge 2$ for some $r\in\{1,\ldots,k\}$, then
$$
(x_jE_{cd})^r=x_j^r(E_{cd})^r=\left\{\begin{array}{ll}
x_j^rE_{cc}&\quad\textrm{if}~d=c\\
\vspace{6pt}
0&\quad\textrm{otherwise}\end{array}\right. .
$$
But if $d=c$, then
$$
m_{\alpha,\beta}(x_iE_{ab},x_jE_{cd})=X^{\alpha_1}Y^{\beta_1}X^{\alpha_2}\cdots=x_i^sx_j^t(E_{ab})^{\alpha_1}E_{cc}(E_{ab})^{\alpha_2}\cdots
$$
and (\ref{eq:aug28e10}) follows from (\ref{eq:aug28c10}).
\end{proof}

\begin{remark}
\label{rem:aug29a10}
Condition (\ref{eq:aug28d10}) is automatically met if either
$$
s>\ell,\quad\textrm{or}\quad t>\ell,\quad\textrm{or}\quad s=t=\ell\ \textrm{and}\ \beta_\ell=0.
$$
\end{remark}

It remains to consider the case
\begin{equation}
\label{eq:aug29a10}
\max\{\alpha_1,\ldots,\alpha_\ell;\,\beta_1,\ldots,\beta_\ell\}\le 1.
\end{equation}

\begin{lemma}
\label{lem:aug25b10}
If (\ref{eq:aug27a10}), (\ref{eq:aug25a10}) and (\ref{eq:aug29a10}) are in
force and $t\ge 2$, then there are four possibilities:
\bigskip

\noindent
{\bf 1.}  $\beta_\ell=1$: In this setting
$s=t$, $\ell=s$, $m_{\alpha,\beta}(X,Y)=(XY)^t$ and
$$
m_{\alpha,\beta}(x_iE_{ab},x_jE_{cd})=x_i^sx_j^t(E_{ab}E_{cd})^t=
\left\{\begin{array}{lr}x_i^sx_j^tE_{aa}&\quad\textrm{if}\ c=b\ \textrm{and}\ d=a\\
0&\quad\textrm{otherwise}\end{array}\right. .
$$

\noindent
{\bf 2.}  $\beta_\ell=0$: In this setting
$s=t+1$, $\ell=s$, $m_{\alpha,\beta}(X,Y)=(XY)^tX$ and
$$
m_{\alpha,\beta}(x_iE_{ab},x_jE_{cd})=x_i^sx_j^t(E_{ab}E_{cd})^tE_{ab}=
\left\{\begin{array}{lr}x_i^sx_j^tE_{ab}&\quad\textrm{if}\ c=b\ \textrm{and}\ d=a\\
0&\quad\textrm{otherwise}\end{array}\right. .
$$

\noindent
{\bf 3.}  $\beta_\ell=1$: In this setting
$s=t$, $\ell=s$,  $m_{\alpha,\beta}(Y,X)=(YX)^t$ and
$$
m_{\alpha,\beta}(x_jE_{cd}, x_iE_{ab})=x_i^sx_j^t(E_{cd}E_{ab})^t=
\left\{\begin{array}{lr}x_i^sx_j^tE_{ba}&\quad\textrm{if}\ c=b\ \textrm{and}\ d=a\\
0&\quad\textrm{otherwise}\end{array}\right. .
$$

\noindent
{\bf 4.}  $\beta_\ell=0$: In this setting
$s=t+1$, $\ell=s$, $m_{\alpha,\beta}(Y,X)=(YX)^t Y$ and
$$
m_{\alpha,\beta}(x_jE_{cd},x_iE_{ab})=x_i^tx_j^s(E_{cd}E_{ab})^t=
\left\{\begin{array}{lr}x_j^sx_i^tE_{bb}&\quad\textrm{if}\ c=b\ \textrm{and}\ d=a\\
0&\quad\textrm{otherwise}\end{array}\right. .
$$
\end{lemma}

\begin{proof}

In view of (\ref{eq:aug27a10}), the constraint (\ref{eq:aug29a10}) implies that
$$
\alpha_1=\cdots=\alpha_\ell=1,\quad \beta_1=\cdots=\beta_{\ell-1}=1\quad
\textrm{and $\beta_\ell=1$ or $\beta_\ell=0$}.
$$
Correspondingly
$$
m_{\alpha,\beta}(X,Y)=\left\{\begin{array}{ll}(XY)^\ell&\quad\textrm{if}\ \beta_\ell=1\\
\vspace{2mm}
(XY)^{\ell-1}X&\quad\textrm{if}\ \beta_\ell=0\end{array}\right.
$$
and
$$
m_{\alpha,\beta}(Y,X)=\left\{\begin{array}{ll}(YX)^\ell&
\quad\textrm{if}\ \beta_\ell=1\\
\vspace{2mm}
(YX)^{\ell-1}Y&\quad\textrm{if}\ \beta_\ell=0\end{array}\right. .
$$
The remaining conclusions are self-evident.
\end{proof}

\begin{definition}
The two $r$ letter monomials $ex_{i_1}^{\alpha_1}\cdots x_{i_r}^{\alpha_r}$ and $fx_{j_1}^{\beta_1}\cdots x_{j_r}^{\beta_r}$ with $e\ne 0$ and
$f\ne 0$ are said to be
{\bf $\vtr$-equivalent} if there exists a permutation $\sigma$ of the
integers $\{1,\ldots,r\}$
such that $\beta_j=\alpha_{\sigma(j)}$ for $j=1,\ldots r$.  This will be
indicated by
writing
$$
ex_{i_1}^{\alpha_1}\cdots x_{i_r}^{\alpha_r}\vtr fx_{j_1}^{\beta_1}
\cdots x_{j_r}^{\beta_r}.
$$
These two monomials are
{\bf structurally equivalent}  (SE) if  they are $\vtr$-equivalent and $e=f$.
This will be indicated by writing
$$
ex_{i_1}^{\alpha_1}\cdots x_{i_r}^{\alpha_r} \quad SE  \quad
fx_{j_1}^{\beta_1}\cdots x_{j_r}^{\beta_r}.
$$
\end{definition}

Thus, for example, if $a,b,c,d\in \RR\setminus\{0\}$, then the four two
letter words
$$
ax_1^2x_3^4, \quad bx_3^2x_1^4, \quad cx_3^2x_4^4\quad\textrm{and}\quad
dx_5^2x_6^4
$$
are $\vtr$-equivalent; they will be SE if and only if $a=b=c=d$.

\subsection{Enumerating one letter monomials in the $k\times k$ array}

\begin{lemma}
\label{lem:aug16b9}
If a family of polynomials $p_1,\ldots,p_{k^2}$ in the $2k^2$ commuting
variables $x_1,\ldots,x_{2k^2}$ admits an nc representation
$p(X,Y)$, then for each positive integer $n>1$ exactly one of the following
situations prevails:
\begin{enumerate}
\item[\rm(1)] There are no one letter monomials of degree $n$ in any one of the given
polynomials.
\vspace{2mm}
\item[\rm(2)] At least one of the given polynomials
contains exactly one one letter monomial of degree $n$.
\vspace{2mm}
\item[\rm(3)] At least one of the given polynomials contains exactly two
one letter monomials of degree $n$.
\end{enumerate}
Moreover,
\begin{enumerate}
\item[\rm(2)]  holds $\Longleftrightarrow$ there exist exactly $k$ polynomials
each one of which
contains exactly one one letter monomial $ex_{i_s}^n$
of degree $n$  (all with the same coefficient).
\vspace{2mm}
\item[\rm(3)] holds $\Longleftrightarrow$ there exist exactly $k$ polynomials
each one
of which contains exactly two one letter monomials  $ex_{i_s}^n+fx_{j_t}^n$.
\end{enumerate}
Further, if $e\ne f$ and $ef\ne 0$ then the letters $x_{i_m}$, $m=1,\ldots,k$
in the monomials $ex_{i_1}^n,\ldots,ex_{i_k}^n$ are the diagonal entries of one
of the matrices and the letters $x_{j_m}$, $m=1,\ldots,k$
in the monomials $fx_{j_1}^n,\ldots,fx_{j_k}^n$ are the diagonal entries of the
other.

\end{lemma}

\begin{proof}
Clearly (1), (2) and (3) are mutually exclusive possibilities that correspond
to
$$
p(X,Y)=aX^n+bY^n+\cdots
$$
with either (1) $a=0$ and $b=0$ for all $n$, (2) $ab=0$ and $a\ne b$

Suppose first that at least one of the given $k^2$ polynomials contains
exactly one term of the form $ax_i^n$ with $n>1$, a real coefficient
$a\ne 0$ and $x_i\in X$. Then
$$
X=x_i E_{ss}+\cdots\quad\textrm{for some $s\in\{1,\ldots,k\}$}
$$
and
$$
aX^n=a(x_i^nE_{ss}+\cdots).
$$
Moreover, since the polynomial that contains the term $ax_i^n$ contains
only one term of this form, it follows that
$$
p(X,Y)-aX^n\quad \textrm{does not contain a term of the form $cY^n$ with
$c\ne 0$}.
$$
Thus, as $X$ has $k$ diagonal entries, there will be exactly $k$
polynomials each one of which contains a exactly one term of this form.
This completes the proof of (a). The proof of (b) is similar to the proof
of (a).

Finally, a term of the form $ax_i^n+bx_j^n$ with $i\ne j$ and $ab\ne 0$ will
be present in one of the polynomials if and only if either
$$
X=x_iE_{ss}+\cdots\quad\textrm{and}\quad Y=x_jE_{ss}+\cdots
$$
for some choice of $s\in\{1,\ldots,k\}$, or
$$
X=x_jE_{ss}+\cdots\quad\textrm{and}\quad Y=x_iE_{ss}+\cdots
$$
for some choice of $s\in\{1,\ldots,k\}$. The rest of the proof
goes through much as before.
\end{proof}

\begin{lemma}\label{lem2:sep9:2010}
Let $p_1,\cdots,p_{k^2}$ be a family of polynomials with an nc representation
$p(X,Y)$. Suppose that some polynomial $p$ in the family contains the
terms $ax_i^n + bx_j^n$ with
$a \ne b$ and either $a \ne 0$ or $b \ne 0$.  Then there exist $k$ polynomials $p_{i_1},\cdots,p_{i_k}$
containing the respective terms $ax_{i_1}^n + bx_{j_1}^n,\cdots,ax_{i_k}^n + bx_{j_k}^n$.  Moreover,
the terms $ax_{i_m}^n$ with $(1 \le m \le k)$ that have coefficient $a$ are the diagonal variables of
one matrix and the terms with $bx_{j_m}^n$ with $(1 \le m \le k)$ that have coefficient $b$ are the diagonal
terms of the other matrix.
\end{lemma}
\begin{proof}
This follows from Lemma~\ref{lem:aug16b9}.
\end{proof}

\subsection{Enumerating one and two letter monomials in the $k\times k$ array}

 The symbol
$$
\chi(a)=\left\{\begin{array}{ll} 1&\quad\textrm{if}\ a\ne 0\\
\vspace{6pt}
0&\quad\textrm{if}\ a=0
\end{array}\right.
$$
will be used in the next  lemma.

\begin{lemma}
\label{lem:aug24a10}
Let $p_1,\ldots,p_{k^2}$ be a family of polynomials in the $2k^2$
commuting variables $x_1,\ldots,x_{2k^2}$
that admits an nc representation $p(X,Y)$ of degree $d>1$.
Suppose further that $s=t$, $t \ge 2$,
\begin{equation}
\label{eq:aug27b10}
p(X,Y)=e_1(XY)^t+e_2(YX)^t+e_3X^{2t}+e_4Y^{2t}+q(X,Y)
\end{equation}
where $q$ is a polynomial that does not contain any scalar multiples of the
first four monomials listed in (\ref{eq:aug27b10}),
that the coefficients $e_1,\ldots, e_4$ are all distinct and that
$x_u\ne x_v$. Then the family of $k^2$ polynomials will contain
$$
\begin{array}{lll}
k^2-k&\quad \textrm{2 letter monomials SE to $e_1 x_u^tx_v^t$}&\quad
\textrm{if $e_1\ne 0$} \\
\vspace{6pt}
k^2-k&\quad \textrm{2 letter monomials SE to  $ e_2 x_u^tx_v^t$}
&\quad
\textrm{if $e_2\ne 0$} \\
\vspace{6pt}
k^2-k&\quad \textrm{2 letter monomials SE to  $ e_3 x_u^tx_v^t$}
&\quad
\textrm{if $e_3\ne 0$} \\
\vspace{6pt}
k^2-k&\quad \textrm{2 letter monomials SE to   $e_4 x_u^tx_v^t$}
&\quad
\textrm{if $e_4\ne 0$} \\
\vspace{6pt}
k&\quad \textrm{2 letter monomials SE to  $\varphi(t,t) x_u^tx_v^t$}
&\quad
\textrm{if $\varphi(t,t)\ne 0$} \\
\vspace{6pt}
k&\quad \textrm{1 letter monomials SE to $ e_3 x_u^{2t}$}
&\quad
\textrm{if $e_3\ne 0$} \\
\vspace{6pt}
k&\quad \textrm{1 letter monomials SE to $e_4 x_u^{2t}$}
&\quad\textrm{if $e_4\ne 0$}
\end{array}
$$
This list incorporates all the ways that one letter monomials of degree
$2t$ and two letter monomials that are $\vtr$-equivalent to $x_u^tx_v^t$
can appear in the given family of polynomials.
Moreover, there is no cancellation:

Each of the $k$ polynomials that sit on the diagonal in the $k\times k$
array corresponding to $p(X,Y)$ contains
$$
(\chi(e_1)+\chi(e_2)+\chi(e_3)+\chi(e_4))(k-1)\quad
\textrm{monomials $\vtr$ to $x_u^tx_v^t$}
$$
made up of off-diagonal letters
and
\begin{equation}
\label{eq:kondiag}
\chi(\varphi(t,t))\quad\textrm{monomials $\vtr$ to $x_u^tx_v^t$}
\end{equation}
made up of diagonal letters, as well as exactly one
 one letter monomial of degree $2t$ with
coefficient $e_3$ and exactly one
 one letter monomial of degree $2t$ with coefficient $e_4$,
both of which are diagonal entries.
\\

No off diagonal polynomial contains any two letter monomials $\vtr$ to
$x_u^tx_v^t$.
\end{lemma}

\begin{proof} Two letter words $x_i^tx_j^t$ with $i\ne j$ and $t\ge 2$ may be generated in four
different ways:
\begin{enumerate}
\item[\rm(1)]  as entries in either $m_{\alpha,\beta}(X,Y)$ or
$m_{\alpha,\beta}(Y,X)$ with $\alpha$ and $\beta$ subject to
(\ref{eq:aug27a10}) and (\ref{eq:aug25a10}) with $s=t$  if $x_i$ and
$x_j$ are in different matrices;
\vspace{6pt}
\item[\rm(2)] as entries in $X^{2t}$ (resp., $Y^{2t}$) if $x_i$ and
$x_j$ are both in $X$ (resp., $Y$).
\end{enumerate}
Suppose first that $x_i$ and $x_j$ are in different matrices and that
(\ref{eq:aug28d10}) is in force. Then Lemma \ref{lem:aug28a10} implies that
the scalar multiples of the words $x_i^tx_j^t$ will appear in at least one of
the $k^2$ polynomials if and only if $x_i$ and $x_j$ are diagonal entries in
the same position. In this instance, $x_i^tx_j^t$ will appear in a
polynomial that sits
in the same diagonal position as $x_i$ and $x_j$ with coefficient
$\varphi(t,t)$.

Suppose next that $x_i$ and $x_j$ are in different matrices and that
(\ref{eq:aug29a10}) is in force. Then, in view of Lemma \ref{lem:aug25b10},
it remains only to consider the contributions from $(XY)^t$ and $(YX)^t$:
If
$$
X=x_iE_{ab}+\cdots \quad\textrm{and}\quad Y=x_jE_{ba}+\cdots,
$$
then
\begin{equation}
\label{eq:aug31a10}
(XY)^t=x_i^tx_j^tE_{aa}+\cdots \quad\textrm{and}\quad
(YX)^t=x_i^tx_j^tE_{bb}+\cdots.
\end{equation}
Since there are $k^2-k$ off-diagonal positions in a $k\times k$ matrix, there are
$k^2-k$ choices of $E_{ab}$ with $a\ne b$. Moreover, since the entry
$x_i^tx_j^t$ appears in the $aa$ position in $(XY)^t$ and the $bb$ position in
$(YX)^t$ there will be no cancellation, even if $e_2=-e_1$.

On the other hand contributions that come from diagonal entries of $X$ and $Y$ can interact with each other, i.e., if
$$
X=x_iE_{aa}+\cdots\quad\textrm{and}\quad Y=x_jE_{aa}+\cdots,
$$
then
$$
e_1(XY)^t+e_2(YX)^t=(e_1+e_2)x_i^tx_j^tE_{aa}+\cdots
$$
and
$$
e_1(XY)^t+e_2(YX)^t+q(X,Y)=\varphi(t,t)x_i^tx_j^tE_{aa}+\cdots .
$$
Thus, if $\varphi(t,t)\ne 0$, there will be $k$ contributions, one for each choice of $a\in\{1,\ldots,k\}$.

The contributions from $e_3X^{2t}$ and $e_4Y^{2t}$ are enumerated in much the same way. Moreover, there is no cancellation, because the monomials with coefficient $e_3$ have all their letters in $X$ and the monomials with coefficient $e_4$ have all their letters in $Y$.
\end{proof}

\begin{remark}
The list in Lemma \ref{lem:aug24a10} is written under the assumption that
$e_1$, $e_2$, $e_3$ and $e_4$ are four distinct numbers.
If, say, $e_1$, $e_2$ and $e_3$ are three distinct numbers and $e_4=e_3$, then
 there will instead be $2(k^2-k)$ two letter monomials SE to
$e_3x_u^tx_v^t$,  $k^2-k$  two letter terms SE to $e_1x_u^tx_v^t$,
$k^2-k$ two letter terms SE to $e_2x_u^tx_v^t$, $k$ two letter terms
SE to $\varphi(t,t)x_u^tx_v^t$ and the one letter monomials would be
as they are stated above.
\end{remark}

\begin{lemma}
\label{lem:aug26a10}
Let $p_1,\ldots,p_{k^2}$ be a family of polynomials in the $2k^2$
commuting variables $x_1,\ldots,x_{2k^2}$
that admits an nc representation $p(X,Y)$ of degree $d>1$.
Suppose further that $t \ge 1$,
\begin{equation}
\label{eq:aug27c10}
p(X,Y)=f_1(XY)^tX+f_2(YX)^tY+f_3X^{2t+1}+f_4Y^{2t+1}+q(X,Y),
\end{equation}
where $q$ is a polynomial that does not contain any scalar multiples of the
first four monomials listed in (\ref{eq:aug27c10}),
the coefficients $f_1,\ldots,f_4$ are distinct and $x_u\ne x_v$.
Then the family of $k^2$ polynomials will contain exactly
$$
\begin{array}{lll}
k^2-k&\quad \textrm{2 letter monomials SE  $f_1x_u^{t+1}x_v^t$}
&\quad\textrm{if $f_1\ne 0$} \\
\vspace{6pt}
k^2-k&\quad \textrm{2 letter monomials SE  $f_2x_u^tx_v^{t+1}$}
&\quad\textrm{if $f_2\ne 0$}\\
\vspace{6pt}
k^2-k&\quad \textrm{2 letter monomials SE  $f_3x_u^{t+1}x_v^t$}
&\quad\textrm{if $f_3\ne 0$}\\
\vspace{6pt}
k^2-k&\quad \textrm{2 letter monomials SE  $f_4x_u^tx_v^{t+1}$}
&\quad\textrm{if $f_4\ne 0$}\\
\vspace{6pt}
k&\quad \textrm{2 letter monomials SE  $\varphi(t+1,t)x_u^{t+1}x_v^t$}
&\quad\textrm{if $\varphi(t+1,t)\ne 0$}\\
\vspace{6pt}
k&\quad \textrm{2 letter monomials SE  $\varphi(t,t+1)x_u^tx_v^{t+1}$}
&\quad\textrm{if $\varphi(t,t+1)\ne 0$}\\
\vspace{6pt}
k&\quad \textrm{1 letter monomials SE  $f_3x_u^{2t+1}$}
&\quad\textrm{if $f_3\ne 0$}\\
\vspace{6pt}
k&\quad \textrm{1 letter monomials SE  $f_4x_u^{2t+1}$}
&\quad\textrm{if $f_4\ne 0$}
\end{array}
$$
This list incorporates all the ways that one letter monomials of degree
$2t+1$ and two letter monomials $\vtr$  to $x_u^{t+1}x_v^t$ can appear in the
given family of polynomials.
Moreover, there is no cancellation.

Each of the $k^2-k$ polynomials that are off  the diagonal in the $k\times k$
array corresponding to $p(X,Y)$ contains
$$
(\chi(f_1)+\chi(f_2)+\chi(f_3)+\chi(f_4))\quad
\text{two letter monomials $\vtr$ to
 $x_u^{t+1}x_v^t$}
$$
made up of off-diagonal letters.  Each of the $k$ polynomials that are on the
diagonal contains
\begin{equation}
\label{eq:kondiag2}
\chi(\varphi(t+1,t))+\chi(\varphi(t,t+1))\quad\textrm{two letter monomials
 $\vtr$ $x_u^{t+1}x_v^t$}
\end{equation}
made up of diagonal letters, as well as a one letter monomial SE to
$f_3x_u^{2t+1}$ if $f_3\ne 0$ and a one letter monomial SE to $f_4x_u^{2t+1}$
if $f_4\ne 0$. The letters in these one letter monomials are diagonal entries.
\end{lemma}

\begin{proof}
Let $\alpha=(\alpha_1,\ldots,\alpha_\ell)$ and $(\beta_1,\ldots,\beta_\ell)$
be multi-indices that meet conditions (\ref{eq:aug27a10}) and
(\ref{eq:aug25a10}) and set $s=t+1$.
Two letter words $x_i^{t+1}x_j^t$ with $i\ne j$ and $t\ge 2$ may be
generated in
four different ways:
\begin{enumerate}
\item[\rm(1)]  by $m_{\alpha,\beta}(X,Y)$ if $x_i\in X$ and $x_j\in Y$;
\vspace{6pt}
\item[\rm(2)] by $m_{\alpha,\beta}(Y,X)$ if $x_i\in Y$ and $x_j\in X$;
\vspace{6pt}
\item[\rm(3)] by $X^{2t+1}$ if $x_i\in X$ and $x_j\in Y$; \ and
\vspace{6pt}
\item[\rm(4)] by $Y^{2t+1}$ if $x_i\in Y$ and $x_j\in Y$.
\end{enumerate}
If (\ref{eq:aug28d10}) is in force, then the two letter words $x_i^sx_j^t$
with $i\ne j$ can only come from diagonal pairs in the same position.
The same terms with possibly different coefficients may appear from the
diagonal entries in $X$ and $Y$ from polynomials of degree $t+1$ in $X$ and
$t$ in $Y$  or degree $t$ in $X$ and degree $t+1$ in $Y$ in $q(X,Y)$.
The coefficients of the net contribution are $\varphi(t+1,t)$ and
$\varphi(t,t+1)$, respectively, and there will be a total of
$k\chi(\varphi(t+1,t)$ and $k\chi(\varphi(t,t+1)$ such pairs, one of each
sort in each  polynomial on the diagonal of the $k\times k$ array of $p(X,Y)$.

On the other hand, if (\ref{eq:aug29a10}) is in force, then
$$
\alpha_1=\cdots=\alpha_\ell=1,\quad \beta_1=\cdots=\beta_{\ell-1}=1,
\quad\beta_\ell=0
$$
$$
m_{\alpha,\beta}(X,Y)=(XY)^tX\quad\textrm{and}\quad
m_{\alpha,\beta}(Y,X)=(YX)^tY.
$$
Thus, if
$$
X=x_iE_{ab}+\cdots \quad\textrm{and}\quad Y=x_jE_{cd}+\cdots ,
$$
then, in view of assertions 2 and 4 of Lemma \ref{lem:aug25b10}, the
coefficient of $x_i^sx_j^t$ in
$m_{\alpha,\beta}(X,Y)$ and $m_{\alpha,\beta}(Y,X)$ will be nonzero if and
only if
$c=b$ and $d=a$. Correspondingly,
$$
(XY)^tX=x_i^{t+1}x_j^tE_{ab}+\cdots \quad\textrm{and}\quad
(YX)^tY=x_i^tx_j^{t+1}E_{ba}+\cdots.
$$
Since there $k^2-k$ off-diagonal positions in a $k\times k$ matrix, there are
$k^2-k$ choices of $E_{ab}$ with $a\ne b$. Moreover, the entry
$f_1x_i^{t+1}x_j^t$ can not cancel the entry $f_2x_i^tx_j^{t+1}$ even if
$a=b$,
since $x_i$ and $x_j$ are in different matrices. However, there can be
contributions from monomials in $q(X,Y)$ of degree $t+1$ in $X$ and $t$ in $Y$
or degree $t$ in $X$ and $t+1$ in $Y$.

Similarly, if  $a\ne b$ and $X=x_iE_{ab}+x_jE_{ba}+\cdots$ (resp.,
$Y=x_iE_{ab}+x_jE_{ba}+\cdots$), then
$$
X^{s+t}=x_i^sx_j^tE_{ab}+x_i^tx_j^sE_{ba}+\cdots\quad(\textrm{resp.,}\
Y^{s+t}=x_i^sx_j^tE_{ab}+x_i^tx_j^sE_{ba}+\cdots ).
$$
The final assertion comes by counting the contributions discussed above.
\end{proof}

\begin{remark}
The list in Lemma \ref{lem:aug26a10} is written under the assumption that
$f_1$, $f_2$, $f_3$ and $f_4$ are four distinct numbers.
If, say, $f_1$, $f_2$ and $f_3$ are three distinct numbers and $f_4=f_3$, then
 there will instead be $2(k^2-k)$ two letter monomials SE to
$f_3x_u^{t+1}x_v^t$,  $(k^2-k)$  two letter terms SE to $f_1x_u^{t+1}x_v^t$,
$(k^2-k)$ two letter terms SE to $f_2x_u^{t+1}x_v^t$, $k$ two letter terms
SE to $\varphi(t+1,t)x_u^{t+1}x_v^t$ and the one letter monomials would be
as they are stated above.
\end{remark}

\subsubsection{Enumeration of two letter monomials with one letter on the
diagonal}

\medskip

\begin{lemma}\label{lem:sep19a10}
Let $p_1,\ldots,p_{k^2}$ be a family of polynomials in the $2k^2$
commuting variables $x_1,\ldots,x_{2k^2}$
that admits an nc representation $p(X,Y)$ of degree $d>1$.
Suppose further that $t \ge 2$,
\begin{equation}
\label{eq:sep19e10}
p(X,Y)=d_1X^t +d_2X^{t-1}Y+d_3YX^{t-1} +q(X,Y),
\end{equation}
where $q(X,Y)$ does not contain any scalar multiples of the first three
monomials in (\ref{eq:sep19e10})
and the coefficients $d_1,\ldots,d_3$ are distinct.

If $x_u$ is a diagonal element of $X$, then the family of $k^2$ polynomials
will contain exactly
$$
\begin{array}{ll}
2k-2& \textrm{2 letter monomials $d_1  x_u^{t-1}x_v$ with $x_v\ne x_u$ if
$d_1\ne 0$}\\
\vspace{2mm}
k-1& \textrm{2 letter mnmls $d_2  x_u^{t-1}x_v$ with $x_v$ an off-diagonal
entry of $Y$ if $d_2\ne 0$}\\
\vspace{2mm}
k-1& \textrm{2 letter mnmls $d_3  x_u^{t-1}x_v$ with $x_v$ an
off-diagonal entry of $Y$ if $d_3\ne 0$}\\
\vspace{2mm}
1&\textrm{2 letter mnml $\varphi(t-1,1)x_u^{t-1}x_v$ with $x_v$ a
diagonal entry of $Y$ if $\varphi(t-1,1)\ne 0$}
\end{array}
$$

This list incorporates all the ways that two letter monomials of degree $t$
with $x_u$ of degree $t-1$  can appear in the given family of polynomials.
Moreover, no two monomials in this list are the same.

If $x_u$ is in the $aa$ position of $X$ and $p_{ab}$ denotes the polynomial
in the $ab$ position in the $k\times k$ array corresponding to $p(X,Y)$,
then
\begin{equation}
\label{eq:sep20a10}
p_{ab}(x_1,\ldots,x_{k^2})=d_1x_u^{t-1}x_v+d_2x_u^{t-1}x_w+\cdots\quad
\textrm{for $a\ne b$},
\end{equation}
where $x_v$ (resp., $x_w$) is in the $ab$ position in $X$ (resp., $Y$) and
there are no other two letter monomials of degree $t$ in $p_{ab}$ with
$x_u^{t-1}$ as a factor.  Similarly,
\begin{equation}
\label{eq:sep20b10}
p_{ba}(x_1,\ldots,x_{k^2})=d_1x_u^{t-1}x_m+d_2x_u^{t-1}x_n+\cdots\quad
\textrm{for $a\ne b$},
\end{equation}
where $x_m$ (resp., $x_n$) is in the $ba$ position in $X$ (resp., $Y$) and
there are no other two letter monomials of degree $t$ in $p_{ba}$ with
$x_u^{t-1}$ as a factor.
\end{lemma}

\begin{proof}
It is readily checked with the aid of the calculations in \S
\ref{subsec:precalc} that if
\begin{enumerate}
\item[\rm(1)] $x_u$ is in the $aa$ position of $X$, $x_v\in X$ and
$x_v\ne x_u$, then $x_u^{t-1}x_v$ is either in the $ab$ position of $X^t$ or
the $ba$ position of $X^t$ for some $b\ne a$;
\vspace{2mm}
\item[\rm(2)]  $x_u$ is in the $aa$ position of $X$ and $x_v\in Y$,
then $x_u^{t-1}x_v$ is either in the $ab$ position of $X^{t-1}Y$ or
the $ba$ position of $YX^{t-1}$ for some $b\ne a$.
\end{enumerate}
The rest of the proof is straight forward counting and is left to the reader.
\end{proof}

\begin{remark}
\label{rem:sep20a10}
Let
\begin{equation}
\label{eq:sep20c10}
p(X,Y)=d_1X^t+d_2X^{t-1}Y+d_3YX^{t-1}+d_4Y^t+q(X,Y)\quad\textrm{for some
integer $t\ge 2$},
\end{equation}
where $q(X,Y)$ does not contain any scalar multiples of the first four
monomials in (\ref{eq:sep20c10}) and assume that
$d_1,\ldots,d_4$ are subject to the constraints
\begin{equation}
\label{eq:sep20d10}
d_1\ne 0,\quad d_1\ne d_4\quad and\quad d_1\ne d_2\quad or \quad d_1\ne d_3.
\end{equation}
Then there will be $k$ terms
$$
d_1x_{i_1}^t,\ldots,d_1x_{i_k}^t
$$
in the family of $k^2$ polynomials and the corresponding letters
$x_{i_1}^t,\ldots,   x_{i_k}^t$ may be identified as the diagonal elements of
say $X$. The assumption $d_1\ne d_4$ insures that they can be chosen
unambiguously. Thus, if $x_u$ is one of these diagonal elements and it is in
the $aa$ position of $X$, then
\begin{eqnarray*}
p_{aa}&=&d_1x_u^t+\varphi(t-1,1)x_u^{t-1}x_h+\cdots,\\
p_{ab}&=&d_1x_u^{t-1}x_v+d_2x_u^{t-1}x_f+\cdots, \\
p_{ba}&=&d_1x_u^{t-1}x_z+d_3x_u^{t-1}x_g+\cdots,
\end{eqnarray*}
where $x_h$ is in the $aa$ position of $Y$, $x_v$ is in the $ab$ position of
$X$, $x_f$ is in the $ab$ position of $Y$, $x_z$ is in the $ba$ position of
$X$ and $x_g$ is in the $ba$ position of $Y$.
\end{remark}

\subsection{Determining Diagonal Elements}
 Lemmas \ref{lem:aug24a10} and \ref{lem:aug26a10} serve to enumerate
 the diagonal entries in $X$ and $Y$ when the given set of $k^2$
 polynomials contain one letter monomials. But if say
$e_3=e_4$ in Lemma \ref{lem:aug24a10} and $f_3=f_4$ in Lemma
 \ref{lem:aug26a10},
then it is not immediately obvious which entries belong to $X$ and which
entries belong to $Y$.

\begin{definition}  A pair of variables $x_i$ and $x_j$ will be called a \df{ partitioned (resp., dyslexic) diagonal pair} if both occur
in the $aa$ position for some choice of $a\in\{1,\ldots,k\}$ and  we know (resp., do not know) which variable occurs in $X$ and
 which occurs in $Y$.
 \end{definition}

 \begin{lemma}
\label{lem:aug18a9}
Suppose $p_1,\ldots,p_{2k^2}$ is a family of polynomials in the $2k^2$
commuting
variables $x_1,\ldots,x_{2k^2}$ that admits an nc representation
$p(X,Y)$ of degree $d$ with $d>1$ such that at least one of the given
polynomials contains a term of the form
$ex_{i}^{s}m(x_{1},...,\what{x_{i}},...,x_{2k^{2}})$, where
$e\in\RR\setminus\{0\}$ and $m(x_{1},...,\what{x_{i}},...,x_{2k^{2}})$ is a
monomial of degree $t\ge 1$ that does not contain any $x_i$ terms and
$s \ge t+2.$
Then $x_{i}$ lies on
the diagonal of either $X$ or $Y$.
\end{lemma}

\begin{proof}
For the sake of definiteness, assume $x_i \in X$.  Then, since
$s\ge t+2$, every permutation of the symbols $X^sY^t$ must contain at least two adjacent $X^\prime$s. But if
$$
X=x_iE_{cd}+\cdots\quad\textrm{with}\ c\ne d,
$$
then $X^2=0$. Thus, the given family of polynomials will only contain terms of the form
$ex_{i}^{s}m(x_{1},...,\what{x_{i}},...,x_{2k^{2}})$ if $c=d$, i.e., if
$x_i$ is a diagonal element of $X$.
\end{proof}

\begin{lemma}
\label{lem:aug16a10}
Let $p_1,\ldots,p_{k^2}$ be a family of polynomials in the $2k^2$
commuting variables $x_1,\ldots,x_{2k^2}$
that admits an nc representation $p(X,Y)$ of degree $d>1$.
Suppose further that at least one of the polynomials
contains at least one term of the form $ex_i^sx_j^t$ with $s \ge t+2$, and
$t \ge 2$.  Then

\begin{enumerate}
\item[\rm(1)] $x_i$ and $x_j$ are a dyslexic diagonal pair.
\vspace{6pt}
\item[\rm(2)] There exist exactly $k$ dyslexic diagonal pairs
$\{x_{i_1}, x_{j_1}\}, \ldots,\{ x_{i_k}, x_{j_k}\}$ and $k$ polynomials
$p_{\ell_1},\ldots,p_{\ell_k}$ such that
$p_{\ell_n}=ex_{i_n}^sx_{j_n}^t+\cdots$ for $n=1,\ldots,k$. Moreover,
$p_{\ell_n}$
occupies the same diagonal position as $x_{i_n}$ and $x_{j_n}$.
\end{enumerate}

\end{lemma}

\begin{proof}  If  $x_i\in X$, then (since $s\ge t+2$ and $t\ge 2$) Lemma
\ref{lem:aug18a9} guarantees
that $x_i$ lies on the diagonal of $X$ and that $x_j$ lies on the
diagonal of the matrix that
contains $x_j$. Since $x_i$ and $x_j$ must be in the same diagonal
position, this forces $x_j$ to belong to $Y$. Similarly, if $x_i\in Y$,
then $x_j$ must be in $X$ and both variables must be in the same
diagonal position. Thus (1) holds; (2) follows automatically from (1),
since all
the dyslexic diagonal pairs are subject to the same constraints.
\end{proof}

\begin{lemma}
\label{lem:aug16b10}
Let $p_1,\ldots,p_{k^2}$ be a family of polynomials in the $2k^2$
commuting variables $x_1,\ldots,x_{2k^2}$
that admits an nc representation $p(X,Y)$ of degree $d>1$.
If $t \ge 2$ and $s = t+1$ or $s=t$ and it is also assumed that

\begin{enumerate}
\item[\rm(1)] $k\ge 3$ (so that $k^2 - k > k $) and
\vspace{6pt}
\item[\rm(2)] there exist exactly $k$ monomials
$ex_{i_1}^sx_{j_1}^t, \cdots,e x_{i_k}^sx_{j_k}^t$ in the given
family of polynomials that are structurally equivalent to $ex_u^sx_v^t$ with
$x_u\ne x_v$ and $e\ne 0$,
\end{enumerate}
then $x_{i_n}$ and $x_{j_n}$ are a dyslexic diagonal pair for each
$1 \le n \le k$.
Moreover, if $p_{\ell_n}=ex_{i_n}^tx_{j_n}^t+\cdots$, then $p_{\ell_n}$ occupies the same diagonal position as $x_{i_n}$ and $x_{j_n}$.
\end{lemma}

\begin{proof}
Since there are only $k$ terms in this list and $k^2-k>k$ when $k\ge 3$,
the conclusion follows from Lemma \ref{lem:aug24a10} if $s=t$ and from
Lemma \ref{lem:aug26a10} if $s=t+1$. In the first case $e=\varphi(t,t)$; in
the second case either $e=\varphi(t+1,t)$ or $e=\varphi(t,t+1)$ and the
two numbers $\varphi(t+1,t)$ and $\varphi(t,t+1)$ are different.
\end{proof}

\begin{lemma}
Suppose that the term $ax_i^sx_j^t$, $a\ne 0$, occurs in some polynomial $p$ in the given family.  If the pair $x_i$ and $x_j$ is a
dyslexic diagonal pair, then the polynomial is a diagonal polynomial that occurs in the same diagonal position as the dyslexic pair.
\end{lemma}

\begin{proof}
If, say, $X=x_iE_{aa}+\cdots$, and $Y=x_jE_{aa}+\cdots$ and the multi-indices
$\alpha$ and $\beta$ are subject to (\ref{eq:aug27a10}) and
(\ref{eq:aug25a10}), then $m_{\alpha,\beta}(X,Y)=x_1^sx_j^tE_{aa}+\cdots$.
\end{proof}

\begin{lemma} \label{lem:aug16d10}
Let $p_1, \cdots, p_{2k^2}$ be a family of polynomials in the $2k^2$
commuting
variables $x_1,\ldots,x_{2k^2}$ that admits an nc representation
$p(X,Y)$ of degree $d$ with $d>1$ such that
exactly $k$ SE terms of the form
$ax_{i_1}^sx_{j_1}^t, \cdots,a x_{i_k}^sx_{j_k}^t$ appear in the
family with $s > t$,  $t\ge 2$ and $a \ne 0$.  If
$\{x_{i_1},x_{j_1}\}, \cdots, \{x_{i_k}, x_{j_k}\}$
are dyslexic diagonal pairs, then the variables $x_{i_1}, \cdots, x_{i_k}$ of degree $s$
in the terms $ax_{i_1}^sx_{j_1}^t, \cdots,a x_{i_k}^sx_{j_k}^t$
are the diagonal elements of one matrix and and the variables $x_{j_1}, \cdots, x_{j_k}$ of degree $t$ are the diagonal elements
of the other matrix.
\end{lemma}
\begin{proof}
If the lemma is false, then  without loss of generality, we may suppose that
$x_{i_1} \in X$ and $x_{i_2} \in Y$.   By Lemmas~\ref{lem:aug16a10}
and~\ref{lem:aug26a10}, $\varphi(s,t)=\varphi(t,s)=a\ne 0$. Thus, as
$s > t$, each diagonal pair $x_{i_n}, x_{j_n}$ will
occur in the monomials $ax_{i_n}^sx_{j_n}^t$ and $ax_{i_n}^tx_{j_n}^s$.
Therefore, there will be $2k$ terms structurally equivalent to
$ax_{i_1}^sx_{j_1}^t$, which contradicts one of the given assumptions.
\end{proof}

\section{Partitioning Algorithms for families containing single letter
monomials $ax_i^n$, $n\ge 2$}\label{sec:algPosit}
\label{sec:nov25c11}

In the previous section we developed a number of methods to determine the
diagonal variables for a family $\cP$ of polynomials $p_1,\cdots,p_{k^2}$
with an nc representation. The next step is to
utilize this information to determine which of the commutative variables
 $x_1,\cdots,x_{2k^2}$ are entries in $X$ and which are
entries in $Y$. In this section we develop two algorithms that we will
refer to
as \df{partitioning algorithms},  since they partition the commutative
variables between the two matrices.
First, however,  we review some preliminary calculations that will be
 essential in the
development of the  first partitioning algorithm.

The following assumptions will be in force for the rest of this section:
\medskip

{\bf (A1)}  $p_1,\ldots,p_{k^2}$ is a family of polynomials in  $2k^2$
commuting
variables $x_1,\ldots,x_{2k^2}$ that admits an nc representation
$p(X,Y)$ of degree $d$ with $d>1$.
\medskip

{\bf (A2)} $\varphi(n,0)\ne 0$ for some $n\ge 2$.

\medskip

{\bf (A3)} $p(0,0)=0$. (This involves no real loss of generality, because
constant terms can be reinserted at the end.)
\medskip

\subsection{Preliminary calculations  for the partitioning algorithm }\label{prelimPart}
\bigskip

First observe that
$$
(\alpha I_k+\beta E_{st})^n=\alpha^nI_k+n\alpha^{n-1}\beta E_{st}\quad
\textrm{if}\ s\ne t.
$$
Then without loss of generality, assume that the variables
$\{x_{i_1}, \cdots, x_{i_k}\}$ are the commutative variables
that occur on the diagonal of $X$.

Next, choose a variable $x_j$ such that
$$
x_j\not\in\{x_{i_1},\ldots,x_{i_k}\},
$$
set
$$
x_{i_1}=\cdots=x_{i_k}=\alpha\quad\textrm{and}\quad x_j=\beta
$$
and set all the other variables equal to zero.

Then there are three mutually exclusive possibilities:
\bigskip

\noindent
{\bf 1.} $x_j\in X$. In this case
$$
X=\alpha I_k+\beta E_{st}\quad\textrm{with $s\ne t$ and}\quad Y=0.
$$

\noindent
{\bf 2.} $x_j\in Y$ but is not on the diagonal of $Y$. In this case
$$
X=\alpha I_k\quad\textrm{and}\quad Y=\beta E_{st}\
\textrm{with $s\ne t$}.
$$

\noindent
{\bf 3.} $x_j\in Y$ and is on the diagonal of $Y$. In this case
$$
X=\alpha I_k\quad\textrm{and}\quad Y=\beta E_{ss}\quad
\textrm{for some integer}\ s\in\{1,\ldots,k\}.
$$
\\

All three of these cases fit into the common framework of choosing
$X=A\in\RR^{k\times k}$ and $Y=B\in\RR^{k\times k}$ with $AB=BA$.
Thus, if
$$
p(X,Y)=\sum_{i=0}^d p_{[i]}(X,Y),
$$
where $p_{[i]}(X,Y)$ denotes the terms in $p(X,Y)$ of degree $i$, then
the condition $AB=BA$ insures that
\begin{equation}
\label{eq:aug4a10}
p_{[i]}(A,B)=\sum_{j=0}^i c_{i-j,j}A^{i-j}B^j=c_{i0}A^i
+\sum_{j=1}^i c_{i-j,j}A^{i-j}B^j =\what p_{[i]}(A,B),
\end{equation}
where
$$
c_{ij}=\varphi(i,j)\quad\textrm{(for short) for $i,j=0,\ldots,d$ with
$c_{ij}=0$ for $i+j>d$}
$$
and hence,
\begin{equation}
\label{eq:aug4b10}
p(A,B)=\sum_{i=1}^d\sum_{j=0}^ic_{i-j,j}A^{i-j}B^j=\sum_{i=1}^dc_{i0}A^i
+\sum_{i=1}^d\sum_{j=1}^ic_{i-j,j}A^{i-j}B^j =\what p(A,B),
\end{equation}
since the assumption $p(0,0)=0$ forces $c_{00}=0$.
\bigskip

In Case 1, $A=\alpha I_k+\beta E_{st}$ with $s\ne t$ and $B=0$. Therefore,
\begin{equation}
\label{eq:aug4c10}
p_{[i]}(A,B)=c_{i0}(\alpha^i I_k+i\alpha^{i-1}\beta E_{st})\quad\textrm{for}\
i\ge 1    \qquad and
\end{equation}
\begin{eqnarray}
\nonumber
p(A,B)
&=&\sum_{i=1}^dc_{i0}\alpha^iI_k+\sum_{i=1}^dc_{i0} i\alpha^{i-1}
\beta E_{st}\\
\label{eq:aug4d10}
&=&\sum_{i=1}^dc_{i0}\alpha^iI_k+
\sum_{i=0}^{d-1}c_{i+1,0} (i+1)\alpha^i\beta E_{st}.
\end{eqnarray}
\bigskip

In Case 2, $A=\alpha I_k$ and $B=\beta E_{st}$ with $s\ne t$. Therefore,
\begin{equation}
\label{eq:aug4e10}
p_{[i]}(A,B)=c_{i0}\alpha^iI_k+c_{i-1,1}\alpha^{i-1}\beta E_{st}
\quad\textrm{for}\ i\ge 1 \qquad    and
\end{equation}
\begin{equation}
\label{eq:aug4f10}
p(A,B)=
\sum_{i=1}^dc_{i0}\alpha^iI_k+\sum_{i=1}^dc_{i-1,1} \alpha^{i-1}\beta E_{st}
\end{equation}
\bigskip

In Case 3, $A=\alpha I_k$ and $B=\beta E_{ss}$. Therefore,
\begin{equation}
\label{eq:aug4g10}
p_{[i]}(A,B)=c_{i0}\alpha^i I_k+\sum_{j=1}^i c_{i-j,j}\alpha^{i-j}\beta^j E_{ss}
\quad\textrm{for}\ i\ge 1 \qquad  and
\end{equation}
\begin{equation}
\label{eq:aug4h10}
p(\alpha I_k,\beta E_{ss})=\sum_{i=1}^dc_{i0}\alpha^iI_k+
\sum_{i=1}^d\sum_{j=1}^ic_{i-j,j}\alpha^{i-j}\beta^jE_{ss}.
\end{equation}

We remark that the formula $p(\alpha I_k+\beta E_{st},0)$ in Case 1, can
also be expressed in terms of the
polynomial
$$
\varphi(\alpha)=\sum_{i=1}^dc_{i0}\alpha^i
$$
as
$$
p(\alpha I_k+\beta E_{st}, 0)=\varphi(\alpha)I_k+\varphi^\prime(\alpha)\beta
E_{st}.
$$
The equality $c_{d0}=a\ne 0$ guarantees that $\varphi(\alpha)\not\equiv 0$ and
$\varphi^\prime(\alpha)\not\equiv 0$.


\subsection{Partitioning Algorithm I}
\label{sec:algDiagPar}

Now we develop the fundamental ideas for an algorithm that
will partition the commutative variables between $X$ and $Y$.

Set the diagonal entries of $X$ equal to $\alpha$, one of the
other $2k^2-k$ variables $x_i$ equal to $\beta$ and the remaining
$2k^2-k-1$ variables equal to zero. Then:

\begin{eqnarray*}
x_i&{ }&\  \textrm{is an off-diagonal entry of $X$}\Longleftrightarrow \textrm{this substitution produces}\\
&{ }& k \ \textrm{polynomials equal to
$\sum_{i=1}^dc_{i0}\alpha^i$},\\
&{ }& \textrm{$1$ polynomial equal to}\  \sum_{i=0}^{d-1}c_{i+1,0}(i+1)\alpha^i\beta
\ \textrm{and}\\
&{ }& \textrm{$k^2-k-1$ polynomials equal to $0$};
\end{eqnarray*}

\begin{eqnarray*}
x_i&{ }&\  \textrm{is an off-diagonal entry of $Y$}\ \Longleftrightarrow
\textrm{this substitution produces} \\
&{ }& k \ \textrm{polynomials equal to}\
\sum_{i=1}^dc_{i0}\alpha^i, \\
&{ }& \textrm{$1$ polynomial equal to}\
\sum_{i=0}^{d-1}c_{i1}\alpha^i\beta \  \textrm{and} \\
&{ }& k^2-k-1\ \textrm{polynomials equal to $0$};
\end{eqnarray*}

\begin{eqnarray*}
x_i &{ }& \ \textrm{is a diagonal entry of $Y$} \Longleftrightarrow   \textrm{this substitution produces}\\
&{ }&
k-1 \ \textrm{polynomials equal to} \
\sum_{i=1}^dc_{i0}\alpha^i, \\
&{ }& \textrm{$1$ polynomial equal to}\
\sum_{i=1}^dc_{i0}\alpha^i+\sum_{i=1}^d\sum_{j=1}^{i}c_{i-j,j}\alpha^{i-j}\beta^j\quad
\textrm{and}\\
&{ }&  k^2-k\ \textrm{polynomials equal to}\  0.
\end{eqnarray*}

The first two cases will be indistinguishable if and only if
$$
\sum_{i=0}^{d-1}c_{i+1,0} (i+1)\alpha^i\beta=
  \sum_{i=0}^{d-1}c_{i1} \alpha^i\beta
$$
for every choice of $\alpha$ and $\beta$, i.e., if and only if
$$
c_{i+1,0}(i+1)=c_{i1} \quad\textrm{for}\ i=0,\ldots,d-1.
$$

\subsection{Partitioning with homogeneous components:
Algorithm DiagPar1}\label{sec1:11may12}
It is often advantageous to focus on the homogeneous components of the
given set of polynomials $p_1,\ldots,p_{k^2}$,
i.e., on the sub-polynomials
of $p_1,\ldots,p_{k^2}$ of specified degree. It is readily seen that if
$p_{[n]}(X,Y)$ denotes the terms in the nc polynomial of degree
$n$,
then the $k^2$ commuting polynomials in the $k\times k$ array corresponding
to $p_{[n]}(X,Y)$ are the $k^2$ polynomials $q_1,\ldots,q_{k^2}$, where
$q_i(x_1,\ldots,x_{2k^2})$ is the sum of the monomials in $p_i(x_1,\ldots,x_{2k^2})$
of degree $n$.

The three possibilities considered earlier applied to polynomials of degree
$n$ lead to
simpler criteria:
\bigskip

\noindent
{\bf 1.}   If $A=\alpha I_k+\beta E_{st}$ with $s\ne t$ and $B=0$, then
$$
p_{[n]}(A,B)=c_{n0}(\alpha^nI_k+n\alpha^{n-1}\beta E_{st}).
$$
{\bf 2.}  If $A=\alpha I_k$ and $B=\beta E_{st}$ with $\beta\ne 0$
and $s\ne t$, then
$$
p_{[n]}(A,B)=c_{n0}\alpha^n I_k+c_{n-1,1}\alpha^{n-1}\beta E_{st}.
$$
{\bf 3.}  If $A=\alpha I_k$ and $B=\beta E_{ss}$, then,
$$
p_{[n]}(A,B)=c_{n0}\alpha^n I_k+\sum_{j=1}^nc_{n-j,j}\alpha^{n-j}\beta^jE_{ss}.
$$

Thus, if the diagonal entries of $X$ are set equal to $\alpha$, one of the
other $2k^2-k$ variables $x_i$ is set equal to $\beta$  and
the remaining
$2k^2-k-1$ variables are set equal to zero,
and if $\beta\ne\alpha$ and $\varphi(n,0)\ne 0$,
then:
\begin{eqnarray*}
&x_i&\  \textrm{is an off-diagonal entry of  $X$}\Longleftrightarrow \textrm{this substitution
produces} \\
&k&\  \textrm{ polynomials equal to
$c_{n0}\alpha^n$}, \\
&1&\  \textrm{polynomial equal to  $c_{n0}n\alpha^{n-1}\beta$ and}\\
&k^2-k-1&\ \textrm{ polynomials equal to $0$};
\end{eqnarray*}

\begin{eqnarray*}
&x_i&\  \textrm{is an off-diagonal entry of $Y$}  \Longleftrightarrow
\textrm{this substitution produces}\\
 &k&\  \textrm{polynomials equal to}\  c_{n0}\alpha^n\\
&1&\  \textrm{ polynomial equal to $c_{n-1,1}\alpha^{n-1}\beta$  and}\\
&k^2-k-1&\ \textrm{polynomials equal to $0$};
\end{eqnarray*}

\begin{eqnarray*}
&x_i&\  \textrm{is a diagonal entry of $Y$} \Longleftrightarrow   \textrm{this substitution produces} \\
&k-1&\ \textrm{polynomials equal to}\   c_{n0}\alpha^n\\
&1&\  \textrm{polynomial equal to}\
c_{n0}\alpha^n+
 \sum_{j=1}^nc_{n-j,j}\alpha^{n-j}\beta^j\quad \textrm{and}\\
 &k^2-k&\ \textrm{polynomials equal to}\ 0.
\end{eqnarray*}

\ben
\item
The first two possibilities will be distinguishable
if and only if

 $n \varphi(n,0) \ne \varphi(n-1,1)$.
\item
The third will
be distinguishable from the first if
$\varphi(n,0) \ne 0$
 by counting the number of polynomials equal to zero,
  $k+1$ vs $k$

\item
The third will be distinguishable from the second if
 and  at least two of the coefficients
  $\varphi(n-j,j) \ne 0$
$j=0,1,\ldots,n$ are nonzero.
This is done by comparing number of terms of polynomials.
 \bigskip
\een

We will refer to the process developed in the previous discussion
as \df{Algorithm DiagPar1}.  In the following theorem we summarize
the conditions under which this algorithm will partition the commutative
variables. The reader should keep in mind that if
there are $k$ single letter monomials (as opposed to $2k$ or $0$) of
degree $n$ in a homogeneous family $\cP$ of polynomials of degree $n$, then
we always assume that the associated variables lie on the diagonal of $X$ and hence that $\varphi(n,0)\ne 0$.

\begin{theorem}[DiagPar1 Algorithm]
\label{lem:diagToPartit}
Let $p_1,\cdots,p_{k^2}$ be a family of homogeneous
 polynomials with an nc representation
 $p(X,Y)$ of degree $n$ with $n \geq 2$.
Then Algorithm DiagPar1 successfully partitions the variables
$x_1,\cdots,x_{2k^2}$
between $X$ and $Y$ if and only if
\begin{align}\label{eq1:30jan12}
\varphi(n,0) \ne 0,\quad \varphi(n,0) \ne \varphi(0,n), \quad
\text{and} \quad  n \varphi(n,0) \ne \varphi(n-1,1).
\end{align}
\end{theorem}
\begin{proof}
If $\varphi(n,0) \ne 0$ and \eqref{eq1:30jan12} holds, then
the above discussion implies that DiagPar1 will successfully
partition the variables between $X$ and $Y$.

Conversely, suppose
that DiagPar1 successfully partitions the variables between $X$
and $Y$.  This implies that the algorithm can first determine
the set of diagonal variables of $X$ by analyzing single letter monomials
of the form $ax_i^n$ in the polynomials
in $\cP$ .  By Lemma~\ref{lem:aug16b9},
we see that this is possible only if $\varphi(n,0) \ne 0$
and $\varphi(n,0) \ne \varphi(0,n)$.
Finally, the above discussion implies that DiagPar1 partitions the off-diagonal
elements between $X$ and $Y$ if and only if $\varphi(n,0) \ne \varphi(n-1,1)$.
\end{proof}

\subsection{Partitioning: Algorithm DiagPar2}\label{sec:DiagPar2}

In this subsection we develop another partitioning algorithm that is
closely related to theDiagPar1 Algorithm and is based on the following
observation:
Let $p_1,\cdots,p_{k^2}$ be a family of polynomials with an nc representation
$p(X,Y)$
and suppose that for some $t \ge 2$, $\varphi(t,0) \ne 0$ and $\varphi(t,0) \ne \varphi(0,t)$.
Then by Lemma \ref{lem:aug16b9}, there exists exactly $k$ polynomials
$p_{i_1},\cdots, p_{i_k}$,
each of which contains exactly one
one letter monomial of degree $t$ with coefficient $\varphi(t,0)$.

The next step rests on Lemma \ref{lem:sep19a10}. But
for ease of understanding, let $x_{ij}$ denote the $ij$ entry in $X$,
$y_{ij}$ the $ij$ entry in $Y$ and let $p_{ij}$ denote the $ij$ entry in an
array of commutative polynomials that admits an nc representation $p(X,Y)$.
Then
$$
p_{ii}=\varphi(t,0)x_{ii}^t+\varphi(0,t)y_{ii}^t+\cdots,
$$
whereas for $i\ne j$,
\begin{equation*}
\begin{split}
p_{ij}=&\varphi(t,0)(x_{ii}^{t-1}x_{ij}+x_{ij}x_{jj}^{t-1})+
\varphi(0,t)(y_{ii}^{t-1}y_{ij}+y_{ij}y_{jj}^{t-1})\\
&\quad
+\varphi(Y;t-1,1)y_{ij}x_{jj}^{t-1}+\varphi(t-1,1;Y)x_{ii}^{t-1}y_{ij}+\cdots
\end{split}
\end{equation*}
and
\begin{equation*}
\begin{split}
p_{ji}=&\varphi(t,0)(x_{ji}x_{ii}^{t-1}+x_{jj}^{t-1}x_{ji})+
\varphi(0,t)(y_{ji}y_{ii}^{t-1}+y_{jj}^{t-1}y_{ji})\\
&\quad
+\varphi(Y;t-1,1)y_{ji}x_{ii}^{t-1}+\varphi(t-1,1;Y)x_{jj}^{t-1}y_{ji}+\cdots
\end{split}
\end{equation*}
Thus, if the entry $x_{ii}$ and $t$ are known, then there will be exactly
two two
letter monomials in $p_{ij}$, $i\ne j$ with $x_{ii}^{t-1}$ as a factor,
namely, $\varphi(t,0)x_{ji}x_{ii}^{t-1}$ and
$\varphi(t-1,1;Y)x_{ii}^{t-1}y_{ij}$, but only the first of these will have
the correct coefficient if $\varphi(t,0)\ne \varphi(t-1,1;Y)$. Thus, under
this condition, it is possible to isolate all the entries in $X$ by repeating
the argument for $i=1,\ldots,k$.

Similar considerations based on inspection of the polynomials $p_{ji}$
will also yield all the entries in $X$ if $\varphi(t,0)\ne \varphi(Y;t-1,1)$.

We will refer to the above procedure for partitioning the commutative
variables as
\df{Algorithm DiagPar2}.  The conditions under which this algorithm
works are summarized in the following theorem.

\begin{theorem}[DiagPar2 Algorithm]\label{lem1:sep8:2010}
Let $p_1,\cdots,p_{k^2}$ be a family of polynomials with an nc representation $p(X,Y)$.
Then Algorithm DiagPar2 will partition the variables $x_1,\cdots,x_{2k^2}$ between $X$ and $Y$
if and only if there exists an $n \ge 2$ such that $\varphi(n,0) \ne 0$, $\varphi(n,0) \ne \varphi(0,n)$, and
\begin{align}\label{eq1:31jan12}
\varphi(n,0) \ne \varphi(n-1,1;Y) \quad \text{ or } \quad \varphi(n,0) \ne \varphi(Y;n-1,1).
\end{align}
\end{theorem}
\begin{proof}
If the above conditions hold then the above discussion implies that
Algorithm DiagPar2 will successfully partition the variables between
$X$ and $Y$.  Conversely, if DiagPar2 partitions the variables between $X$ and
$Y$, then it must first determine the diagonal elements.  Lemma~\ref{lem:aug16b9}
implies that this is only possible if $\varphi(n,0) \ne 0$ and $\varphi(n,0) \ne \varphi(0,n)$.
The above discussion implies that DiagPar2 will successfully partition the off-diagonal
entries only if the above conditions hold.
\end{proof}

\subsection{Summary of Partitioning Algorithms}  The main conclusion of this section is that if the given system $\cP$ of $k^2$
commutative polynomials admits an nc representation $p(X,Y)$ in
the set $NC_{(\ref{eq2:31jan12})}$ of
nc polynomials
$p(X,Y)$ of degree $d>1$ for which there exists an integer
$n \ge 2$ such that
\begin{equation}
\label{eq2:31jan12}
\left\{\begin{array}{l}
\varphi(n,0) \ne 0,\quad  \varphi(n,0)\ne \varphi(0,n),
\quad \textrm{and either} \\
n\varphi(n,0) \ne \varphi(n-1,1) \quad \textrm{or}\\
\varphi(n,0) \ne \varphi(n-1,1;Y) \quad \textrm{or}\quad \varphi(n,0)
\ne \varphi(Y;n-1,1),
\end{array}\right.
\end{equation}
then either Algorithm DiagPar1 or Algorithm DiagPar2
will
successfully partition the variables between $X$ and $Y$.

\section{Positioning algorithms for families of polynomials containing one
letter monomials}\label{sec:singlepos}
\label{sec:alg}

In this section we shall present algorithms for positioning the variables in
$X$,  given that the diagonal entries of $X$ are known
and that the remaining $2k^2-k$  commutative variables are
partitioned between $X$ and $Y$. The assumptions (A1), (A3) that are
listed at the beginning of \S \ref{sec:algPosit} and a weaker form of (A2):
\medskip

(A2$^\prime$) $\vert\varphi(n,0)\vert+\vert\varphi(0,n)\vert>0$ for some
$n\ge 2$,
\medskip

will be in force for the rest of this section.

\subsection{Positioning the variables within $X$: Algorithm ParPosX}
\label{subsec1:sep3:2010}

\bigskip

If we can determine the diagonal variables of the matrix
$X$ and implement Algorithm DiagPar1 or DiagPar2 in \S \ref{sec:algPosit},
then we may assume that $x_{i_1},\ldots,x_{i_{k^2}}\in X$,
and the remaining $k^2$ variables belong to $Y$.
To ease the notation,
assume that $x_1,\ldots,x_{k^2}\in X$ and that $x_i$ is in the $ii$ position
for $i=1,\ldots,k$ and let
\begin{eqnarray*}\label{eq1:31oct12}
R_i &{}&  \textrm{denote the remaining $k-1$ entries in
the $i$th row of}\ X,\\
C_i &{}&  \textrm{denote the remaining $k-1$ entries in
the $i$th column of}\ X, \
{\textrm{and}}\\
L_i&=& R_i\cup C_i.
\end{eqnarray*}
Then, since
\begin{equation}\label{eq1:sep3:2010}
X^n=(x_i^{n-1}E_{ii}+\cdots)(x_jE_{st}+\cdots)=
(x_i^{n-1}x_jE_{ii}E_{st}+\cdots)
\end{equation}
and
\begin{equation}\label{eq2:sep3:2010}
X^n=(x_jE_{st}+\cdots)(x_i^{n-1}E_{ii}+\cdots)
=(x_i^{n-1}x_jE_{st}E_{ii}+\cdots)
\end{equation}
for $1\le i\le k<j$, it is readily seen that the term $ax_i^{n-1}x_j$ appears
in one of the entries of $aX^n$ if and only if either $s=i$ or $t=i$, i.e.,
if and only if $x_j\in R_i\cup C_i$. Moreover, since $R_i\cap C_i=\emptyset$
there will be $2k-2$ such
terms in $X^n$.

Let
$$
L_i\cap L_r=\{x_j, x_\ell\}\quad
\textrm{for some integer $r\in\{1,\ldots,k\}\setminus\{i\}$}.
$$
Then one of these two variables will be in the $ir$ position of $X$, while
the other is in the $ri$ position, and it is impossible to decide which is
where without extra information. Let
us assume for the sake of definiteness that
$x_j$ is in the $ir$ position of
$X$, then $x_\ell$ will be in the $ri$ position. Moreover, since
$$
X^n=(x_jE_{ir}+\cdots)(x_r^{n-1}E_{rr}+\cdots)=(x_jx_r^{n-1}E_{ir}+\cdots),
$$
the term $x_jx_r^{n-1}$ also belongs to the same polynomial. Thus,
$$
X^n=\begin{bmatrix}q_{11}&\cdots&q_{1k}\\
\vdots& &\vdots\\ q_{k1}&\cdots&q_{kk}\end{bmatrix},
$$
where the $q_{st}$ are either homogeneous polynomials of degree $n$ in the
variables  $x_1,\ldots,x_{k^2}$ or zero. In particular, if $k\ge 3$,
$n\ge 2$ and $i\ne r$, then
$$
q_{ir}=(x_i^{n-1}x_j+x_jx_r^{n-1}+ x_i^{n-2}x_t x_m+\cdots)
$$
with $x_t$ and $x_m$ off the diagonal and $x_t\ne x_m$. This
automatically insures that $x_t$ and $x_m$ differ from $x_i$ and $x_j$,
i.e.,
$$
\{x_t,\, x_m\}\cap \{x_i,\,x_j\}=\emptyset,
$$
and that there exists an integer
$j\in\{1,\ldots,k\}\setminus\{i,r\}$ such that
$$
\textrm{either}\ x_t\in R_i\cap C_j\quad\textrm{and}\quad x_m\in
R_j\cap C_r \quad\textrm{or vice versa}.
$$
Therefore, since $R_i\cap R_j=\emptyset$ for $j\ne i$, only one of the two
variables $x_t$, $x_m$ is listed in the set $R_i\cup C_i$ and hence
these two variables can be positioned unambiguously. This procedure adds one
more variable to each of $R_i$ and  $C_r$. The remaining entries in $R_i$ and
$C_r$ are obtained by repeating this procedure $k-3$
more times by running through all the other triples of the form
$x_i^{n-2}x_ux_v$ in $q_{ir}$.

After $R_i$ and $C_r$ are filled in, the procedure is repeated with some
other diagonal element as a starting point, and then repeated again and
again until all the $k^2$ variables are positioned in $X$.

We will refer to the procedure for positioning the variables outlined
in the previous discussion as \df{Algorithm ParPosX}.  The
conditions under which this algorithm will position the commutative variables
are summarized in the following theorem.

\begin{theorem}\label{thm1:sep9:2010}[Algorithm ParPosX]
Let $p_1,\cdots,p_{k^2}$ be a family of polynomials with an nc
representation $p(X,Y)$.
Then Algorithm ParPosX will position the variables in $X$
if and only if $p \in NC_{(\ref{eq2:31jan12})}$.
\end{theorem}
\begin{proof}
DiagPar1 and DiagPar2 will partition the commutative variables between
$X$ and $Y$
if and only if $p \in NC_{(\ref{eq2:31jan12})}$.
The above algorithms will position the variables in $X$ if and only if the
commutative variables are partitioned between $X$ and $Y$.
\end{proof}
\medskip

\subsection{Algorithm for positioning the polynomials given positioning in $X$}

If $X$ and $Y$ are general matrices containing the variables $x_1,\cdots, x_{2k^2}$, an nc representation
$p(X,Y)$ produces a family of polynomials $p_1,\cdots,p_{k^2}$ that is arranged in a $k \times k$ matrix.  The purpose of
this section is to investigate when it is possible to determine the position of the polynomials $p_1,\cdots,p_{k^2}$ in the resulting $k \times k$ matrix.
We start with a lemma that will be useful in developing a procedure to accomplish this task.

\begin{lemma}\label{lem1:aug30:2010}
Let $p_1,\cdots,p_{k^2}$ be a homogeneous family of polynomials with an
nc representation $p(X,Y)$ of degree $d>1$.
Suppose that the variables $x_1,\cdots,x_{2k^2}$ have been partitioned
between $X$ and $Y$ and positioned in $X$.
Then, if $x_i$ is in the $aa$ position of $X$, $x_j$ is in the $ab$ position of
$X$, $x_\ell$ is in the $bb$ position of $X$ and
$$
\varphi(n,0) \ne 0 \quad \text{for some } \ n \ge 2,
$$
then there exists exactly one polynomial that contains the monomials
$\varphi(n,0)x_i^{n-1}x_j$ and $\varphi(n,0)x_jx_\ell^{n-1}$. Moreover,
this polynomial is in the $ab$ position in the $k\times k$ array.
\end{lemma}
\begin{proof}
This is an easy consequence of the discussion in
Subsection~\ref{subsec1:sep3:2010}.
\end{proof}

Lemma~\ref{lem1:aug30:2010}  suggests a very simple algorithm for positioning
the polynomials in our family.
If the commutative variables are partitioned and positioned in the matrix
$X$ and if $\varphi(n,0) \ne 0$ for some
$n \ge 2$, we simply run through the terms $\varphi(n,0)x_i^{n-1}x_j$ for
each $x_i$ in the $ii$ position
in $X$ and $x_j$ in the $ij$ position in $X$.  The polynomial
that contains this monomial will be in the $ij$ position in the array of
polynomials.  We will refer to this procedure
as \df{Algorithm PosPol}.  The next proposition summarizes the
conditions under which it is applicable.

\begin{proposition}[Algorithm PosPol]\label{cor1:aug30:2010}
Let $p_1,\cdots,p_{k^2}$ be a homogeneous family of polynomials with an
nc representation $p(X,Y)$ of degree $d>1$.
Then Algorithm PosPol will successfully position the polynomials
$p_1,\cdots,p_{k^2}$
if and only if  $p \in NC_{(\ref{eq2:31jan12})}$.
\end{proposition}

\begin{proof}
Algorithm PosPol will position the polynomials if and only if the $2k^2$
variables are partitioned, the variables in
$X$ are positioned
and $\varphi(n,0) \ne 0$.
Algorithm ParPosX will position the variables in $X$ if and only if
$p \in NC_{(\ref{eq2:31jan12})}$.
 \end{proof}
\medskip

\subsection{Positioning $Y$ given the position of the polynomials: Algorithm
PosY }

Given a family $\cP$ of polynomials $p_1,\cdots,p_{k^2}$ in the variables
$x_1,\cdots, x_{2k^2}$ with an nc
 representation $p(X,Y)$, we have to this point developed algorithms to
partition the variables between $X$ and $Y$, position the variables
in $X$ and position the polynomials $p_1,\cdots, p_{k^2}$.  The next step
is to  develop an algorithm that positions the commutative
variables in $Y$.

Suppose that the commutative variables $x_1,\cdots,x_{2k^2}$ have been
partitioned between $X$ and $Y$ and positioned
in $X$.  Furthermore, suppose that the polynomials in the family $\cP$ have
been positioned and that
$\varphi(s,t) \ne 0$ for some $s \ge 0$ and $t\ge 1$.
Let $p_{ij}$ denote the polynomial in the $ij$ position in $p(X,Y)$.
Let $x_1,\cdots, x_{k^2}$
be the variables contained in $X$ and let $x_1,\cdots,x_{k}$ be the diagonal
variables of $X$.
Set $x_1 = \cdots =x_{k} =1$ and $x_{k+1} = \cdots = x_{k^2} = 0$.  On the
nc
level this is equivalent to setting $X = I_k$.   Then consider the polynomials
$\what{p}_{ij} = p_{ij}(1,\cdots,1,0,\cdots,0,x_{k^2+1},\cdots,x_{2k^2})$ in
$k^2$ commuting variables.
If $\cP$ has an nc representation $p(X,Y)$, this collection of
polynomials will have an nc representation $p(I,Y) = p(Y)$ that contains the
monomial
$\varphi(s,t)Y^t \ne 0$.  Therefore, each diagonal polynomial
$\what{p}_{ii}$ will contain a monomial of the
form $\varphi(s,t)x_{j_i}^t$, where $j_i \ge k^2+1$.
This implies that $x_{j_i}$ must be the $ii$ entry of $Y$.
By repeating this argument for each $i$ we can position the other diagonal
elements of $Y$.  To position the remaining variables,
observe that a monomial of the form
$\varphi(s,t)x_{j_i}^{t-1}x_u$ will appear in each $\what{p}_{ij}$
where $u \ge k^2+1$.  Given that $x_{j_i}$ lies in the $ii$ position
of $Y$ and $\what{p}_{ij}$ is in the $ij$ position in $p(Y)$, it follows that
$x_u$ is in the $ij$ position in $Y$.  By repeating
this argument for each $ij$ we can position the other non diagonal entries of
$Y$.

We will refer to the process described above as \df{Algorithm PosY}.
We summarize the conditions  under which
it will successfully position the variables in $Y$ in the following
proposition.

\begin{proposition}\label{prop1:30nov11}[Algorithm PosY]
Let $p_1,\cdots,p_{k^2}$ be a family of polynomials with an nc representation
$p(X,Y)$.  Then Algorithm PosY will successfully position the commutative
variables in $Y$
if and only if $p \in NC_{(\ref{eq2:31jan12})}$ and there exists a pair of
integers
$s \ge 0$ and $t \ge 1$
such that $\varphi(s,t) \ne 0$.
\end{proposition}
\begin{proof}
This follows from the preceding discussion and the fact that PosPol will
position the polynomials if and only if $p \in NC_{(\ref{eq2:31jan12})}$.
\end{proof}

\medskip

\subsection{Uniqueness results for one-letter algorithms}
\label{sec:variations}

In this section we investigate the possible variations in the matrices $X$
and $Y$ that are obtained by Alg.1 and Alg.2. We shall
assume that the algorithms are
applied to the homogeneous components of degree $n$ in the family of
polynomials  $\cP$, where $n\ge 2$ and is also the lowest degree of single
letter monomials in $\cP$.

If $\PP$ denotes an array of the given $k^2$ polynomials (as on the right
hand side of (\ref{eq:oct18h11})) that admits the
nc representation
\begin{equation}
\label{eq:dec24c13}
\PP=p(X,Y)=\sum_{\alpha,\beta}c_{\alpha,\beta}m_{\alpha.\beta}(X,Y), \quad
\textrm{where $|\alpha|+|\beta|=n$ in the sum}
\end{equation}
and $\Pi$ is a $k\times k$ permutation matrix, then
\begin{equation}
\label{eq:dec24a13}
\Pi^T\PP\Pi=\sum_{\alpha,\beta}c_{\alpha,\beta}m_{\alpha.\beta}(\Pi^TX\Pi,
\Pi^TY\Pi).
\end{equation}
Moreover, since
$$
(X^{\alpha_1}Y^{\beta_1}\cdots X^{\alpha_r}Y^{\beta_r})^T=
(Y^T)^{\beta_r}(X^T)^{\alpha_r}\cdots (Y^T)^{\beta_1}(X^T)^{\alpha_1},
$$
it is readily seen that
\begin{equation}
\label{eq:dec24b13}
\PP^T=\sum_{\alpha,\beta}c_{\alpha,\beta}m_{\beta^\prime,\alpha^\prime}
(Y^T,X^T),
\end{equation}
where
$$
\beta=(\beta_1,\ldots,\beta_r)\Longrightarrow
\beta^\prime=(\beta_r,\ldots,\beta_1)\ \text{and}\
\alpha=(\alpha_1,\ldots,\alpha_r)\Longrightarrow
\alpha^\prime=(\alpha_r,\ldots,\alpha_1).
$$
Below we shall show that, aside from a possible interchange of $X$ and $Y$,
formulas (\ref{eq:dec24a13}) and  (\ref{eq:dec24b13}) account for the only
possible variation in the matrices $X$ and $Y$ that are
generated by our algorithms. They correspond to the fact that the
ParPosX algorithm allows for the diagonal variables to be placed
in any order along the diagonal of
$X$ and the ambiguity in the next step of that algorithm, in which two
variables $x_u$ and
$x_v$ are arbitrarily assigned to be either the $st$ entry or the $ts$ entry
of $X$ (for some non ambiguous choice of $s$ and $t$ with $s\ne t$).
This is in fact the only freedom that one has in positioning the variables
within $X$, i.e., once
these $k+2$ variables are
allocated, the positions of the remaining variables in $X$ are
fully determined. This is substantiated by the next theorem.

Recall that  the first matrix constructed
by either Alg.1 or Alg.2 is always designated $X$
(or $\wtilde{X}$).

\begin{theorem}\label{thm1:14dec11}
Suppose that $\cP$ is a homogeneous family of polynomials
$p_1, \cdots, p_{k^2}$ of degree $n$ in the
commutative variables $x_1,\cdots, x_{2k^2}$ that admits an nc
representation.
Then if $X,Y$ and $\wtilde{X}, \wtilde{Y}$  are generated by two different applications
of Alg.1 or Alg.2, they are
 permutation equivalent
(as defined in \eqref{eq:feb23a12}).

\end{theorem}

\begin{proof}
Suppose first that the given family $\cP$ of $k^2$ polynomials contains
exactly $k$ one letter monomials $ax_{i_1}^n,\ldots,ax_{i_k}^n$ (all with the
same coefficient $a\in\RR\setminus\{0\}$) and let $X$ and $\wtilde{X}$ denote
the matrices that are determined by successive applications of either Alg.1
or Alg.2. Then the set of diagonal entries of $X$ (without regard to to their
positions on the diagonal) will coincide with the diagonal entries of
$\wtilde{X}$.
Moreover, since D2Pa1 and D2Pa2 depend only upon the diagonal
variables and not upon how they are positioned along the diagonal, the set of
$k^2$ variables in $X$ coincides with the set of $k^2$ variables in
$\wtilde{X}$.

Let
$$
L(x_{i_s})=\{x_u\in X\setminus\{x_{i_s}\}: \textrm{$ax_{i_s}^{n-1}x_u$
appears in one of the polynomials in $\cP$}\}.
$$
There are $2k-2$ distinct variables $x_u$ in $L(x_{i_s})$. Moreover,
the
selection of these terms is based totally on $\cP$ and not upon the position of
 $x_{i_s}$ in $X$. Therefore, if $s\ne t$, then the two variables in the
intersection
$$
L(x_{i_s})\cap L(x_{i_t})=\{x_u,x_v\},
$$
are also independent of the position
of  $x_{i_s}$  and $x_{i_t}$.

If  $x_{i_s}$ is the $ss$ entry of $X$ and $x_{i_t}$ is the
$tt$ entry of $X$, then either $x_u$ is the $st$ entry and $x_v$ is the $ts$
entry, or vice versa. It is impossible to decide. However, once one of these
two possibilities is chosen, then the position of all the remaining
$k^2-k-2$ variables in $X$ are determined by the algorithm.

The algorithm ParPosX accomplishes this by inspecting terms of the
form $x_{i_j}^{n-2}x_tx_m$ in $\cP$, and so
once one variable is placed, the algorithm is able to partition the set
$L(x_{i_s})$ into the sets
$$
R(x_{i_s})=\{x_u\in X:\, x_u \ \textrm{is in the same row as $x_{i_s}$}\}
$$
and
$$
C(x_{i_s})=\{x_u\in X:\, x_u \ \textrm{is in the same column as $x_{i_s}$}\}
$$
for each integer $s$, $1 \le s \le k$, solely
by analyzing monomials contained in $\cP$. Consequently, these sets
do not depend on the position of the diagonal entry $x_{i_s}$.
Since  $\wtilde{X}$ has the same diagonal elements as $X$,
there exists a permutation $\pi$ of the integers $\{1,\ldots,k\}$ such that
if $\wtilde{x}_{ss}$ denotes the $ss$ entry in the matrix $\wtilde{X}$
obtained in a second application of either of the two algorithms, then
$\wtilde{x}_{ss}=x_{i_{\pi(s)}}$ and correspondingly, if
\begin{align}\label{eq1:14dec11}
\Pi = \left[ \begin{array}{c} {\bf e}^T_{\pi(1)}\\ \vdots\\{\bf e}^T_{\pi(k)}\end{array} \right]=\sum_{j=1}^k{\bf e}_j{\bf e}_{\pi(j)}^T, \quad
\text{where}~~{\bf e}_i ~~\text{denotes the $i$-th column of $I_k$,}
\end{align}
then clearly the diagonal entries of the matrices $X$ and
$\Pi^T\wtilde{X}\Pi$ will will be positioned in the same way
along the diagonal.

The fundamental observation is that
\begin{equation}
\label{eq:jan1a12}
\Pi^TE_{st} \Pi=\left(\sum_{i=1}^k \be_{\pi(i)}\be_i^T\right)\be_s\be_t^T
\left(\sum_{j=1}^k \be_j\be_{\pi(j)}^T\right)=
\be_{\pi(s)}\be_{\pi(t)}^T=E_{\pi(s),\pi(t)}.
\end{equation}

This accounts for the permutations. Transposition is a little more
complicated. The point is that after the diagonal variables in $X$ are
positioned, say $x_{i_s}$ is the $ss$ entry of $X$ for $s=1,\ldots,k$, as
above, and if $s\ne t$ and $L(x_{i_s})\cap L(x_{i_t})=\{x_u,x_v\}$, then
either $x_u$ is $st$ entry and $x_v$ is the $ts$ entry, or the other way
around. Thus, if
\begin{enumerate}
\item[\rm(1)] If
$x_u$ is assigned to the $st$ position of $X$, then the
polynomial containing terms of the form
$$
ax_{i_s}^{n-1}x_u+ax_ux_{i_t}^{n-1}+ax_{i_s}^{n-2}x_wx_z+\cdots
$$
must be placed in the $st$ position in the array $\PP$. But this means that
either $x_w$ is in the $sr$ position or the $rt$ position of $X$ for some $r$
other than $s$ or $t$, since the $ss$, $tt$ and $st$ positions are already
occupied.  (If it is in the $rt$ position, then $x_z$ will
be in the $sr$ position, so there is no loss of generality in assuming that
$x_w$ is in the $sr$ position of $X$.) Therefore, there is no loss of
generality in assuming that
$x_w\in L(x_{i_s})\cap L(x_{i_r})$. Consequently,
$$
\textrm{if $x_u$ is put in the $st$ position of $X$, $x_w$ will be in the
$sr$ position}.
$$

\vspace{2mm}
\item[\rm(2)] If $x_u$ is placed in the $ts$ position, then the polynomial
displayed above must be placed in the $ts$ position of the array $\PP$.
Consequently $x_w$ must be in either the $tr$ position or the $rs$ position.
But since  it is in $L(x_{i_s})\cap L(x_{i_r})$, the only viable option is
that it is in the $rs$ position of $X$:
$$
\textrm{if $x_u$ is put in the $ts$ position of $X$, $x_w$ will be in the
$rs$ position}.
$$
\end{enumerate}
Thus, transposition of the two variables in the first step after the
diagonals are fixed, moves $X$ to $X^T$.

Now suppose that there are two sets of single letter monomials
$ax_{i_1}^n,\cdots, ax_{i_k}^n$ and $bx_{j_1}^n, \cdots, bx_{j_k}^n$, where
the monomials $ax_{i_k}$ and $bx_{j_k}$ must occur in the same polynomial in
$p$.  In this case it is possible for $X$ to have diagonal variables $x_{i_1}, \cdots, x_{i_k}$
and for $\wtilde{X}$ to have diagonal variables $x_{j_1}, \cdots, x_{j_k}$.  Therefore,
we must show that if this happens that $X$ is pt equivalent to $\wtilde{Y}$
and that $\wtilde{X}$ is pt equivalent to $Y$.

If $x_{i_1},\cdots, x_{i_k}$ are the diagonal variables for $X$,
then $x_{j_1}, \cdots, x_{j_k}$ must be the diagonal variables of $Y$ and
$x_{i_s}$ and $x_{j_s}$ must occur in the same diagonal position of $X$ and $Y$.
Similarly, if $x_{j_1},\cdots, x_{j_k}$ are the diagonal variables of $\wtilde{X}$ and
$x_{i_1},\cdots, x_{i_k}$ are the diagonal variables of $\wtilde{Y}$, then
$x_{j_s}$ and $x_{i_s}$ must occur in the same diagonal position in
$\wtilde{X}$ and $\wtilde{Y}$.  As in \eqref{eq:jan1a12}, let $\Pi$ be
the permutation matrix such that the diagonal entries of $\Pi^TX\Pi$ and $\wtilde{Y}$
are positioned in the same way along the diagonal.  It follows that the diagonal entries
of $\Pi^TY\Pi$ and $\wtilde{X}$ are positioned in the same way as well.

When Alg.1 or Alg.2 determines the sets $L(x_{i_s})$
for $X$ and $L(x_{j_s})$ for $\wtilde{X}$ for $1\le s \le k$, it does so by
considering terms of the form $ax_{i_s}^{n-1}x_u$ and $bx_{j_s}^{n-1}x_v$.  One
readily sees that for each term of the form $ax_{i_s}^{n-1}x_u$ that occurs in a polynomial
in $\cP$, there is a corresponding term $bx_{j_s}^{n-1}x_v$ that occurs in the same
polynomial.  Therefore the fundamental observation is that
for each term $ x_u \in L(x_{i_s})$, there is a corresponding term
$x_v \in L(x_{j_s})$ and furthermore, if $x_u$ is placed in the $st$ position in $X$,
Algorithm PolyPos will place $x_v$ in the $st$ position in $Y$.  Similarly, if $x_v$
is placed in the $st$ position in $\wtilde{X}$, this correspondence ensures
that Algorithm PolyPos will place $x_u$ in the $st$ position in $\wtilde{Y}$.  Thus
the nc variables $X$ and $Y$ are pt equivalent to $\wtilde{X}$ and $\wtilde{Y}$.

\end{proof}

Now that we have shown that there is a strong relationship between any two
pairs of matrices constructed by our algorithms, we want to exploit this relationship
to construct nc polynomials for matrix pairs determined by our algorithms.  In particular, if $X,Y$ and $\wtilde{X}, \wtilde{Y}$
are two pairs of matrices determined by our algorithms for a given
family $\cP$ and $p(X,Y)$ is an nc representation of $\cP$, we would like to conclude
that there exists an nc polynomial $\wtilde{p}$ such that $\wtilde{p}(\wtilde{X},\wtilde{Y})$
is also an nc representation of $\cP$.  We shall see in Lemma~\ref{lem1:17jan12} below that
this is true.

Lemma~\ref{lem1:17jan12} supplements Theorem \ref{cor:uniqIntro} and is formulated
in terms of a pair of auxiliary nc polynomials that are expressed in terms
of the notation introduced in \eqref{eq:dec24b13}:

\begin{align}\label{eq6:30jan12}
p_t(X,Y) = \sum_{\alpha,\beta} c_{\alpha, \beta}m_{\beta', \alpha'}(Y,X)
\quad \text{and} \quad \overline{p}(X,Y) = p(Y,X).
\end{align}
Then it is clear that
$$
p_t(X^T,Y^T) = p(X,Y)^T.
$$

\begin{lemma}\label{lem1:17jan12}
Suppose that $p_1, \cdots, p_{k^2}$ is a collection of polynomials $\cP$ with an nc representation $p(X,Y)$
satisfying the conditions in Theorem~\ref{thm:nov25a11}.  If DiagPar1 or DiagPar2, ParPosX and PosY
generate a pair of matrices $\wtilde{X}$ and
$\wtilde{Y}$, then there exists a nc polynomial $\wtilde{p}$ such that $\wtilde{p}(\wtilde{X},\wtilde{Y})$
is an nc representation of $\cP$.
\end{lemma}

\begin{proof}
If $p$ satisfies the conditions in Theorem~\ref{thm:nov25a11}, then the
algorithms DiagPar1 or DiagPar2, ParPosX, PolyPos and PosY can be applied
to construct $X$ and $Y$.  Since $p(X,Y)$ is an nc representation of $\cP$,
by Theorem~\ref{thm1:14dec11} and equations
\eqref{eq:dec24c13}, \eqref{eq:dec24b13}, and \eqref{eq6:30jan12} the following must hold:

\begin{eqnarray}\label{eq2:30jan12}
\begin{tabular}{lll}
${\bf (1)}\quad X =\Pi^T\wtilde{X}\Pi, $&$Y = \Pi^T\wtilde{Y}\Pi,$ &$ \Rightarrow  p(\wtilde{X},\wtilde{Y})\quad \textit{is an nc rep.}$ \\
${\bf (2)}\quad X= \Pi^T\wtilde{X}^T\Pi,$&$Y=\Pi^T\wtilde{Y}^T\Pi, $ &$ \Rightarrow  p_t(\wtilde{X},\wtilde{Y})\quad \textit{is an nc rep.}$  \\
${\bf (3)}\quad X= \Pi^T\wtilde{Y}\Pi,$ &$Y=\Pi^T\wtilde{X}\Pi, $ & $\Rightarrow \overline{p}(\wtilde{X},\wtilde{Y})\quad \textit{is an nc rep.}$\\
${\bf (4)}\quad X= \Pi^T\wtilde{Y}^T\Pi, $&$Y= \Pi^T\wtilde{X}^T\Pi $ &$  \Rightarrow \overline{p}_t(\wtilde{X},\wtilde{Y})\quad \textit{is an nc rep.}$ \\
\end{tabular}
\end{eqnarray}
\end{proof}

Using the uniqueness results developed in this section, we can
now prove Theorem~\ref{cor:uniqIntro}.

\subsubsection{Proof of Theorem~\ref{cor:uniqIntro}}
\begin{proof}
Given that $p(X,Y)$ and $\wtilde{p}(\wtilde{X},\wtilde{Y})$ satisfy the
conditions in Theorem~\ref{thm:nov25a11}, Algorithms ParPos1 or ParPos2,
ParPosX, PolyPos, and PosY can successfully determine the pairs $X, Y$ and
$\wtilde{X}, \wtilde{Y}$.
Therefore, we may apply Theorem~\ref{thm1:14dec11}
to obtain the desired result.
\end{proof}

\subsection{Determining $p(X,Y)$ given $X$, $Y$ and the positions of the
polynomials in $\cP$}
\label{subsec:linSolver}

Once the matrices $X$ and $Y$ and the positions of the
polynomials in $\cP$ are determined
by the previous algorithms, it remains only to find an nc representation
for $\cP$. This rests on the following elementary observation, which is
formulated in terms of the notation introduced in \eqref{eq:sep19c10}:

\begin{lemma}
\label{lem:monHomog}
The monomial
$m_{\alpha, \beta}(X,Y)$ is a
$k\times k$ array of polynomials in the variables $x_1,\ldots,x_{k^2}$
of degree $\vert\alpha\vert+\vert\beta\vert$.
\end{lemma}

\begin{proof}
This is immediate from the rules of matrix
multiplication.
\end{proof}

\subsubsection{Algorithm NcCoef}
Suppose that

\begin{align}\label{eq2:22jan12}
 \cP  = \{p_1, \cdots, p_{k^2}\} \quad \text{is a family of $k^2$ polynomials}
\end{align}
of degree $d$ or less
in the commutative variables $x_1,\cdots, x_{2k^2}$ and that these variables
are positioned in $X$ and $Y$
and the polynomials in $\cP$ are positioned in an array
\begin{align}\label{eq1:30oct12}
\mathbb{P} = \left[ \begin{array}{ccc} p_{\lambda(1)} & \cdots &p_{\lambda(k)} \\ \vdots & \ddots &\vdots \\ p_{\lambda(k^2-k+1)}&\cdots & p_{\lambda(k^2)} \end{array} \right]
\end{align}
as in (\ref{eq:oct18h11}).
Then,  since any nc
polynomial $p(X,Y)$ in the variables $X$ and $Y$ of degree $d$
can be expressed in terms of monomials $m_{\alpha,\beta}(X,Y)$ as
\begin{align}
\label{eq1:22dec11}
p(X,Y) = \sum_{|\alpha|+|\beta|\le d} c_{\alpha,\beta}~m_{\alpha,\beta}(X,Y),
\end{align}
the NcCoef algorithm reduces to solving the system of linear equations
\beq
\label{eq:PbbP}
p(X,Y)= \bbP,
\eeq
for the unknown coefficients  $c_{\alpha,\beta}$. This system of equations
has a solution if and only if there exists
an nc polynomial $p$ such that $p(X,Y)$ is an nc representation of
$\cP$. In view of Lemma \ref{lem:monHomog}, the addition of monomials
$m_{\alpha,\beta}(X,Y)$
of degree higher than $d$ in \eqref{eq1:22dec11}
does not effect the solvability of \eqref{eq:PbbP}.

\begin{remark}
Algorithm NcCoef determines whether or not an nc representation exists
for a given pair of nc variables $X$ and $Y$.
Klep and Vinnikov \cite{KVprep} have been working on various elegant abstract characterizations of those sets of $X,Y, \text{and}$ ${\mathbb P}$  which admit an nc  representation.
However, their work \cite{KVprep} does not address the issue of implementing tests for these
characterizations.
\end{remark}

The following example illustrates how Algorithm NcCoef
can be applied to determine $p(X,Y)$ for a given family $\cP$.

\begin{example}
\label{ex:dec29a11}
Let
\begin{align}
&p_1 =4x_1^2+4x_2x_3+2x_1x_5+6x_5^2+x_3x_6+6x_6x_7\nonumber \\
&p_2 =4x_1x_2+4x_2x_4+x_2x_5+x_1x_6+x_4x_6+6x_5x_6+x_2x_8+6x_6x_8 \nonumber \\
&p_3 =4x_1x_3+4x_3x_4+x_3x_5+x_1x_7+x_4x_7+6x_5x_7+x_3x_8+6x_7x_8\nonumber\\
&p_4 = 4x_2x_3+4x_4^2+x_3x_6+x_2x_7+6x_6x_7+2x_4x_8+6x_8^2. \nonumber
\end{align}
and  suppose that
\begin{align}
X =\left(\begin{array}{cc}x_1&x_2\\x_3&x_4\end{array}\right), \quad Y = \left(\begin{array}{cc}x_5&x_6\\x_7&x_8\end{array}\right)\quad\textrm{and}\quad
\mathbb{P} = \left(\begin{array}{cc}p_1&p_2\\p_3&p_4\end{array}\right).
 \end{align}
 \noindent
Our goal is to find an nc representation $p$ or to refute its
existence.
\end{example}
\bigskip

\noindent  {\bf Discussion} Since the given family of polynomials is homogeneous of
degree two, the first step of Algorithm NcCoef is to form the polynomial
$$
aX^2+bXY+cYX+dY^2 = p(X,Y),
$$
\noindent
and then set $p(X,Y) = \mathbb{P}$.
This yields four relations, one for each entry:
\begin{align}
&ax_1^2+ax_2x_3+bx_1x_5+cx_1x_5+dx_5^2+cx_3x_6+dx_6x_7=p_1\nonumber \\
&ax_1x_2+ax_2x_4+cx_2x_5+bx_1x_6+cx_4x_6+dx_5x_6+bx_2x_8+dx_6x_8=p_2
\nonumber \\
&ax_1x_3+ax_3x_4+bx_3x_5+cx_1x_7+bx_4x_7+dx_5x_7+cx_3x_8+dx_7x_8=p_3\nonumber\\
&ax_2x_3+ax_4^2+bx_3x_6+cx_2x_7+dx_6x_7+bx_4x_8++cx_4x_8+dx_8^2=p_4. \nonumber
\end{align}
Upon matching the coefficients of the left hand side of the first row in the
preceding array with those of $p_1$,  we readily obtain
the list of linear equations
$$
a=4, \ b+c = 2, \ c=1 \ \textrm{and}\ d=6,
$$
the unknowns must satisfy; hence  $a=4, b=1, c=1, d=6$.
It is then easily checked that this choice of coefficients works for the
remaining  three rows of the array.
Therefore, for this particular choice of $X$ and $Y$ and positioning of the
polynomials
  $p_1,p_2,p_3, p_4$,
$$
p(X,Y)=4X^2+XY+ YX+6Y^2.
$$
\qed

\subsubsection{Homogeneous sorting and Implementation of NcCoef}\label{sec1:30oct12}

In order to implement the NcCoef algorithm, it is convenient to first
sort each polynomial $p_j$ in $\cP$ as a sum
$$
p_j=\sum_{i=1}^d (p_j)_{[i]}(x_1,\ldots,x_{2k^2})
$$
of homogeneous polynomials
 $(p_j)_{[i]}$  of degree $i$ in which $(p_j)_{[i]}$ is the sum of terms
in $p_j$ that are homogeneous of degree $i$ and
is taken equal to zero if there are no such  terms.

Let
\begin{align}\label{eq3:22jan12}
\cP_i = \{(p_1)_{[i]},\ldots,(p_{k^2})_{[i]}\}
\end{align}
and, if at least one of the polynomials in $\cP_i$ is nonzero, try to find a
homogeneous nc polynomial representation
$p_i(X,Y)$ of degree $i$ for $\cP_i$. We shall refer to this procedure as
{\bf homogeneous sorting}.

An obvious consequence of Lemma \ref{lem:monHomog} is:


\begin{lemma}\label{prop1:16jan12}
A family $\cP$ of $k^2$ polynomials of degree $\le d$ in $2k^2$
commuting variables admits an nc representation $p(X,Y)$ if and only if
$\cP_i$ admits an nc representation $p_i(X,Y)$ with the same $X$ and $Y$
for each $1 \le i \le d$.
\end{lemma}

The primary advantage of homogeneous sorting
is apparent when implementing
Algorithms DiagPar1, DiagPar2, ParPosX, PolyPos, PosY and NcCoef.  One applies
these algorithms to the nonzero
family $\cP_i$ for each $i$ separately.
Typically, to save computational cost, choose $i$
as small as possible in
order to minimize computations.
Then, once $X$, $Y$ and the position of the polynomials are determined,
it is
easy to fill in the coefficients of the terms in $p_i(X,Y)$
by comparison with the array of terms
$\{(p_1)_{[i]},\ldots,(p_{k^2})_{[i]}\}$ for the remaining choices of
$i$, one degree at a time, just as in Example \ref{ex:dec29a11}.
It is important that the same $X$ and $Y$ are used for each choice of $i$.
Now we give a cautionary example.

\begin{example}
\label{ex:dec29b11}
If $p_1,\ldots,p_4$ are as in Example \ref{ex:dec29a11}, then the set
of polynomials
\begin{align}
&q_1=p_1+(x_1^2+x_2x_3)x_1+(x_1x_2+x_2x_5)x_3\\
&q_2=p_2+(x_1^2+x_2x_3)x_2+(x_1x_2+x_2x_5)x_5\\
&q_3=p_3+(x_3x_1+x_5x_3)x_1+(x_3x_2+x_5^2)x_3\\
&q_4=p_4+(x_3x_1+x_5x_3)x_2+(x_3x_2+x_5^2)x_5
\end{align}
does not admit an nc representation even though the terms of degree two
admit an nc representation and the terms of degree three admit an nc
representation.
\end{example}

\noindent  {\bf Discussion}
The terms of degree two are exactly the polynomials
$p_1,\ldots,p_4$ considered in in Example \ref{ex:dec29a11} and hence
either lead to the representation considered there, or to an equivalent
representation that corresponds to conjugation of the matrices $X$, $Y$ and
the polynomial array matrix $\PP$ by a $2\times 2$ permutation matrix $\Pi$
(to obtain $\Pi^TX\Pi$, $\Pi^TY\Pi$ and $\Pi^T\PP\Pi$ in place of $X$,
$Y$ and $\PP$), or transposition or to an interchange of $X$ and $Y$. But in
all these shufflings, the variables $\{x_1,x_2,x_3,x_4\}$ will belong to
one of
the matrices and the remaining variables $\{x_5,x_6,x_7,x_8\}$ will belong to
the other.

The polynomials $q_1,\ldots,q_4$, will not admit an nc
representation because the added terms come from
$$
\begin{bmatrix}x_1&x_2\\x_3&x_5\end{bmatrix}^3
$$
which involves a mixing of the variables in the matrices $X$ and $Y$ that are
obtained by analyzing $p_1,\dots,p_4$. \qed

\subsubsection{Effectiveness of NcCoef}
The final task in our analysis of Algorithm NcCoef is to show that
it will either be successful for all pairs $X,Y$ produced our algorithms or it will fail for all such
pairs.  The next proposition (\ref{prop1:23jan12}) characterizes the nc polynomials $p$ such that
$p(X,Y)$ is an nc representation of $\cP$ when $X$ and $Y$
are determined by Algorithms DiagPar1 or DiagPar2, ParPosX PolyPos and PosY.
Proposition~\ref{prop2:23jan12} then insures
that a given family $\cP$ will either have an nc representation
$p(X,Y)$ for all $X,Y$ produced by our algorithms or will have
no representation for any such $X$ and $Y$.

\begin{proposition}\label{prop1:23jan12}
Suppose that a given family $\cP$ of $k^2$ polynomials in $2k^2$ commuting
variables admits an nc representation $p(X,Y)$ and that the algorithms
DiagPar1 or DiagPar2, ParPosX PolyPos and PosY
produce the matrices $X$ and $Y$.  Then  $p$ belongs to the set  $\cW$
of nc polynomials that is specified in Theorem~\ref{thm:nov25a11}.
\end{proposition}
\begin{proof}
The coefficient conditions defining $\mathcal{W}$ are the exact
conditions required for the algorithms DiagPar1 or DiagPar2, ParPosX,
PolyPos and PosY to be successful.  Furthermore, these algorithms
will only be successful if these conditions are satisfied.  Therefore,
the assumption that our algorithms generate $X$ and $Y$ implies
that if $p$ exists, then $p \in \mathcal{W}$.
\end{proof}

\begin{proposition}\label{prop2:23jan12}
Suppose that $\cP$ is a family of $k^2$ polynomials in $2k^2$ commuting
variables that satisfies \eqref{eq2:22jan12}.
Let $X$, $Y$ and $\wtilde{X}$,$\wtilde{Y}$ be distinct pairs of matrices
determined by separate applications of DiagPar1 or DiagPar2, ParPosX, PolyPos
and PosY.
Then Algorithm NcCoef will successfully determine
an nc representation $p(X,Y)$ of $\cP$ if and only if it
successfully determines
an nc representation $\wtilde{p}(\wtilde{X},\wtilde{Y})$of $\cP$.
\end{proposition}
\begin{proof}
If the NcCoef algorithm produces an
nc representation $p(X,Y)$ of $\cP$, then Proposition~\ref{prop1:23jan12}
implies that $p$ belongs to the set
$\mathcal{W}$ specified in Theorem~\ref{thm:nov25a11}.
Then Lemma~\ref{lem1:17jan12} implies that
there exists an nc polynomial $\wtilde{p} \in \mathcal{W}$ such that
$\wtilde{p}(\wtilde{X},\wtilde{Y})$
is an nc representation of $\cP$.  The reverse direction
follows by a similar argument.
\end{proof}

\subsection{The Size and Cost of NcCoef}\label{sec2:30oct12}

We now determine the cost
required to implement NcCoef.
For now, we neglect the cost to form the linear systems associated with NcCoef
and only focus on the cost of solving these systems. 

For given arrangements $\sigma$ and $\lambda$ of the variables $x_1,\cdots,x_{2k^2}$ in $X$ and $Y$ and the polynomials $p_1,\cdots, p_{k^2}$ in $\cP$ we
apply algorithm NcCoef to obtain a system of equations 
in the undetermined variables $c_{\alpha,\beta}$ as 
in  \eqref{eq1:22dec11} and \eqref{eq:PbbP}.
By applying the method of homogeneous sorting as described in Section~\ref{sec1:30oct12}, we obtain $(d-1)$ systems of equations
$M_i$ in the unknowns $c_{\alpha\beta}$ formed by equating
\begin{align}\label{eq2:30oct12}
\mathbb{P}^i_{\lambda} = p^{i}_{\sigma}(X,Y),
\end{align}
where $\mathbb{P}^i_{\lambda}$ and
$p^i_{\sigma}(X,Y)$ are formed from \eqref{eq1:30oct12} 
and \eqref{eq1:22dec11}. 

To determine the cost of solving this linear system let
\begin{align}
&t_{ij} = \text{ number of monomials in $p^{[i]}_{\lambda(j)}$},
\end{align}
and set $\tau_i = \sum_{j=1}^{k^2}t_{ij}$.  The number of
noncommutative monomials in the variables $X$ and $Y$ of degree $i$ is $2^i$.  Therefore
the number of unknowns $c_{\alpha\beta}$ in our system corresponding to 
homogeneous terms of degree $i$ will be $2^i$.  It follows that
$$M_{i} \quad \text{is a}~~\tau_i \times 2^i \quad \text{system of equations}$$ 
in the unknowns $c_{\alpha\beta}$ satisfying $|\alpha+\beta| = i$. 

When  $\tau_i > 2^i$, $M_i$ will be overdetermined.
The cost of solving this system using an LU decomposition is 
\begin{align}
k^2\tau_i 4^i-\frac{8^i}{3} \quad \text{arithmetic operations 
(see \cite{GvL} \S 3.2.11)}
\end{align} 
We must solve such a system for each $2 \le i \le d$, 
so the total cost to solve the linear system 
when each  $k^2\tau_i > 2^i$ is
\begin{align}\label{eq3:2oct12}
\text{TotLinC}_{\sigma\lambda}\le \sum_{i=2}^d \left(\tau_i4^i-\frac{8^i}{3}\right) \quad \text{arithmetic operations}.
\end{align}

Similarly, when $\tau_i  \leq  2^i$ we get
\begin{align}\label{eq3:underDet}
\text{TotLinC}_{\sigma\lambda}\le\sum_{i=2}^d \frac{2^{3i+1}}{3} \quad \text{arithmetic operations (see \cite{GvL} \S 3.2.9)}.
\end{align}


\subsection{Final Results
for One letter Algorithms}\label{finalresultone}

To this point we have developed the following algorithms:

\medskip

\begin{itemize}
\item[{\bf DiagPar1:}] Partitions the commutative variables between the
matrices $X$ and $Y$ and works under the assumption
that $\varphi(n,0) \ne 0$ for some $n\ge 2$ and $n\varphi(n,0) \ne \\
\varphi(n-1,1)$.

\item[{\bf DiagPar2:}] Partitions the commutative variables between the matrices $X$ and $Y$ and works under the assumption
that $\varphi(n,0) \ne 0$ for some $n\ge 2$ and $\varphi(n,0) \ne\\
 \varphi(n-1,1;Y)$ or $\varphi(n,0) \ne \varphi(Y;n-1,1)$.

\item[{\bf ParPosX:}] Positions the variables in the matrix $X$ if $\varphi(n,0) \ne 0 $ for some $n\ge 2$ and the commutative
variables are partitioned.

\item[{\bf PosPol:}] Positions the polynomials in the family if $\varphi(n,0) \ne 0$ for some $n\ge 2$ and the commutative
variables are partitioned and positioned in $X$.

\item[{\bf PosY}] Positions the commutative variables in $Y$ if $\varphi(0,n) \ne 0$ for some $n\ge 1$, the variables
are partitioned and positioned in $X$ and the polynomials are positioned.

\item[{\bf NcCoef}]
Given $X,Y$ and the positioning of the
$p_j$ in a matrix, this algorithm determines whether or not an nc $p$
exists such that $p(X,Y)$ generates the matrix containing the $p_j$.
\end{itemize}

{\bf Alg.1} will refer to the sequence DiagPar1, ParPosX,  PosPol, PosY and
NcCoef.
\medskip

{\bf Alg.2} will refer to the sequence DiagPar2, ParPosX,  PosPol, PosY.
and NcCoef.
\bigskip

The next theorem summarizes
the applicability of these two algorithms.

\bigskip

Let $NC_{(\ref{eq:nov6a11})}$ denote the class of all nc
polynomials $p(X,Y)$ of degree $d>1$ for which there exists integers $s \ge0$ , $t \ge 1$
and $n\ge 2$  such that
\begin{equation}
\label{eq:nov6a11}
\left\{\begin{array}{ll}
\varphi(s,t) \ne 0, \quad& \varphi(n,0) \ne 0,\quad  \varphi(n,0)
\ne \varphi(0,n),  \\
\quad\quad\text{ and}& \quad n\varphi(n,0) \ne \varphi(n-1,1).
\end{array}\right.
\end{equation}

\begin{theorem}
\label{thm:nov6a11}
Let $\cP$ be a family of $k^2$ polynomials in $2k^2$ commuting
variables. Then Alg.1 yields an
nc representation $p(X,Y)$ of $\cP$
 if and only if the given family $\cP$ admits
an nc representation in $NC_{(\ref{eq:nov6a11})}$.
\end{theorem}
\begin{proof}
If Alg.1 determines an nc representation $p(X,Y)$, then
Proposition~\ref{prop1:23jan12} implies that $p \in NC_{(\ref{eq:nov6a11})}$.
Now suppose that $\cP$ admits an nc representation in
$NC_{(\ref{eq:nov6a11})}$.
Then Theorem~\ref{lem:diagToPartit}, Theorem~\ref{thm1:sep9:2010},
 Proposition~\ref{cor1:aug30:2010}, and Proposition~\ref{prop1:30nov11}
imply that Alg.1 will successfully determine a pair of nc variables $X$
and $Y$.  Moreover,
Lemma~\ref{lem1:17jan12}
implies that there exists an nc polynomial $p$ such that $p(X,Y)$ is an
nc representation of $\cP$.  But this implies that for this choice of
$X$ and $Y$ that Algorithm NcCoef will be successful. Furthermore,
Proposition~\ref{prop2:23jan12} implies that for any pair of nc variables
produced by Alg.1, algorithm NcCoef will be successful.
Therefore,
Alg.1 will successfully determine an nc representation of $\cP$.
\end{proof}

Let $NC_{(\ref{eq:nov6b11})}$ denote the class of all nc polynomials
$p(X,Y)$ of degree $d>1$ for which there exists integers $s \ge 0$, $t\ge 1$ and $n\ge 2$
such that
\begin{equation}
\label{eq:nov6b11}
\left\{\begin{array}{ll}\varphi(s,t) \ne 0, \quad \varphi(n,0) \ne 0,
&\varphi(n,0) \ne \varphi(0,n),\quad \text{and either } \\
\quad \varphi(n,0) \ne \varphi(n-1,1;Y),\quad&\text{or}\quad \varphi(n,0)
\ne \varphi(Y;n-1,1).
\end{array}\right.
\end{equation}

\begin{theorem}
\label{thm:nov6b11}
Let $\cP$ be a family polynomials in $2k^2$ commuting
variables. Then Alg.2 yields an
nc representation $p(X,Y)$ of $\cP$ if and only if the given family $\cP$ admits
an nc representation in $NC_{(\ref{eq:nov6b11})}$.
\end{theorem}

\begin{proof}
If Alg.2 determines an nc representation $p(X,Y)$, then
Proposition~\ref{prop1:23jan12} implies that $p \in NC_{(\ref{eq:nov6b11})}$.
Now suppose that $\cP$ admits an nc representation in $NC_{(\ref{eq:nov6b11})}$.
Then Theorem~\ref{lem1:sep8:2010}, Theorem~\ref{thm1:sep9:2010},
 Proposition~\ref{cor1:aug30:2010}, and Proposition~\ref{prop1:30nov11}
imply that Alg.2 will successfully determine $X$ and $Y$.  Furthermore,
Lemma~\ref{lem1:17jan12}
implies that there exists an nc polynomial $p$ such that $p(X,Y)$ is an
nc representation of $\cP$.  But this implies that for this choice of
$X$ and $Y$ that Algorithm NcCoef will be successful.  Furthermore,
Proposition~\ref{prop2:23jan12} implies that for any pair of nc variables
produced by Alg.2, algorithm NcCoef will be successful.  Therefore,
Alg.2 will successfully determine an nc representation of $\cP$.
\end{proof}

\begin{remark}
\label{rem:nov10a11}
There are analogues of Theorems \ref{thm:nov6a11} and \ref{thm:nov6b11} in
which $NC_{(\ref{eq:nov6a11})}$ is replaced by the class
$NC_{(\ref{eq:nov10a11})}$ of all nc polynomials of
degree $d>1$ for which there exists integers $s \ge 1$, $t \ge0$ and $n\ge 2$
such that
\begin{equation}
\label{eq:nov10a11}
\left\{\begin{array}{ll}
\varphi(s,t) \ne 0, \quad \varphi(0,n) \ne 0, &\quad \varphi(n,0)
\ne \varphi(0,n) \\
\quad\quad \text{ and} \quad& n\varphi(0,n) \ne \varphi(1,n-1).
\end{array}\right.
\end{equation}
and $NC_{(\ref{eq:nov6b11})}$ is replaced by the class
$NC_{(\ref{eq:nov10b11})}$ of all nc polynomials of
degree $d>1$ for which there exists integers $s \ge 1$, $t \ge 0$ and $n\ge 2$
such that
\begin{equation}
\label{eq:nov10b11}
\left\{\begin{array}{ll}
\varphi(s,t) \ne 0, \quad \varphi(0,n) \ne 0,& \quad \varphi(n,0) \ne
\varphi(0,n),\quad \text{and either } \\
\quad \varphi(0,n) \ne \varphi(1,n-1;Y),&\quad\text{or}\quad \varphi(0,n)
\ne \varphi(Y;1,n-1).
\end{array}
\right.
\end{equation}
\end{remark}

\section{Examples}
\label{sec:nov25d11}
This section is devoted to a number of examples to illustrate the algorithms
that were developed in earlier sections, as well as some variations thereof.

\begin{example}
\label{ex:aug5a10}
The set of $4$ polynomials
\begin{eqnarray*}
p_1&=&x_1x_5+5x_5^2+x_2x_7+5x_6x_7\\
p_2&=&x_1x_6+5x_5x_6+x_2x_8+5x_6x_8 \\
p_3&=&x_3x_5+x_4x_7+5x_5x_7+5x_7x_8\\
p_4&=&x_3x_6+5x_6x_7+x_4x_8+5x_8^2
\end{eqnarray*}
in the commutative variables $x_1,\ldots,x_8$ can be identified
with the entries in an nc polynomial $p(X,Y)$ for appropriate
choices of the $2\times 2$ matrices $X$ and $Y$.  The objective is to find
such an identification.
\end{example}

\noindent
{\bf Discussion}
Since these 4 polynomials are homogeneous of degree 2, it is reasonable to
look for an nc representation must be of the form
$$
p(X,Y)=aX^2+bXY+cYX+dY^2
$$
for some choice of $a,b,c,d\in\RR$. Moreover, since there are only two one
letter monomials in the given family of polynomials:
$$
5x_5^2\ \textrm{in}\ p_1\quad\textrm{and}\quad 5x_8^2\ \textrm{in}\ p_4,
$$
the corresponding variables must sit on the diagonal of either $X$ or $Y$.
We shall arbitrarily place $x_8$ in the $11$ position of $X$ and
$x_5$ in the $22$ position of $X$.  This in turn forces $p_1$ to be in the
$22$ position of $p(X,Y)$ and $p_4$ to be in the $11$ position of $p(X,Y)$
and forces $a=5$ and $d=0$. Furthermore,
 if $A=X$ and $B=Y$ with $x_5=x_8=\alpha$, one of the other variables $x_i=\beta$ and the remaining five variables equal to zero, then $AB=BA$ and
$$
p(A,B)=5A^2+(b+c)AB,
$$
i.e., in terms of the notation introduced in subsection \ref{subs:mdr} with
$c_{ij}=\varphi(i,j)$ for short, $c_{20}=5$, $c_{11}=b+c$ and $c_{02}=0$. Consequently, there are six possibilities:

\noindent
1)  If $x_i$ is the $12$ entry of $X$, then
$$
A=\begin{bmatrix}\alpha&\beta\\ 0&\alpha\end{bmatrix},\quad B
=\begin{bmatrix}0&0\\ 0&0
\end{bmatrix}\quad\textrm{and}\quad p(A,B)
=\begin{bmatrix}5\alpha^2&10\alpha\beta\\ 0&5\alpha^2\end{bmatrix}.
$$
2)  If $x_i$ is the $12$ entry of $Y$, then
$$
A=\begin{bmatrix}\alpha&0 \\ 0&\alpha\end{bmatrix},\quad B
=\begin{bmatrix} 0&\beta\\ 0&0
\end{bmatrix}\quad\textrm{and}\quad p(A,B)
=\begin{bmatrix}5\alpha^2&c_{11}\alpha\beta\\ 0&5\alpha^2\end{bmatrix}.
$$
3)  If $x_i$ is the $21$ entry of $X$, then
$$
A=\begin{bmatrix}\alpha&0 \\ \beta&\alpha\end{bmatrix},\quad B
=\begin{bmatrix} 0&0 \\ 0&0
\end{bmatrix}\quad\textrm{and}\quad p(A,B)
=\begin{bmatrix}5\alpha^2&0\\10\alpha\beta&5\alpha^2\end{bmatrix}.
$$
4)  If $x_i$ is the $21$ entry of $Y$, then
$$
A=\begin{bmatrix}\alpha&0 \\ 0&\alpha\end{bmatrix},\quad B
=\begin{bmatrix} 0&0\\ \beta&0
\end{bmatrix}\quad\textrm{and}\quad p(A,B)
=\begin{bmatrix}5\alpha^2&0\\ c_{11}\alpha\beta&5\alpha^2\end{bmatrix}.
$$
5)  If $x_i$ is the $11$ entry of $Y$, then
$$
A=\begin{bmatrix}\alpha&0 \\ 0&\alpha\end{bmatrix},\quad B
=\begin{bmatrix} \beta&0\\ 0&0
\end{bmatrix}\quad\textrm{and}\quad p(A,B)
=\begin{bmatrix}5\alpha^2+c_{11}\alpha\beta&0\\ 0&5\alpha^2\end{bmatrix}.
$$
6)  If $x_i$ is the $22$ entry of $Y$, then
$$
A=\begin{bmatrix}\alpha&0 \\ 0&\alpha\end{bmatrix},\quad B
=\begin{bmatrix} 0&0\\ 0&\beta
\end{bmatrix}\quad\textrm{and}\quad p(A,B)
=\begin{bmatrix}5\alpha^2&0\\ 0&5\alpha^2+c_{11}\alpha\beta\end{bmatrix}.
$$
Thus,  if:
\begin{enumerate}
\item[\rm(1)] $x_i$ is an off-diagonal entry of $X$, $2$ polynomials will equal
$5\alpha^2$, one will equal $10\alpha\beta$ and one will equal $0$.
\item[\rm(2)] $x_i$ is an off-diagonal entry of $Y$, $2$ polynomials will equal
$5\alpha^2$, one will equal $c_{11}\alpha\beta$ and one will equal $0$.
\item[\rm(3)] $x_i$ is a diagonal entry of $Y$, one polynomial will equal
$5\alpha^2$, one will equal $5\alpha^2+c_{11}\alpha\beta$ and two will equal $0$.
\end{enumerate}

Correspondingly, if $x_5=x_8=\alpha$ and $x_i=\beta$ and all the other variables are set equal to zero,  then the polynomials $p_1,\ldots,p_4$  assume the
 values shown in the following array
 $$
 \begin{bmatrix}
 {}&i=1&i=2&i=3&i=4&i=6&i=7\\
 p_1=&5\alpha^2+\alpha\beta&5\alpha^2&5\alpha^2&5\alpha^2&5\alpha^2&5\alpha^2\\
 p_2=&0&\alpha\beta&0&0&10\alpha\beta&0\\
 p_3=&0&0&\alpha\beta&0&0&10\alpha\beta\\
 p_4=&5\alpha^2&5\alpha^2&5\alpha^2&5\alpha^2+\alpha\beta&5\alpha^2&5\alpha^2
 \end{bmatrix}.
 $$
 Upon comparing the values of the polynomials $p_1,\ldots,p_4$
for the $6$ possible choices of
 $x_i=\beta$ with the possibilities (1)--(3) indicated just above, it is readily seen (from setting (3))
 that $x_1$ and $x_4$ are diagonal entries in $Y$ and $c_{11}=1$. (In fact since $p_1$ is in the $22$ position of $p(X,Y)$, $x_1$ must be in the $22$ position of $Y$, which forces $x_4$ to be
 in the $11$ position of $Y$.)   The variables $x_2$ and $x_3$ will be off-diagonal entries of
 $Y$ by setting (2); and hence the remaining variables $x_6$ and $x_7$ must be off-diagonal entries in $X$.  A more detailed analysis would serve to position these last $4$ variables in
$X$ and $Y$, after fixing one of them; see Remark \ref{rem:jul30a10}.
\bigskip

\begin{remark}
\label{rem:jul30a10} As $p_1=x_1x_5+\cdots,$ $p_1$ is in the $22$ position of $p(X,Y)$ and $x_5$ is in
the $22$ position of $X$,
it follows that $x_1$ is in the $22$ position of $Y$.
Similarly, since
$p_4=x_4x_8+\cdots,$ $p_4$ is in the $11$ position of $p(X,Y)$ and $x_8$ is in
the $11$ position of $X$,
it follows that $x_4$ is in the $11$ position of $Y$.
Moreover, since the term $5x_6x_7$ in $p_1$ (and $p_4$) can only come from $5X^2$, it follows that  $x_6$ and $x_7$ must belong to $X$.  Consequently, the remaining two variables,  $x_2$ and $x_3$, must belong to $Y$.

To go further,  assume that $x_6$ is in the $12$ position of $X$. Then $x_7$ must be in the $21$ position
of $X$ and the polynomial
$$
p_2=x_1x_6+\cdots\quad\textrm{is in the $12$ position of $p(X,Y)$}.
$$
Therefore, $x_2$  is in the $12$ position of $Y$ and the remaining variable $x_4$ is in the
$21$ position of $Y$.
\end{remark}

To better illustrate the algorithms, we shall return to the case where it is only known
that $x_j\in Y$ for $j=1,\ldots,4$, $x_j\in X$ for $j=5,\ldots,8$, $x_8$ is the $11$ position of $X$
and $x_5$ is in the $22$ position of $X$. Then
$L_1=\{x_6,x_7\}=L_2$ and the positions of these two variables
in $X$ are not uniquely determined.  We shall arbitrarily place $x_6$ in the
$12$ position of $X$. Then  $x_7$ must be in the $21$ position and
$$
5X^2=5\begin{bmatrix}x_8&x_6\\x_7&x_5\end{bmatrix}^2=
5\begin{bmatrix}x_8^2+x_6x_7&x_8x_6+x_6x_5\\ x_7x_8+x_5x_7&x_5^2+x_7x_6\end{bmatrix}.
$$
Thus, $p_2=5x_8x_6+5x_6x_5+\cdots$ must sit in the $12$ position of $p(X,Y)$. Therefore,
$x_2x_8$ is also in the $12$ position of $p(X,Y)$, which forces $x_2$ to be in the $12$ position
of  $Y$.  Therefore,
$$
X=\begin{bmatrix}x_8&x_6\\x_7&x_5\end{bmatrix}\quad\textrm{and}\quad Y=\begin{bmatrix}x_4&x_2\\x_3&x_1\end{bmatrix}.
$$
It is now readily checked that
$$
\begin{bmatrix}p_4&p_2\\p_3&p_1\end{bmatrix}=5X^2+XY.
$$
{\bf Note!}
$$
 \begin{bmatrix}p_4&p_2\\p_3&p_1\end{bmatrix}\ne 5X^2+YX.
$$

\begin{example}
\label{ex:aug5b10}
The polynomials
\begin{eqnarray*}
p_1&=&x_1^2+x_2x_3-x_3x_6+x_2x_7\\
p_2&=&x_1x_2+x_2x_4-x_2x_5+x_1x_6-x_4x_6+x_2x_8\\
p_3&=&x_1x_3+x_3x_4+x_3x_5-x_1x_7+x_4x_7-x_3x_8\\
p_4&=&x_2x_3+x_4^2+x_3x_6-x_2x_7
\end{eqnarray*}
admit an nc representation.
\end{example}

\noindent
{\bf Discussion}
Since these $4$ polynomials are homogeneous of degree $2$, it is reasonable to
look for an nc
representation must be of the form
$$
p(X,Y)=aX^2+bXY+cYX+dY^2
$$
for some choice of $a, b, c,d\in\RR$, just as in Example \ref{ex:aug5a10}.
Moreover, since there are only two one
letter monomials
$$
x_1^2\quad \textrm{in}\ p_1\quad\textrm{and}\quad x_4^2\quad  \textrm{in}\
p_4,
$$
the corresponding variables are placed on the diagonal of $X$. Thus, $a=1$
and $d=0$, and the substitutions $X=A$ and $Y=B$ in the candidate $p(X,Y)$
for the
nc representation yields the formula
$$
p(A,B)=c_{20}A^2+c_{11}AB+c_{02}B^2,
$$
with $c_{20}=1$, $c_{11}=b+c$ and $c_{02}=0$.

We shall arbitrarily place $x_4$ in the $11$ position of $X$ and $x_1$ in
the $22$ position of $X$.
This forces $p_4$  and $p_1$ to be in the $11$ and $22$ positions of
$p(X,Y)$, respectively. Then, setting $x_1=x_4=\alpha$
and one of the other variables $x_i=\beta$ leads to the following sets of
values:
$$
\begin{bmatrix}
{}&x_2=\beta&x_3=\beta&x_5=\beta&x_6=\beta&x_7=\beta&x_8=\beta\\
p_1&\alpha^2& \alpha^2& \alpha^2& \alpha^2& \alpha^2& \alpha^2\\
p_2&2\alpha\beta&0&0&0&0&0\\
p_3&0&2\alpha\beta&0&0&0&0\\
p_4&\alpha^2& \alpha^2& \alpha^2& \alpha^2& \alpha^2& \alpha^2
\end{bmatrix}.
$$
Since $c_{20}=1$ and $c_{02}=0$, the discussion in Section~\ref{sec1:11may12} predicts
\begin{eqnarray*}
&2&\  \textrm{polynomials equal to $\alpha^2$, one equal to $2\alpha\beta$ and one equal to
$0$, or}\\
& 2&\  \textrm{polynomials equal to $\alpha^2$, one equal to $c_{11}\alpha\beta$ and one equal to $0$, or}\\
&1&\  \textrm{polynomial equal to $\alpha^2$, one equal to $\alpha^2+c_{11}\alpha\beta$ and one equal to $0$},
\end{eqnarray*}
according to whether $x_i$ is an off-diagonal entry of  $X$, $x_i$ is an off-diagonal entry of $Y$ or $x_i$ is a diagonal entry of $Y$.
Thus, $c_{11}=0$, $x_2$ and $x_3$ must belong to $X$, and $x_5,\ldots,x_8$ belong to  $Y$.

Next, we shall arbitrarily place $x_2$ in the $12$ spot of $X$. Then $x_3$ must be in the $21$ spot of $X$, and $p_2$ is in the $12$ spot of $p(X,Y)$. Now, having placed the entries in $X$
and positioned the polynomials $p_1,\ldots,p_4$ in $p(X,Y)$,
it is readily seen that
\begin{eqnarray*}
&{}&x_1x_6\  \textrm{a term in $p_2$, $x_1$ in the $22$ position of $X$ $\Longrightarrow x_6$ in the $12$ position of $Y$}\\
&{}&x_2x_7\  \textrm{ a term in $p_1$, $x_2$ in the $12$ position of $X$ $\Longrightarrow x_7$ in the $21$ position of $Y$}.
\end{eqnarray*}
However, it  is not possible to fix the positions of $x_5$ and $x_8$ within $Y$ from the available information. Thus, to this point we know that
$$
X=\begin{bmatrix}x_4&x_2\\x_3&x_1\end{bmatrix}\quad\textrm{and that either}\quad
Y=\begin{bmatrix}x_5&x_6\\x_7&x_8\end{bmatrix}\quad\textrm{or}\quad
Y=\begin{bmatrix}x_8&x_6\\x_7&x_5\end{bmatrix},
$$
and, since $c_{20}=1$ and $c_{11}=c_{02}=0$, that there should be an  nc
polynomial
of the form
$$
p(X,Y)=X^2+a(XY-YX)
$$
for some $a\in\RR$ Since the coefficients of all the terms in $p_1,\ldots,p_4$ are $\pm 1$, it follows that $a=\pm 1$. It is then readily checked that
$$
\begin{bmatrix} p_4&p_2\\p_3&p_1\end{bmatrix}= X^2+X\begin{bmatrix}x_8&x_6\\x_7&x_5
\end{bmatrix}-\begin{bmatrix}x_8&x_6\\x_7&x_5\end{bmatrix}X,
$$
i.e., the second choice of $Y$ works, the first does not.

\begin{example}
\label{ex:aug5c09}
Let $k=5$ and $n=3$ and suppose that $x_i$ is in the $ii$ position
for $i=1,\ldots,5$ and that
$$
X=\begin{bmatrix} x_1&x_6&x_8  &x_{10}&x_{12}\\x_7&x_2&x_{14}&x_{16}&x_{18}\\
x_9&x_{15}&x_3&x_{20}&x_{22}\\
x_{11}&x_{17}&x_{21}&x_4&x_{24}\\
x_{13}&x_{19}&x_{23}&x_{25}&x_5\end{bmatrix}.
$$

\noindent
The objective is to analyze the family of $25$ commutative polynomials in $25$ variables
determined by $X^3$ to recreate $X$.  As in the previous examples, the family of polynomials
is the given data; we only include $X$ here because we will not write out the full family determined
by $X^3$ due to its prohibitive size.  The reader is encouraged to create the family generated by
$X^3$ using Mathematica and then follow along in the analysis to recreate $X$.

\end{example}

\noindent
{\bf Discussion}
If $X$ is of the given form, then $X^3$ generates a $5\times 5$ array of homogeneous polynomials of degree three, $p_1,\ldots,p_{25}$.  Five of these polynomials will each contain exactly one term of the form $x_i^3$. To simplify the exposition, we shall assume that the
variables are indexed  so that $x_i$ is in the $ii$ position of $X$ for $i=1,\ldots,5$. This in turn forces the polynomial that contains $x_i^3$ to be in the $ii$ position of $X^3$.
The rest of the construction is broken into steps.
\bigskip

{\bf 1.}  Find the entries in $L_i=R_i\cup C_i$ for $i=1,\ldots,5$ by
considering the terms $x_i^2x_j$  that appear in $X^3$, $i=1,\ldots,5$
and $j=6,\ldots,25$. For the given $X$ we will obtain:
\begin{eqnarray*}
L_1&=&\{x_6, x_7, x_8, x_9, x_{10}, x_{11}, x_{12}, x_{13}\},\\
L_2&=&\{x_6, x_7, x_{14}, x_{15}, x_{16}, x_{17}, x_{18}, x_{19}\},\\
L_3&=&\{x_8, x_9, x_{14}, x_{15}, x_{20}, x_{21}, x_{22}, x_{23}\},\\
L_4&=&\{x_{10}, x_{11}, x_{16}, x_{17},x_{20},x_{21},x_{24},x_{25}\}\\
L_5&=&\{x_{12}, x_{13}, x_{18},x_{19},x_{22},x_{23},x_{24},x_{25}\}
\end{eqnarray*}
\bigskip

\noindent
{\bf 2.} Position the entries in $L_1$.
\bigskip

Observe that $L_1\cap L_2=\{x_6,x_7\}$. This means that one of these
variables is in the $12$ position and the other is in the $21$ position.
We shall assume that $x_6$ is in the $12$ position and shall deduce the
positions of all the other variables from the $25$ polynomials $p_1,\ldots,p_{25}$,
corresponding to $X^3$. In particular, the assumption
that $x_6$ is in the $12$ position implies that the polynomial
$$
x_1^2x_6+x_6x_2^2+x_1x_6x_2+x_1x_8x_{15}+x_1x_{10}x_{17}
+x_1x_{12}x_{19}+\cdots
$$
is in the $12$ position of $X^3$.
But the remaining variables $x_1x_sx_t$ in that polynomial with off-diagonal variables $x_s$
and $x_t$, $s\ne t$, will be in the $12$ position of
$X^3$ if and only if
they are positioned in one of the ways indicated in the following array:
$$
\begin{array}{ccccccc}
x_s&13&14&15&32&42&52\\
x_t&32&42&52&13&14&15\end{array},
$$
which is to be read as:
$$
\textrm{either $x_s$ is in the $13$ position and $x_t$ is in the $32$
 position, or vice versa}
$$
The pair  $x_s=x_8$ and $x_t=x_{15}$ are subject to these constraints.
On the other hand, since $x_8\in L_1\cap L_3$ it can only be in either the
$13$ position or the $31$ position, whereas $x_{15}\in L_2\cap L_3$
and hence can only be in the $23$ position or the $32$ position. Thus,
the only viable solution to all these constraints is that
$$
x_8 \ \textrm{is in the $13$ position and $x_{15}$ is in the $32$ position}.
$$
Similarly, since $x_1x_{10}x_{17}$ and $x_1x_{12}x_{19}$ are in $p_{12}$,
whereas  $x_{10}\in L_1\cap L_4$,  $x_{17}\in L_2\cap L_4$,
$x_{12}\in L_1\cap L_5$
and $x_{19}\in L_2\cap L_5$, it is readily seen that
$$
x_{10} \ \textrm{is in the $14$ position and $x_{17}$ is in the
$42$ position}
$$
and
$$
x_{12} \ \textrm{is in the $15$ position and $x_{19}$ is in the
$52$ position}.
$$
Moreover, now that $x_6, x_8$, $x_{10}$ and $x_{12}$ are positioned, we consider the
following monomials in $p_{12}$:
\begin{eqnarray*}
x_6x_8x_9&{}& \Longrightarrow \textrm{$x_9$ is in the $31$ position of $X$}\\
x_6x_{10}x_{11}&{}& \Longrightarrow \textrm{$x_{11}$ is in the $41$ position of $X$}\\
x_6x_{12}x_{13}&{}& \Longrightarrow \textrm{$x_{13}$ is in the $51$ position of $X$}.
\end{eqnarray*}

\smallskip

\begin{remark}
\label{rem:aug4a10}
The preceding calculations exploit the entries in the polynomial in the $12$ position of $p(X,Y)$
to calculate terms in $R_1$, $C_1$ and $R_2$. A variant of this is to just fill in $R_1$ and then to
rely on successive steps to fill in $R_2,\ldots,R_5$, one row at a time. At the other extreme, it is
also possible to use all the entries in this polynomial  to fill in the whole matrix; see item 6, below.
\end{remark}

\noindent
{\bf 3.} Position the remaining entries in $L_2$ by the algorithm.
\smallskip

To this point we know that
\begin{equation}
\label{eq:oct24a11}
X=\begin{bmatrix}
x_1&x_6&x_8  &x_{10}&x_{12}\\x_7&x_2&\cdot&\cdot&\cdot\\
x_9&x_{15}&x_3&\cdot&\cdot\\
x_{11}&x_{17}&\cdot&x_4&\cdot\\
x_{13}&x_{19}&\cdot&\cdot&x_5\end{bmatrix},
\end{equation}
and it remains to fill in the dots. Since the polynomial
$$
x_7x_1^2+x_2^2x_7+x_2x_{14}x_9+x_2x_{16}x_{11}+x_2x_{18}x_{13}+\cdots
$$
sits in the $21$ position of $X^3$ and the positions of $x_9$, $x_{11}$ and $x_{13}$ are known, it follows that
$$
x_{14}\ \textrm{is in the $23$ position, $x_{16}$ is in
the $24$ and $x_{18}$ in
the $25$}.
$$
Moreover, since
$$
\{x_{14}, x_{15}\}=   L_2\cap L_3\quad \textrm{and $x_{14}$ is in the $23$
position}
$$
it follows that $x_{15}$ is in the $32$ position. Similarly, since
$$
\{x_{16}, x_{17}\}=   L_2\cap L_4\quad \textrm{and $x_{16}$ is in the $24$
position}
$$
it follows that $x_{17}$ is in the $42$ position. Much the same argument
based on  the position of $x_{18}$ and the observation that
$L_2\cap L_5=\{x_{18}, x_{19}\}$ shows that $x_{19}$ is in the $52$ position
of $X$. This completes the positioning of the variables in $L_2$.
\bigskip

\noindent
{\bf 4.} Position the remaining entries in $R_3$ using Algorithm ParPosX outlined in Section~\ref{subsec1:sep3:2010}
\smallskip

Since the polynomial
$$
x_3^2x_9+x_9x_1^2+x_3x_{15}x_7+x_3x_{20}x_{11}+x_3x_{22}x_{13}+\cdots
$$
\noindent
is in the $31$ position of $X^3$ and the positions of $x_3$, $x_7$, $x_{11}$ and    $x_{13}$
are known, it is readily checked that
$x_{20}$ is in the $34$ position
and $x_{22}$ is in the $35$ position.

\bigskip

\noindent
{\bf 5.} Position the remaining entries in $X$ using Algorithm ParPosX.
\smallskip

Since the  polynomial
$$
x_4^2x_{11}+x_{11}x_1^2+x_4x_{17}x_7+x_4x_{21}x_9+x_4x_{24}x_{13}+
\cdots
$$
is in the $41$ position of $X^3$ and the positions of $x_4$, $x_7$, $x_9$ and $x_{13}$
are known, it follows that
$x_{21}$ is in the $43$ position and $x_{24}$ is in the $45$ position.
Similarly, as the  polynomial
$$x_5^2x_{13}+x_{13}x_1^2+x_5x_{19}x_7+x_5x_{23}x_9+x_5x_{25}x_{11}
+\cdots
$$
is in the $51$ position of $X^3$, the positions of $x_5$, $x_7$, $x_9$ and $x_{11}$ serve to locate
$x_{19}$ in the $52$ position, $x_{23}$ in the $53$ position and $x_{25}$ in
the $54$ position. This completes the computation via Algorithm ParPosX.
\bigskip

\noindent
{\bf 6.}   Variations on the theme.
\smallskip

The preceding steps fill in the entries
of the partially specified matrix $X$ in (\ref{eq:oct24a11})
by studying the entries in $5$ of the
given $25$ polynomials by  sweeping along rows. It is also possible to fill in $X$ by using all the entries in the polynomial
$$
q=x_1^2x_6+x_6x_2^2+x_1x_6x_2+x_1x_8x_{15}+x_1x_{10}x_{17}
+x_1x_{12}x_{19}+\cdots
$$
that sits in the $12$ position of $X^3$:
\begin{eqnarray*}
x_6x_8x_9&{}&\textrm{a term in $q$, $x_6$ in the $12$ spot, $x_8$\  in the $13$ spot}\Longrightarrow x_9 \  \textrm{in the $31$ spot}\\
x_6x_{10}x_{11}&{}&\textrm{a term in $q$, $x_6$ in the $12$ spot, $x_{10}$ in the $14$ spot}\Longrightarrow x_{11}\  \textrm{in the $41$ spot}\\
x_6x_{12}x_{13}&{}&\textrm{a term in $q$, $x_6$ in the $12$ spot, $x_{12}$ in the $15$ spot}\Longrightarrow x_{13}\  \textrm{in the $51$ spot}\\
x_6x_{14}x_{15}&{}&\textrm{a term in $q$, $x_6$ in the $12$ spot, $x_{15}$ in the $32$ spot}\Longrightarrow x_{14}\  \textrm{in the $23$ spot}\\
x_6x_{16}x_{17}&{}&\textrm{a term in $q$, $x_6$ in the $12$ spot, $x_{17}$ in the $42$ spot}\Longrightarrow x_{16}\  \textrm{in the $24$ spot}\\
x_6x_{18}x_{19}&{}&\textrm{a term in $q$, $x_6$ in the $12$ spot, $x_{19}$ in the $52$ spot}\Longrightarrow x_{18}\  \textrm{in the $25$ spot}\\
x_8x_{20}x_{17}&{}&\textrm{a term in $q$, $x_8$ in the $13$ spot, $x_{17}$ in the $42$ spot}\Longrightarrow x_{20}\  \textrm{in the $34$ spot}\\
x_8x_{22}x_{19}&{}&\textrm{a term in $q$, $x_8$ in the $13$ spot, $x_{19}$ in the $52$ spot}\Longrightarrow x_{22}\  \textrm{in the $35$ spot}\\
x_{10}x_{15}x_{21}&{}&\textrm{a term in $q$, $x_{10}$ in the $14$ spot, $x_{15}$ in the $32$ spot}\Longrightarrow x_{21}\  \textrm{in the $43$ spot}\\
x_{12}x_{15}x_{23}&{}&\textrm{a term in $q$, $x_{12}$ in the $15$ spot, $x_{15}$ in the $32$ spot}\Longrightarrow x_{23}\  \textrm{in the $53$ spot}\\
x_{12}x_{17}x_{25}&{}&\textrm{a term in $q$, $x_{12}$ in the $15$ spot, $x_{17}$ in the $42$ spot}\Longrightarrow x_{25}\  \textrm{in the $54$ spot}\
\end{eqnarray*}
Therefore, the single remaining variable $x_{24}$ must be in the $45$ spot.

\begin{remark}
\label{rem:dec29a11}
If, in Step 2,  $x_6$ is placed in the $21$ position and $x_7$ is placed in
the $12$ position, then the algorithm will generate the transpose $X^T$
of the matrix $X$ that is specified in the statement of Example
\ref{ex:aug5c09}.
\end{remark}

\section{Cost of Alg.2 vs. a Brute Force Approach}
\label{sec:cost}

Now that we have developed our one-letter methods 
for determining an nc representation and
determined which families they will work for,
 we would like to get some  idea of their cost. 
 As a sample we get some rough cost estimates for Alg.2
 and omit estimates for Alg.1.  As Alg.1 and Alg.2 differ only in the partitioning algorithms
 used, we suspect that the 
 costs for Alg.1 will be similar to the costs for Alg.2.
The goal of determining these estimates is to compare the gains in efficiency
that our methods provide over the more direct, brute force approach
described in \S \ref{sec:introcost}.
As usual, we assume $\cP$ is a
family containing $k^2$ polynomials of degree $d$ or 
less in the commutative variables $x_1,\cdots,x_{2k^2}$.

\subsection{Brute Force}

Recall that in Section~\ref{sec2:30oct12} we determined
the cost of solving the linear systems determined by NcCoef.  
As the Brute Force Method merely applies
algorithm NcCoef for each arrangement 
of the commutative variables
and polynomials, our work is nearly finished.  
However, recall that in our cost analysis of 
algorithm NcCoef we omitted the cost to form
the systems.  For the sake of completeness, we briefly mention
how to form these linear systems using a function
like the one in Mathematica called CoefficientList[] to get an idea of 
the cost. 

The function CoefficientList[] in Mathematica takes as an input
a polynomial and a collection of commutative variables and
outputs the coefficients of the polynomial associated with
the given variables.
For example, the command
$$A[i] = \text{CoefficientList}[p_i,\{x_1, \cdots, x_{2k^2}\}]$$ 
creates a multidimensional array $A[i]$, where the entry $A[i](i_1,\cdots,i_{2k^2})$ corresponds to 
the coefficient of the monomial $x_1^{i_1}\cdots x_{2k^2}^{i_{2k^2}}$ in $p_i$.  The starting
point for our implementations of both the brute force approach and Alg.2 will be to call CoefficientList[ for each
polynomial in the given family $\cP$.  
We note that any implementation of our algorithms will require a function similar to CoefficientList[]
in order to of access the coefficients of the family $p_1,\cdots,p_{2k^2}$ in a systematic way.  

Let $M_i$ be the system formed by equating
\begin{align}\label{eq2:30oct12}
\mathbb{P}^i_{\lambda} = p^{i}_{\sigma}(X,Y),
\end{align}
where $\mathbb{P}^i_{\lambda}$ and
$p^i_{\sigma}(X,Y)$ are formed from \eqref{eq1:30oct12} and \eqref{eq1:22dec11} respectively,  by homogeneous sorting. 
Then let $p_{\lambda(j)}$ be an element of the matrix $ \mathbb{P}_{\lambda}$, and let $q_j$ be an element of  
$p_{\sigma}(X,Y)$ with undetermined coefficients $c_{\alpha\beta}$.  Letting
$A[j] = \text{CoefficientList}[q_j,\{x_1,\cdots,x_{2k^2}\}]$ for each $1 \le j \le k^2$,
and $B[j] = \text{CoefficientList}[p_{\lambda(j)},\{x_1,\cdots,x_{2k^2}\}]$, we
obtain the system associated with Eq. \eqref{eq2:30oct12} by setting 
\begin{align}\label{eq2:19nov12}
A[j](i_1,\cdots, i_{2k^2}) = B[j](i_1,\cdots,i_{2k^2}) \quad (1 \le j \le k^2)
\end{align}
for $(i_1,\cdots, i_{2k^2})$ satisfying $\sum_{n=1}^{2k^2} i_n = i$.  

Excluding the cost of iterating through the terms in Eq. \eqref{eq2:19nov12}, the cost to form
the systems of equations associated with NcCoef involves $k^2$ calls to the function
CoefficientList[].  In addition to this, recall that the cost to solve the linear systems formed by NcCoef for
a given arrangement $\sigma$ of the variables and $\lambda$ 
of the polynomials satisfies
\begin{align}
\text{CostNcCoef}_{\sigma\lambda}~~ \le
\sum_{i=2}^d \left(\tau_i 4^i-\frac{8^i}{3}\right) ~~ 
\text{arithmetic operations},
\end{align}
provided $\tau_i >2^i$.  The formula is similar if $\tau \le 2^i$.

Recall that there are $(2k^2)!$ arrangements $\sigma$ 
of $x_1,\cdots,x_{2k^2}$ in $X$ and $Y$ and $(k^2)!$ 
arrangements $\lambda$ of
$p_1,\cdots,p_{k^2}$.  We note that we must call CoefficientList[]
once for each $p_j$ in the given family $\cP$ and once for each undetermined polynomial 
$q_j$ determined by $\sigma$.  Therefore if $\tau > 2^i$, the cost satisfies
\begin{align}\label{eq2:31oct12}
\text{CostBruteForce} \le  (2k^2)!(k^2)!\left( \sum_{i=2}^d 
\left(\tau_i 4^i-\frac{8^i}{3}\right)\right)~~ \text{arithmetic operations}\\
+ \left(\left(k^2+(2k^2)!\right)~~\text{Calls to CoefficientList[]}\right). \nonumber
\end{align}

\subsection{Cost of Algorithm 2}
In this section we roughly analyze the cost associated with each subroutine
required to execute Alg.2.  Then in Section~\ref{Alg2cost} we summarize these 
costs and  compare them to the brute force cost determined in Eq \eqref{eq2:31oct12}.

\subsubsection{Algorithm DiagA: Determining the Diagonal Entries}
\label{sec3:30oct12}
The first step in Alg.2 is to determine the 
commutative variables that occur in one-letter monomials of the form $ax_i^n$.
Here we give a rough outline of a procedure
based on Lemma~\ref{lem2:sep9:2010} that will do this, and 
then analyze the cost.  

As in the case of the brute force algorithm and algorithm NcCoef, we implement
this algorithm using the function CoefficientList[].  Again, we let
$$A[i] = \text{CoefficientList}[p_i,\{x_1, \cdots, x_{2k^2}\}].$$
Using the coefficient matrix $A[i]$ for each $1 \le i \le k^2$, 
we inspect the coefficients of the single letter monomials of degree $n$, 
where $2 \le n \le d$.
For example, the coefficients of the one letter monomials 
$ax_1^n$ and $bx_2^n$ in $p_i$ are given by
$$
A[i](n,0,\cdots,0) = a \quad \text{and} \quad A[i](0,n,0,\cdots,0) = b.
$$
\smallskip

 The core of  Algorithm DiagA is to: 
\medskip

 Fix a degree $2 \le n \le d$.  Then 
 iterate through all of the coefficients of one letter monomials 
 of degree $n$ in $\cP$. 
 That is, 
 iterate through the expressions
\begin{align}\label{eq1:3oct12}
&A[1](n,0\cdots,0), \ \cdots, A[1](0,\cdots,0,n)\\
&A[2](n,0,\cdots,0), \cdots, A[2](0,\cdots,0,n) \nonumber \\
&\quad \vdots \qquad \qquad   \qquad  \quad \qquad \qquad \vdots \nonumber\\
&A[k^2](n,0,\cdots,0), \cdots, A[k^2](0,\cdots,0,n) \nonumber 
\end{align} 
and verify that exactly $0, k$ or $2k$ of the above expressions 
are nonzero and 
that they assume at most two distinct, nonzero values. 
If this is satisfied, then store the nonzero coefficients  
and store the associated diagonal variable in one 
of two lists based on its coefficient.
If this is not satisfied then quit.

\subsubsection{Cost Estimate for DiagA}\label{sec5:30oct12} 

The cost of algorithm DiagA is easy to bound
 and can be expressed in terms of the parameters 
 $k$ and $d$ and the number of calls to CoefficientList[].  
Define an \df{equality check} to be a test if two given  numbers
are equal. 
 
 Algorithm DiagA begins with
$k^2$ calls to the function 
CoefficientList[], one for each polynomial $p_i$.  
Iterating through 
the $(2k^2)(k^2)$ coefficients listed in \eqref{eq1:3oct12} 
for a given degree $2 \le n \le d$  
and performing a small
number $C_A$ of \ec  s 
for each coefficient requires 
no more than $(d-1)(k^2)(2C_Ak^2) = 2C_A(d-1)k^4$ 
 \ec s. 
 Therefore, the total cost for applying DiagA is
\begin{align}\label{eq3:3oct12}
TC = \left(C_A(d-1)k^4 ~~ \text{\ec s}\right) 
+ \left(k^2 ~~ \text{Calls to CoefficientList[]}\right).
\end{align} 
Here  $C_A $ is certainly no greater than 5.

\subsubsection{Cost of DiagPar2}\label{sec4:30oct12}
Here we determine an upper bound for the cost of Algorithm DiagPar2.  
We assume that
DiagA has already been executed and that CoefficientList[] has been 
called for each polynomial
$p_i \in \cP$.   See \S \ref{sec3:30oct12} for a definition of these procedures.

Algorithm DiagPar2 utilizes the coefficients of monomials of the form $ax_i^{n-1}x_j$, where $x_i$ is a diagonal
variable and $x_i \ne x_j$, to partition the variables between $X$ and $Y$.  The bulk of the cost of this algorithm will lie in 
iterating through the coefficients of the two-letter monomials of this form.  
We fix a
degree $2 \le n \le d$, a polynomial $p_j$ and a diagonal element $x_{i}$ 
(which can be determined using DiagA)
and iterate through the coefficients of monomials in $p_j$ of the 
form $x_i^{n-1}x_l$, where $x_l$ is a non-diagonal element;  
therefore, we
iterate through $2k^2-2$ coefficients for each diagonal element. 
For each coefficient  we perform no more than $C_{Par2}$, 
a small number,  of equality checks
to ensure that the conditions
in Lemma~\ref{lem:sep19a10} are satisfied and to determine 
whether $x_l$ is in $X$ or $Y$.  

If at any point we find that these conditions are violated or
we cannot find an $n$ such that the two-letter monomials
 $ax_i^{n-1}x_j$ satisfy the conditions 
 consistent with \eqref{eq1:31jan12}, 
 we terminate the iteration and conclude
either that $\cP$ has no nc representation 
or that the algorithm is inconclusive. 
As we are only considering  $k$ diagonal variables, $k^2$ polynomials
and $d-1$ different degrees that need to be considered, we have that
\begin{align}\label{eq17oct12}
\text{TotCostDiagPar2} \le C_{Par2}(d-1)k^5 
\quad \text{\ec s}.
\end{align}
Here $C_{Par2} \leq 5$.
\medskip

\subsubsection{ Cost of ParPosX}\label{sec6:30oct12}  
We now determine a rough bound on the cost of Algorithm ParPosX.  Again, we assume Algorithm DiagA and DiagPar2 have been executed and
that CoefficientList[] has been called for each of the polynomials $p_i \in \cP$.   

Algorithm ParPosX positions the variables in $X$ by using lists $L_i$ defined in \eqref{eq1:31oct12} and three letter monomials of the form
$x_i^{n-2}x_lx_m$, where $x_i$ is a diagonal and $x_l \ne x_m$ are non-diagonal elements in $X$.  We observe that we can build the lists $L_i$ when we iterate through the two
letter monomials $ax_i^{n-1}x_j$ in DiagPar2, so other than the cost of DiagA, DiagPar1 or DiagPar2, 
the primary cost associated with this algorithm lies in iterating through the three-letter monomials.  For a given non-diagonal polynomial $p_j$ and diagonal variable $x_i$, we must iterate through the coefficients of 
$(k^2-k)(k^2-k-1)$ such monomials, and for each nonzero coefficient (of which there will be at most $k-1$),  iterate through a list of length $2k-2$ and
 perform a fixed number of equality checks which we
will designate by $C_{ParP}$.  
Therefore,
\begin{align}
\text{TotCostPar2Pos} =  &C_{ParP}(k^2-k)k [((k^2-k)(k^2-k-1))-k] +\\
&C_{ParP}(k^2-k)k(k-1)(2k-2)
\quad \text{\ec s}\nonumber \\
 &~~~~~~~~\qquad \qquad  \qquad \qquad \qquad \qquad \le 
 3C_{ParP}k^7 \quad \text{\ec s.} \nonumber
\end{align}
Here $C_{ParP}\leq 10$.
\medskip

\subsubsection{Cost of PosPol}\label{sec7:30oct12} As in the case of Algorithm DiagPar2, the cost of Algorithm PosPol is effectively reduced to the cost of iterating through two-letter
coefficients.  As in DiagPar2, we inspect two-letter monomials of the form $ax_i^{n-1}x_j$, where $x_i$ is a diagonal element of $X$ and
$x_j$ is a non-diagonal element of $X$.  For a fixed degree $2 \le n \le d$, polynomial $p_m$ and diagonal element $x_i$ of $X$, we iterate through the coefficients of
all possible two-letter monomials in $p_m$, which allows for $k^2-k$ possibilities.  Once we find a coefficient $a\ne 0$ corresponding to the monomial $ax_i^{n-1}x_j$,
we immediately determine the position of $p_m$ using the position of $x_i$ and $x_j$.  Therefore 
\begin{align}
\text{TotalCostPolyPos} \le C_{PosP} (d-1)k^5 \hspace{2mm} \text{\ec s, $~~(C_{PosP} \le 10)$}.
\end{align}
We note that this is probably a gross overestimate of the cost of PosPol and that one does not need to iterate through all of the two-letter monomials again.  It seems possible
that we could store additional data associated with the non-diagonal entries $x_j$ when we iterate through the two-letter algorithms in DiagPar2 in order
to position the polynomials.  
\medskip

\subsubsection{Cost of PosY}\label{sec8:30oct12} We now determine a bound for the cost of PosY for families $\cP \in NC_{\ref{eq2:31jan12}}$.  We again assume
that all necessary arguments to implement PosY have been executed so that the commutative variables are positioned in $X$ and the polynomials are positioned.  
Therefore we assume that half of the commutative variables have been positioned in $X$ and that the polynomials $p_1,\cdots,p_{k^2}$
have been assigned positions in a $k\times k$ matrix.
The goal is to
to position the remaining variables in $Y$

The first step 
is to determine the diagonal variables of $Y$. 
PosY uses the fact that $\varphi(s,t) \ne 0$ for some $s,t$ such that
$2 \le s+t = n \le d$ and $t>0$.  
Then we 
look at two letter monomials 
of the form $\varphi(s,t)x_i^sx_l^t$ in the diagonal polynomials
$p_j$,where $x_i$ is a diagonal variable of $X$.  
This will allow us to conclude that $x_l$ is a diagonal variable in $Y$ that must be located in the same position as $x_i$ is in $X$. 
 To implement this 
step, for a fixed degree $n$, diagonal polynomial $p_j$ and corresponding diagonal  variable $x_i$ in $X$, and $s,t$ satisfying $s+t =n $, we must iterate through $k^2-k$  coefficients corresponding to two letter monomials of the
form $x_i^{s}x_l^t$, where $x_l$ is in $Y$. 
Therefore
\begin{align}
\text{CostDiagY}  &\le C_{PosY} (d-1)(k)\binom{d}{2}(k^2-k)
~~\text{\ec s} \\ 
&\le C_{PosY}(d^3k^3)~~ \text{\ec s}, \nonumber
\end{align}    
where $C_{PosY} $  dominates the operation per step count.

Finally, to determine the position of the off-diagonal variables of $Y$, we iterate through three-letter monomials of the form $ax_i^{s}x_l^{t-1}x_v$, where
$x_i$ and $x_l$ are diagonal variables of $X$ and $Y$ respectively, and $x_v$ is a non-diagonal term of $Y$.  We observe that if such a term occurs in a polynomial $p_j$
that is in the $qr$-position of the polynomial matrix and $x_i$ and $x_l$ are in the $qq$ position of $X$ and $Y$ respectively, 
then $x_v$ must be in the $qr$-position of $Y$.  Again, the bulk of the cost lies in iterating through these terms.
For a fixed non-diagonal polynomial $p_j$ and fixed diagonal variables $x_i$ and $x_l$ in the $qq$-position in $X$ and $Y$, we consider the coefficients 
of $k^2-k$ such terms and fewer than  $C_{PosY}$ \ec s each .
Therefore, 
\begin{align}
\text{CostNonDiagY} &\le C_{PosY}  (k^2-k)(k)(k^2-k)~~\text{\ec s}\\
 &\le C_{PosY} (k^5)~~ \text{\ec s.} \nonumber
\end{align}   
Thus,
\begin{align}
\text{TotalCostPosY} &\le  C_{PosY} \left((d^3k^3) + (k^5)\right) \text{\ec s}
\end{align}
Here $C_{PosY} \leq 10$.

\subsubsection{Cost of Algorithm 2}\label{Alg2cost}

The following  table lists the subroutines that make up Alg.2, their cost 
in terms of the number of
operations and calls to CoefficientList[], and
the section in which the cost of the subroutines was determined.  
Recall that the parameter $\tau_i$ represents
the total number of commutative monomials of 
degree $2 \le i \le d$ in all polynomials in $\cP$, and $C_{Diag A},  C_{Par2}, 
 C_{ParP},  C_{PosP}, C_{PosY} $
are all constants bounded by 10.

\begin{align}\label{tbl1:22oct12}
\begin{tabular}{ |l| c |r|}
\hline
Algorithm & Operations & Calls to CoefficientList[]  \\ \hline
DiagA &$C_{Diag A}  \; dk^4$ &  $k^2$  \\ \hline
DiagPar2 & $C_{Par2} \; dk^5$ & 0\\ \hline
ParPosX   & $C_{ParP}  \;  k^7$   & 0  \\ \hline
PosPol   & $C_{PosP}  \;  dk^5$   & 0 \\ \hline
PosY        & $C_{PosY}  \;  (d^3k^3+k^5)$ & 0 \\ \hline
\end{tabular}
\end{align}
Combining this with the cost of applying algorithm 
NcCoef  described in  \S \ref{sec2:30oct12}, we obtain
\begin{align}\label{eq3:31oct12}
\text{CostAlg2}  &\leq    10 \left( k^7+(2d+1)k^5+dk^4+d^3k^3 \right) 
~~ \text{\ec s} \\
&+ \sum_{i=2}^d \left(\tau_i 4^i-\frac{8^i}{3} \right) ~~ \text{arithmetic operations} +   k^2 ~~\text{Calls to CoefficientList[]},\nonumber
\end{align}
provided $\tau_i > 2^i$ for each $i$.
When $\tau_i  \leq 2^i$ for each $i$, we get
\begin{align}\label{eq1:19nov12}
\text{CostAlg2}  &\leq    10 \left( k^7+(2d+1)k^5+dk^4+d^3k^3 \right) 
~~ \text{\ec s} \\
&+\sum_{i=2}^d \frac{2^{3i+1}}{3}~~ \text{ arithmetic operations} +   k^2 ~~\text{Calls to CoefficientList[]}. \nonumber
\end{align}

\subsubsection{Comparison of Costs}

A comparison of the bounds \eqref{eq2:31oct12},\eqref{eq1:19nov12} and \eqref{eq3:31oct12} 
shows the benefit of our Algorithms.  We first observe
that we need to use the function CoefficientList[] $(2k^2)!$ 
fewer times using Alg.2, which is a huge savings given that the
cost to use CoefficientList[] could potentially be very expensive.  
Even if we neglect the cost of this function, we see that for large $d$ and $k$, 
the cost to form and solve linear systems using
NcCoef dominates both the cost for the Brute Force Method and Alg.2.  
By exploiting the structure of the
polynomials in the given family $\cP$
and performing on the order of
$$
k^7 + 3dk^5+d^3k^3~~ \text{operations},
$$ 
we have effectively reduced the cost from 
solving $(2k^2)!(k^2)!$ such systems to a single system.  
This is a vast improvement.  Also, to rule out the existence of an 
nc representation using the Brute Force method we must 
{\it check all of these cases and verify that they fail.} 
Much to the contrary, Algorithm 2 is likely to determine non existence 
very early when applying it.

\section{Families that may not contain a term of the form $ax_i^n$ with $n>1$}
\label{sec:nov25e11}

Sections~\ref{sec:algPosit}, \ref{sec:singlepos} and
\ref{sec:cost} present algorithms for solving our nc representation problem
when at least one of the given polynomials in $\cP$ contains a term of the form
$ax_i^n$ with $a\in\RR\setminus\{0\}$ and $n>1$. 
This section treats $\cP$
which contain no one
letter monomials 
but which do contain   terms of the form 
$ax_i^sx_j^t$, $a\ne 0$.
Recall  Lemma~\ref{lem:aug24a10} and Lemma~\ref{lem:aug26a10} 
dealt with patterns
two letter monomials in an nc representable $\cP$ must obey. 

The first step in developing these {\bf two letter algorithms} is to determine
the diagonal elements given the existence of two letter monomials.
Lemmas~\ref{lem:aug16a10} and \ref{lem:aug16b10} present conditions under which
the presence of terms that are $\vtr$-equivalent to $x_i^sx_j^t$ allows us to
determine dyslexic diagonal pairs.  The next step is to
partition the $k$ dyslexic pairs
$\{x_{i_1},x_{j_1}\}, \cdots, \{x_{i_k},x_{j_k}\}$, i.e., to determine which
elements
are on the diagonal of $X$ and which are on the diagonal of $Y$.  An
application of Lemma~\ref{lem:aug16d10}
is usually sufficient to partition these diagonal pairs.  

Once the diagonal
variables are determined and partitioned, we can reduce the problem to one that
is manageable for the single variable algorithms by taking derivatives of the polynomials in $\cP$
with respect to the diagonal variables.  This process will be known as the SVR algorithm.  The next section
discusses this reduction.  

We shall not present a cost analysis of our two-letter methods in this section,
since one can see as they unfold that they are clearly far superior to Brute Force. 
For one thing the core of our  two-letter 
procedures are reductions to our single letter algorithms 
(such as Alg. 2 whose cost is vastly less than of Brute Force).

\subsection{Reduction to one letter algorithms: SVR Algorithm}

In this section we develop the SVR (Single Variable Reduction) Algorithm,
which can be used
to reduce a family of polynomials containing terms that are
$\vtr$-equivalent to $x_i^sx_j^t$ to a family of polynomials to which the
single variable algorithms  developed earlier apply.
The next lemma contains the key idea that underlies this algorithm.

\begin{lemma}
\label{diffcor}
If $p_1,\ldots,p_{k^2}$ is a family of polynomials in the $2k^2$
commuting variables $x_1,\ldots,x_{2k^2}$ that admits an nc
representation $p(X,Y)$ of degree $d\ge 2$ such that
$x_{i_r}$ is in
the $ab$ position of $X$
and
$x_{i_1},\ldots,x_{i_k}$ sit on the diagonal of $X$, then:
\begin{enumerate}
\item[\rm(1)]  The polynomials
$$
\frac{\partial p_1}{\partial x_{i_r}},\ldots, \frac{\partial p_{k^2}}
{\partial x_{i_r}}
$$
admit the nc representation
\begin{equation}
\label{eq:aug12a10}
\lim_{t \rightarrow 0} \frac{1}{t} (p(X+tE_{ab},Y) - p(X,Y)).
\end{equation}
\item[\rm(2)] The family of polynomials $q_m(x_1,\ldots,x_{k^2})$,
$m=1,\ldots,k^2$, defined by
the formula
$$
q_m(x_{1}, \cdots , x_{2k^2}) = \sum_{s=1}^{k} \frac{\partial}{\partial x_{i_s}}p_m(x_{1}, \cdots , x_{2k^2}),\quad  m=1,\ldots,k^2,
$$
admits the nc representation
\begin{equation}
\label{eq:aug12b10}
\lim_{t \rightarrow 0} \frac{1}{t} (p(X+tI_k,Y) - p(X,Y)).
\end{equation}
\end{enumerate}
\end{lemma}

\begin{proof}
The first assertion follows from the observation that
$$
\lim_{t \rightarrow 0} \frac{1}{t} (p(X+tE_{ab},Y) - p(X,Y))
$$
is equivalent to computing
$$
\lim_{t \rightarrow 0} \frac{p_m(x_{i_1}, \ldots, x_{i_{r-1}} ,x_{i_r}+t,
x_{i_{r+1}},\ldots,x_{i_{2k^2}})-p_m(x_{i_1}, \ldots, x_{i_{r-1}} ,x_{i_r},
x_{i_{r+1}},
\ldots,x_{i_{2k^2}})}{t}
$$
$$
= \frac{\partial}{\partial x_{i_r}}p_{m}(x_{i_1}, \ldots,x_{i_{2k^2}}).
$$
\bigskip

The second assertion follows from the first by a straightforward calculation.
\end{proof}

To ease future applications of Lemma \ref{diffcor} when $x_{i_1},\cdots,x_{i_k}$ are the diagonal elements of $X$, it is convenient to introduce the notation
\begin{equation}
\label{eq:aug31b10}
(T_Xp_m)(x_{1}, \cdots , x_{2k^2}) = \sum_{s=1}^{k} \frac{\partial}
{\partial x_{i_s}}p_m(x_{1}, \cdots , x_{2k^2}),\quad m=1,\ldots,k^2.
\end{equation}

\begin{remark}
\label{rem:aug12a10}
Repeated application of the formulas in Lemma \ref{diffcor} serves to reduce
nc expressions in $X$ and $Y$ to expressions in the single variable $Y$.
Thus for example, if
$p_1,\ldots,p_{k^2}$ is a family of polynomials corresponding to
$$
p(X,Y)=XY^2X^2,
$$
then the family of polynomials $T^3p_m$, $m=1,\ldots,k^2$,
corresponds to the polynomial
$$
3!\,p(I,Y)=3! \,Y^2.
$$
In this way it is possible to eliminate the dependence on the $k^2$ variables in $X$
by differentiating the given family of polynomials just with respect to the $k$ diagonal entries of $X$.
\end{remark}

Let $D_{X,I_k}p(X,Y)$ (resp., $D_{Y,I_k}p(X,Y)$) denote the directional derivative of $p(X,Y)$ with
respect to $X$(resp., $Y$) in the direction of the identity $I_k$.  More generally,
let $D^i_{X,I_k}p(X,Y)$ (resp., $D^i_{Y,I_k}p(X,Y)$) denote the $i$-th directional derivative
of $p(X,Y)$ with respect to $X$ (resp., $Y$) in the direction of $I_k$.
Then Lemma~\ref{diffcor} states that if the variables $x_{i_1},\cdots,x_{i_k}$
are the diagonal elements of $X$ (resp., $Y$), then $D_{X,I_k}p(X,Y)$ (resp., $ D_{Y,I_k}p(X,Y)$)
 is a representation of
the family
$$
g_n = T_Xp_n ~~~(\text{resp.}~ g_n = T_Yp_n),
$$
where $T_X$ (resp., $T_Y$) is the operator defined in (\ref{eq:aug31b10}).

\begin{remark}\label{rem1:aug31:2010}By repeated application of Lemma~\ref{diffcor}, we have that $D^i_{X,I_k}p(X,Y)$
(resp., $D^i_{Y,I_k}p(X,Y)$) is an nc representation of the family defined by
$$
g_n = T_X^i p_n  ~~~ ( \text{resp.,}~ g_n = T_Y^i p_n)
$$
\end{remark}
\smallskip

\begin{theorem}[SVR Algorithm]\label{reduc:aug31:2010}
Let $p_1, \cdots, p_{k^2}$ be a homogeneous family of polynomials with an
nc representation $p(X,Y)$ of degree $i+j$ where
$i \ge 2$ and $j \ge 2$.
Suppose that the $k$ partitioned diagonal pairs
$$
\{x_{i_1},x_{j_1}\}, \cdots, \{x_{i_k},x_{j_k}\}
$$

\noindent
are known and $\varphi(i,j) \ne 0$. Also assume that
\begin{equation}\label{eq1:aug25:2010}
j\varphi(i,j) \ne (i+1)\varphi(i+1,j-1)
 \end{equation}
or
\begin{equation}\label{eq2:aug25:2010}
i\varphi(i,j) \ne (j+1)\varphi(i-1,j+1) .
\end{equation}
Then if \eqref{eq1:aug25:2010} (resp., \eqref{eq2:aug25:2010}) holds
the variables can be partitioned by Algorithm DiagPar1 and then positioned in
$Y$ (resp., $X$)  by Algorithm ParPosX applied to
$D^i_{X,I_k}p(X,Y)$ (resp., $D^i_{Y,I_k}p(X,Y)$).
\end{theorem}

\begin{proof}
Since the diagonal pairs are partitioned, we may assume that
$x_{i_1},\ldots,x_{i_k}$ are the diagonal entries of $X$ and
$x_{j_1},\ldots,x_{j_k}$ are the diagonal entries of $Y$. Moreover,
the nc polynomial
$$
p(X,Y)=\sum_{s=0}^{i+j} q_s(X,Y),
$$
where $q_s(X,Y)$ denotes the sum of the terms in $p(X,Y)$ that are of
degree $s$ in $X$ and degree $i+j-s$ in $Y$. Consequently,
\begin{equation}
\label{eq:oct24b11}
\begin{split}
(D_{X,I_k}^i p)(X,Y)&= \sum_{s=0}^{i+j} (D_{X,I_k}^i q_s)(X,Y) \\
&=\varphi(i,j)i! Y^j+(D_{X,I_k}^i q_{i+1})(X,Y)
+\sum_{s=i+2}^{i+j} (D_{X,I_k}^i q_s)(X,Y).
\end{split}
\end{equation}
Only the
first two terms on the right in the second line of (\ref{eq:oct24b11}) will
contribute two letter monomials of the form $ax_{j_s}^{j-1}x_u$. If $X$ is
replaced by $A$, $Y$ is replaced by $B$ and $AB=BA$, then these two terms
will be equal to
$$
i!\,\varphi(i,j)B^j+(i+1)!\,\varphi(i+1,j-1)AB^{j-1}
$$
Therefore, if $j(i!\varphi(i,j)) \ne (i+1)!\varphi(i+1,j-1)$, then we may
apply DiagPar1 to partition the variables. But this is the same as
(\ref{eq1:aug25:2010}).  Moreover, in view of Theorem \ref{thm1:sep9:2010},
Algorithm ParPosX, will then serve to position the variables in $Y$.

Similarly, if (\ref{eq2:aug25:2010}) is in force, then DiagPar1 and ParPosX
applied to $D_{Y,I_k}^j$ will serve to partition the variables and to
position them in $X$.
\end{proof}

\begin{remark}
\label{rem:nov4a11}
Theorem \ref{reduc:aug31:2010} remains valid if (\ref{eq1:aug25:2010})
and (\ref{eq2:aug25:2010}) are replaced by the conditions
$$
\varphi(i,j) \ne (i+1)\varphi(i+1,j-1;X)\quad\text{or}\quad
\varphi(i,j) \ne (i+1)\varphi(X;i+1,j-1).
$$
and
$$\varphi(i,j)\ne (j+1)\varphi(i-1,j+1:Y)\quad\textrm{or}\quad
\varphi(i,j)\ne (j+1)\varphi(Y;i-1,j+1),
$$
respectively, but with DiagPar2 in place of DiagPar1.  (In the first case,
DiagPar2 is applied to $D_{X,I_k}^ip$, in the second, it is applied to
$D_{Y,I_k}^jp$.
\end{remark}

The algorithm outlined in Theorem~\ref{reduc:aug31:2010} will be called the
{\bf SVR} (Single Variable Reduction) Algorithm.

\subsection{Algorithms based on two letter words}
\label{sec:25f11}
In this section we consider algorithms for families of polynomials $\cP$
that contain two letter words that are $\vtr$-equivalent to $x_i^sx_j^t$
with $i\ne j$ and $s\ge t\ge 2$. It is convenient to separately analyze the three
mutually exclusive cases 
$$ s\ge t+2,~~~~ s=t+1,~~ ~~ \text{and} ~~~~  s=t.$$ 
Our approach is to use either Alg.1 or Alg.2 in combination with the SVR algorithm.
Recall that Alg.1  refers to the sequential application of DiagPar1, ParPosX,  PosPol and
PosY, whereas Alg.2 refers to the sequential application of DiagPar2, ParPosX,  PosPol and
PosY.  

\subsubsection{Two letter monomials $\vtr$-equivalent to $x_u^sx_v^t$ with
$s \ge t+2$}\label{sec1:29nov11}
This subsection contains two results that provide nc representations
for families containing two letter monomials with $s \ge t+2 \ge 4$; the first
is based on Alg.1, the second on Alg.2.

Let $NC_{(\ref{eq:nov3a11})}$ denote the class of all nc polynomials
$p(X,Y)$ of degree $d$ such that for some set of integers $s,t$ with
\begin{equation}
\label{eq:nov3a11}
\left\{\begin{split}
s &\ge t+2 \ge 4, \quad
\varphi(s,t) \ne 0,\quad \varphi(s,t) \ne \varphi(t,s)\quad
\textrm{and either}\\
t\varphi(s,t)& \ne (s+1)\varphi(s+1,t-1)\
\textrm{or} \
s\varphi(s,t) \ne (t+1)\varphi(s-1,t+1).
\end{split}\right.
\end{equation}

\begin{proposition}
\label{prop:nov3a11}
Let $p_1,\ldots,p_{k^2}$ be a family $\cP$ of polynomials in $2k^2$ commuting
variables. Then the SVR algorithm coupled with Alg.1 will yield an
nc representation of $p(X,Y)$ of $\cP$ with $p$ in
$NC_{(\ref{eq:nov3a11})}$ if and only if the given family $\cP$ admits
an nc representation in this set.
\end{proposition}

\begin{proof}
In view of the assumptions in the first line of (\ref{eq:nov3a11}),
Lemma~\ref{lem:aug16a10} may be applied to obtain the partitioned diagonal
pairs. The constraints in the second line of (\ref{eq:nov3a11}) then insure
that SVR algorithm coupled with Alg.1 serves to partition and position the
variables and to position the given set of polynomials.  More precisely, if
$t\varphi(s,t) \ne (s+1)\varphi(s+1,t-1)$ the differentiation is with
respect to $X$; if
$s\varphi(s,t) \ne \\
(t+1)\varphi(s-1,t+1)$, then the differentiation is with
respect to $Y$. Moreover, if the differentiation in the SVR algorithm is with
respect to $X$, then the ParPosX Algorithm will position the variables in $Y$
and the PosPol Algorithm will position the polynomials.
On the other hand, if the differentiation in
the SVR algorithm is with respect to $Y$, then the ParPosX Algorithm will
position the variables in $X$.
Finally, if the given family does not admit an nc representation,
then the indicated algorithms cannot produce it.
\end{proof}

Let $NC_{(\ref{eq:nov3b11})}$ denote the class of all nc polynomials
$p(X,Y)$ of degree $d$ such that for some set of integers $s,t$ with
\begin{equation}
\label{eq:nov3b11}
\left\{\begin{split}
s &\ge t+2 \ge 4, \quad
\varphi(s,t) \ne 0,\quad \varphi(s,t) \ne \varphi(t,s)\quad\textrm{and}\\
\varphi(s,t)& \ne (s+1)\varphi(s+1,t-1;X)
 \
\textrm{or} \
\varphi(s,t) \ne (s+1)\varphi(X;s+1,t-1)\\
or& \\
\varphi(s,t)& \ne (t+1)\varphi(s-1,t+1;Y)
 \
\textrm{or} \
\varphi(s,t) \ne (t+1)\varphi(Y;s-1,t+1)\\
\end{split}\right.
\end{equation}

\begin{proposition}
\label{prop:nov3b11}
Let $p_1,\ldots,p_{k^2}$ be a family $\cP$ of polynomials in $2k^2$ commuting
variables. Then the SVR algorithm coupled with Alg.2 will yield an
nc representation of $p(X,Y)$ of $\cP$ with $p$ in
$NC_{(\ref{eq:nov3b11})}$ if and only if the given family $\cP$ admits
an nc representation in this set.
\end{proposition}

\begin{proof}
The proof is similar to the proof of Proposition \ref{prop:nov3a11},
except that Alg.2 is used in place of Alg.1.
\end{proof}

\subsubsection{Two letter monomials $\vtr$-equivalent to $x_u^sx_v^t$ with
$s = t+1$}\label{sec2:29nov11}
The result for this case is simpler than the case when $s \ge t+2 \ge 4$
because the condition
\begin{align}
s\varphi(s,t) \ne (t+1)\varphi(s-1,t+1)
\end{align}
required for Alg.1 to work after an application of the SVR algorithm  reduces
to
$$
\varphi(t+1,t) \ne \varphi(t,t+1)
$$
when $s=t+1$.  However, we always require that
$\varphi(t+1,t) \ne \varphi(t,t+1)$ so that we can successfully
partition the diagonal variables.  Therefore
our conditions to insure that we can successfully partition the dyslexic
diagonal pairs imply that Alg.1 will always be successful.

Let $NC_{(\ref{eq:nov3c11})}$ denote the class of all nc polynomials
$p(X,Y)$ of degree $d$ such that for some integer $t$ with
\begin{equation}
\label{eq:nov3c11}
\left\{\begin{split}
t &\ge 2, \quad
\varphi(t+1,t) \ne 0,\quad \varphi(t+1,t) \ne \varphi(t,t+1)
\quad\textrm{and}\\
\varphi(t+1,t) &\ne \textrm{the coefficient of $(XY)^tX$ or the coefficient
of $(YX)^tY$.}\\
\end{split}\right.
\end{equation}

\begin{proposition}
\label{prop:nov3c11}
Let $p_1,\ldots,p_{k^2}$ be a family $\cP$ of polynomials in $2k^2$ commuting
variables. Then the SVR algorithm coupled with Alg.1 will yield an
nc representation of $p(X,Y)$ of $\cP$ with $p$ in
$NC_{(\ref{eq:nov3c11})}$ if and only if the given family $\cP$ admits
an nc representation in this set.
\end{proposition}

\begin{proof}
The assumption that $\varphi(t+1,t) \ne 0$, $\varphi(t+1,t) \ne f_1$,
$\varphi(t+1,t) \ne f_2$, and \\
$\varphi(t+1,t) \ne \varphi(t,t+1)$
allow Lemma~\ref{lem:aug16b10} to determine the partitioned diagonal pairs.
The SVR Algorithm
can then be employed.  The above discussion implies that if we differentiate with respect to
$Y$ in the SVR Algorithm that the DiagPar1 Algorithm will successfully partition
the remaining variables between $X$ and $Y$.  Furthermore,
the ParPosX Algorithm will position $X$ and once $X$ is
determined, Algorithm PosPol
positions the polynomials in the reduced family, which positions the
polynomials in $p(X,Y)$. Once the polynomials are
positioned, then Algorithm PosY positions $Y$ and Algorithm NcCoef
will successfully determine $p(X,Y)$.
\end{proof}

\begin{remark}
We do not write out an analogous result based on Alg.2 because
in order to partition the diagonal elements in Alg.2 we require
that $\varphi(t+1,t) \ne \varphi(t,t+1)$, which insures that
Alg.1 will be successful.
\end{remark}

\subsubsection{Two letter monomials $\vtr$-equivalent to $x_u^sx_v^t$ with
$s  = t$}\label{sec3:29nov11}
The preceding two cases dealt with families of polynomials containing
two-letter monomials $\vtr$-equivalent to $x_i^sx_j^t$ with $s>t \ge 2$.
In those cases Lemma~\ref{lem:aug16a10} or Lemma~\ref{lem:aug16b10} was
applied first to determine the dyslexic diagonal pairs.  Then
Lemma~\ref{lem:aug16d10} was applied to partition
the diagonal elements.  However, if the only two-letter monomials are
$\vtr$-equivalent to $x_i^sx_j^t$ with  $s=t$, then  Lemma~\ref{lem:aug16d10}
is not applicable.
This section is devoted to developing an algorithm
for partitioning the diagonal entries and determining an nc
representation in this particular case.  The
main result of this subsection is Theorem \ref{thm:15aug11}. It is
convenient, however, to first establish a preliminary lemma:
\bigskip

\begin{lemma}\label{lem1:21sep10}
Let $p_1,\cdots,p_{k^2}$ be a family $\cP$ of polynomials in $2k^2$
commuting variables $x_1,\cdots , x_{2k^2}$
with an nc representation $p(X,Y)$.  Let $x_i$
and $x_j$ be a dyslexic diagonal pair and let $s_1,t_1$ and $s_2,t_2$ be
pairs of positive integers such that $s_1+t_1 = s_2 + t_2$
and either $s_1 = s_2$ or $s_1+1= s_2$ and suppose that $p(X,Y)$ satisfies the
following conditions when $s+t=s_1+t_1$:
\begin{itemize}
\item[(1)] If $s>0$, then
$\varphi(s,t+1;X) \ne 0$ if and only if $\varphi(s,t+1;Y) \ne 0$.
\vspace{2mm}
\item[(2)] If $s>0$, then
$\varphi(X;s,t+1) \ne 0$ if and only if $\varphi(Y;s,t+1) \ne 0$.
\vspace{2mm}
\item[(3)] $\varphi(0,s+t+1)= \varphi(s+t+1,0) = 0$.
\end{itemize}
Suppose further that $\alpha x_i^{s_1}x_j^{t_1}x_u$ and
$\beta x_i^{s_2}x_j^{t_2}x_v$ appear in some polynomial $p_r \in \cP$
where $\alpha \ne 0$ and $\beta \ne 0$ and let
$\ell$ and $m$ be the largest integers such that
$x_i^{s_1 - \ell}x_j^{t_1+\ell}x_u$ and $x_i^{s_2-m}x_j^{t_2+m}x_v$
appear in $p_r$. Then
$$
\ell \ge m \Longrightarrow x_u \in L(x_i) \quad and \quad  x_v \in L(x_j)
$$
and
$$
\ell < m \Longrightarrow x_v \in L(x_i) \quad and \quad  x_u \in L(x_j).
$$

\end{lemma}

\begin{proof} There are two steps:
\bigskip

\noindent
{\bf 1.} $\ell\ge m\Longrightarrow x_u\in L(x_{i})$:
If $\ell\ge m$ and $x_u\not\in L(x_{i})$,
then $x_u$ must belong to $L(x_{j})$, i.e., either $x_u\in R(x_{j})$ or
$x_u\in C(x_{j})$. But if $x_u\in R(x_{j})$ and $x_{i}$ and $x_{j}$
are in the $ss$ position of $X$ and $Y$, respectively, then  $x_u$ is in the
$st$ position of $Y$ and $x_v$ is in the $st$ position of
$X$ (as $\alpha x_{i}^{s_1}x_{j}^{t_1}x_u$ and
$\beta x_{i}^{s_2}x_{j}^{t_2}x_v$ are in the same polynomial $p_r$).
Therefore, $\varphi(s_1-\ell,t_1+\ell+1;Y)\ne 0$. If $s_1=\ell$, then this
contradicts (3) and therefore is not a viable possibility. If $s_1>\ell$,
then $\varphi(s_1-\ell,t_1+\ell+1;Y)\ne 0$ and, (1) implies that
$\delta= \varphi(s_1 - \ell,t_1+ \ell+1;X) \ne 0$. Thus,
$\delta x_{i}^{s_1-(\ell+1)}x_{j}^{t_1+\ell+1}x_v$  belongs to $p_r$, i.e.,
\begin{equation*}
\begin{split}
s_1=s_2&\Longrightarrow
\delta x_{i}^{s_2-(\ell+1)}x_{j}^{t_2+\ell+1}x_v\quad
\textrm{belongs to $p_r$}\\
s_1=s_2-1&\Longrightarrow
\delta x_{i}^{s_2-(\ell+2)}x_{j}^{t_2+\ell+2}x_v\quad\textrm{belongs to $p_r$}
\end{split}
\end{equation*}
However, the definition of $m$ implies that $\ell+1\le m$ in the first case
and $\ell+2\le m$ in the second, both of which clearly contradict the
assumption that $\ell \ge m$. Therefore,
$x_u\not\in R(x_{j})$.

A similar argument based on (2)
serves to prove that $x_u\not\in C(x_{j})$. Therefore, $x_u\in L(x_{i})$
as claimed. This completes the proof of {\bf 1}.
\bigskip

\noindent
{\bf 2.} $m>\ell\Longrightarrow x_v\in L(x_{i})$: If $m>\ell$ and
$x_v\not\in L(x_i)$, then $x_v\in L(x_j)$. Suppose that in fact
$x_v\in R(x_j)$. Then $\varphi(s_2-m,t_2+m+1;Y)\ne 0$.
If $s_2=m$, then this contradicts (3); if $s_2>m$, then
$\varphi(s_2-m,t_2+m+1;Y)\ne 0$ and, by (1),
$\gamma = \varphi(s_2-m,t_2+m+1;X)\ne 0$ and hence that the monomial
$\gamma x_i^{s_2-m-1}x_j^{t_2+m+1}x_u$ is in this polynomial, i.e.,
\begin{equation*}
\begin{split}
s_1=s_2&\Longrightarrow
\gamma x_{i}^{s_1-(m+1)}x_{j}^{t_1+m+1}x_u\quad
\textrm{belongs to $p_r$}\\
s_1=s_2-1&\Longrightarrow
\gamma x_{i}^{s_1-m}x_{j}^{t_1+m}x_u\quad\textrm{belongs to $p_r$}
\end{split}
\end{equation*}
Therefore, the definition of $\ell$ implies that $m+1\le\ell$ in the first
case and $m\le\ell$ in the second, which clearly contradicts the
assumption that $\ell<m$.
Consequently, $x_v\in L(x_i)$ as claimed. A similar argument rules out the
case  $x_v\in C(x_j)$.
\bigskip

\end{proof}

Let $NC_{(\ref{eq:nov17a11})}$ denote the class of nc
polynomials $p(X,Y)$ of degree $d\ge 4$ such that
\begin{equation}\label{eq:nov17a11}
\left\{\begin{split}
r &\ge 2, \quad
\varphi(r,r) \ne 0,\quad \varphi(2r,0)= \varphi(0,2r) = 0
\quad\textrm{and}\\
 & \textrm{if $s>0$, $t>0$ and $s+t=2r$, then}\hspace{2mm} \varphi(s,t:X)\ne 0
\Longleftrightarrow \varphi(s,t:Y)\ne 0,\\
 & \textrm{if $s>0$, $t>0$ and $s+t=2r$, then} \hspace{2mm} \varphi(X;s,t)\ne 0
\Longleftrightarrow \varphi(Y;s,t)\ne 0.
\end{split}\right.
\end{equation}

\begin{theorem}\label{thm:15aug11}
Let $p_1,\cdots,p_{k^2}$ be a family $\cP$ of polynomials in $2k^2$
commuting variables $x_1,\cdots,x_{2k^2}$ of degree $d\ge 4$. Then Lemma
\ref{lem1:21sep10}, the SVR Algorithm and either Alg.1 or Alg.2 will yield an nc representation
$p(X,Y)$ with $p$ in the class $NC_{(\ref{eq:nov17a11})}$ if and only if the
given family admits an nc representation in this class.
\end{theorem}

\begin{proof}
Under the given assumptions Lemma~\ref{lem:aug24a10} may be applied to
obtain the dyslexic diagonal pairs
$$
\{x_{i_1},x_{j_1}\},\ldots, \{x_{i_k},x_{j_k}\}.
$$
To ease the notation, we assume that $x_{i_s}$ and $x_{j_s}$ are in
the $ss$ position for $s=1,\ldots,k$ and that one of these pairs is
partitioned, i.e., for some fixed choice of $s$, $x_{i_s}\in X$ and
$x_{j_s}\in Y$.

The objective is to determine $L(x_{i_s})$ and $L(x_{j_s})$ for each dyslexic diagonal pair.
Then for $s$ fixed, there must exist at least $k-1$ other dyslexic diagonal variables $\{ x_{i_1}, \cdots, x_{i_{s-1}},x_{i_{s+1}}, \cdots, x_{i_k}\}$
which have the property that
\begin{align}\label{eq1:17apr12}
L(x_{i_t}) \cap L(x_{i_s}) \ne \emptyset \hspace{2mm} \text{for} \hspace{2mm} t \ne s.
\end{align}
Equation \eqref{eq1:17apr12} implies that
the variables $\{x_{i_1}, \cdots, x_{i_{s-1}},x_{i_{s+1}}, \cdots, x_{i_k}\}$
lie on the diagonal of $X$ with $x_{i_s}$ and that the union of the $L(x_{i_t})$ contains all of the
variables in $X$.  Therefore, this process
serves to partition the variables between $X$ and $Y$.   Once this is done, the
SVR Algorithm and the final steps in Alg.1 or Alg.2 beginning with Algorithm ParPosX
will determine an nc representation for $\cP$.

Let $p_{st}$ denote the polynomial in the $st$ position with $s\ne t$ in
the array corresponding to $p(X,Y)$.
Then $p_{st}$, will be of the form
\begin{equation}
\label{eq:nov10c11}
\begin{split}
p_{st}=&x_{i_s}^{r-1}x_{j_s}^r(ax_u+cx_v)+x_{i_t}^{r-1}x_{j_t}^r(bx_u+dx_v)\\
&\quad +x_{i_s}^rx_{j_s}^{r-1}(gx_u+ex_v)
+x_{i_t}^rx_{j_t}^{r-1}(hx_u+fx_v)+\cdots.
\end{split}
\end{equation}
Given that $\varphi(r,r) \ne 0$, we must have that $\varphi(r,r;X) \ne 0$ and $\varphi(r,r;Y) \ne 0$.
and that one of the following cases must
hold:
\bigskip

{\bf 1.} $a \ne 0, e\ne 0$ and $c = g = 0$,
\smallskip

{\bf 2.} $a = 0, e = 0$, and $c\ne 0, g \ne 0$,
\smallskip

{\bf 3.} $a \ne 0, e \ne 0$ and $c \ne 0, g \ne 0$.
\bigskip

\noindent
Similarly, the assumption that $\varphi(r,r) \ne 0$ implies that $\varphi(X;r,r) \ne 0$ and
$\varphi(Y;r,r) \ne 0$ implies that one the following must also hold:
\bigskip

{\bf 4.} $b \ne 0, f\ne 0$ and $d = h = 0$,
\smallskip

{\bf 5.} $b = 0, f = 0$, and $d\ne 0, h \ne 0$,
\smallskip

{\bf 6.} $b \ne 0, f \ne 0$ and $d \ne 0, h \ne 0$.
\bigskip

\noindent
If case {\bf 1} holds, we have that $a = \varphi(r,r;X)$ and $e = \varphi(r,r;Y)$, and
consequently, that $x_u \in R(x_{i_s})$ and $x_v \in R(x_{j_s})$.  If case {\bf 2} holds
we may conclude that $c = \varphi(r,r;X)$ and $g = \varphi(r,r;Y)$ and that
$x_v \in R(x_{i_s})$ and $x_u \in R(x_{j_s})$.  Finally, if case {\bf 3} holds, we must
resort to more subtle measures to partition $x_u$ and $x_v$.  Here we can apply
Lemma~\ref{lem1:21sep10} with $s_1 = s_2 =r-1$ and $t_1 = t_2 =r$ to partition
$x_u$ and $x_v$ between $R(x_{i_s})$ and $R(x_{j_s})$.  Therefore, by
ranging over $t$ in the polynomials $p_{st}$ for a fixed $s$, we will be able to entirely determine both $R(x_{i_s})$ and
$R(x_{j_s})$ using the above analysis.

We may similarly analyze cases {\bf 4} - {\bf 6} and range over $s$ in the polynomials $p_{st}$ for
a fixed $t$ to determine $C(x_{i_t})$ and
$R(x_{j_t})$.  Thus, by ranging over all polynomials in $\cP$, we will be able to determine
$L(x_{i_s})$ and $L(x_{j_s})$ for each dyslexic pair $x_{i_s}$ and $x_{j_s}$ with $1 \le s \le k$.
By determining which of these sets have a nontrivial intersection, we will obtain a partitioning of the
variables $x_1, \cdots, x_{2k^2}$ between $X$ and $Y$.

\end{proof}


\begin{remark}
\label{rem:aug13a10}
The indicated terms in $p_{st}$ in the first part of the preceding proof
can also be grouped as
\begin{eqnarray*}
p_{st}&=&x_u(ax_{i_s}^{r-1}x_{j_s}^r+bx_{i_t}^{r-1}x_{j_t}^r
+gx_{i_s}^rx_{j_s}^{r-1}+hx_{i_t}^rx_{j_t}^{r-1})\\ &{}&
+x_v(cx_{i_s}^{r-1}x_{j_s}^r+dx_{i_t}^{r-1}x_{j_t}^r+ex_{i_s}^rx_{j_s}^{r-1}+
fx_{i_t}^rx_{j_t}^{r-1})+\cdots .
\end{eqnarray*}
\end{remark}

\subsubsection{ Two letter monomials $\vtr$-equivalent to $x_u^sx_v$}

The next result presents another way of determining an nc polynomial
representation for families of polynomials containing two letter monomials of
the form $ax_u^sx_v$ with $s \ge 2$ and $a\in\RR\setminus\{0\}$.

Let $NC_{(\ref{eq:dec9a11})}$ denote the class of nc
polynomials of degree $d \ge 4$ such that

\begin{equation}
\label{eq:dec9a11}
\left\{\begin{split}
&\varphi(d-1,1) \ne \varphi(1,d-1), \quad \\
&\varphi(d-1,1) \ne \varphi(d,0) \quad \text{if}\quad  \varphi(d-1,1) \ne 0,\\
&\varphi(1,d-1) \ne \varphi(0,d) \quad \text{if} \quad \varphi(1,d-1)  \ne 0,\\
&\varphi(Y;d-1,1) \ne 0, \quad  \varphi(d-1,1;Y) \ne 0,\\
&\varphi(X;1,d-1) \ne 0, \quad \varphi(1,d-1;X) \ne 0.
\end{split}\right.
\end{equation}

\begin{proposition}\label{prop2:30nov11}
Let $p_1,\ldots,p_{k^2}$ be a family $\cP$ of polynomials in $2k^2$ commuting
variables. Then the SVR algorithm will yield an nc
representation of $p(X,Y)$ of $\cP$ with $p$ in
$NC_{(\ref{eq:dec9a11})}$ if and only if the given family $\cP$ admits
an nc representation in this set.
\end{proposition}

\begin{proof}
Suppose first that $\varphi(d-1,1)\ne 0$.
The assumptions in the last two lines of (\ref{eq:dec9a11}) guarantee that
$k^2-k$ polynomials will contain four or more terms $\vtr$-equivalent to
$x_u^{d-1}x_v$ with $x_v\ne x_u$, whereas the assumption that
$\varphi(d-1,1) \ne \varphi(1,d-1)$ guarantees  that exactly $k$ polynomials
$p_{i_1}, \ldots, p_{i_k}$ will contain
either one or two monomials $\vtr$-equivalent to $x_u^{d-1}x_v$ with $x_u\ne x_v$.
This allows us to identify them as diagonal entries and to partition
them between $X$ and $Y$.  Therefore, we apply the SVR
Algorithm to $\cP$ by differentitating $d-1$ times to obtain the
family $\{T_X^{d-1}p_1, \cdots, T_X^{d-1}p_{k^2}\}$
with nc representation
$$
D_{X,I_k}^{d-1}p(X,Y) = (d-1)!\varphi(d,0)X+(d-1)!\varphi(d-1,1)Y.
$$
Since $\varphi(d,0)\ne\varphi(d-1,1)$ and $\varphi(d-1,1)\ne 0$ by
assumption, we can partition the remaining $k^2-2k$
variables between $X$ and $Y$.

The construction of an nc polynomial representation can now be completed by
invoking  the algorithms ParPosX,PosPol, PosY and Algorithm NcCoef.
The remaining case follows similarly.
\end{proof}

\subsection{Summary of Two-letter Algorithms} \label{sec4:29nov11}
We now summarize our results based on the
analysis of two-letter monomials.
The methods of this section bear on
$p(X,Y)$ of the form
\begin{align}\label{eq1:aug20:2011}
p(X,Y) = d_1(XY)^t + d_2(YX)^t + d_3(XY)^tX + d_4(YX)^tY + q(X,Y),
\end{align}
where $q(X,Y)$ is an nc polynomial containing no multiples
of the first four monomials in \eqref{eq1:aug20:2011}.
The effectiveness of the procedures is summarized by:

\begin{theorem}
\label{thm:2varMain}
Let $p_1,\ldots,p_{k^2}$ be a family $\cP$ of polynomials in $2k^2$ commuting
variables $x_1,\ldots,x_{2k^2}$ and let $\cQ$ denote the set of
nc polynomials $p(X,Y)$ of degree $d>1$ that satisfy the
properties  in at least one of the following three lists:

\begin{itemize}

\item[({\bf 1})]
	
For some $s,t \in \mathbb{N}$, $s \ge t+2 \ge 4$, $\varphi(s,t) \ne 0$, and $\varphi(s,t) \ne \varphi(t,s)$.
Additionally, assume that $p(X,Y)$ satisfies one of the following conditions:	
\begin{enumerate}
\item[(1)] $t\varphi(s,t) \ne (s+1)\varphi(s+1,t-1), $

\item[(2)] $s\varphi(s,t) \ne (t+1)\varphi(s-1,t+1) $,

\item[(3)] $\varphi(s,t) \ne \varphi(s+1,t-1;X) \quad \text{ or } \quad \varphi(s,t) \ne \varphi(X;s+1,t-1)$,

\item[(4)]$\varphi(s,t) \ne \varphi(s-1,t+1;Y) \quad \text{ or } \quad \varphi(s,t) \ne \varphi(Y;s-1,t+1).$

\end{enumerate}

\medskip

\item[({\bf 2})]

For some $t \ge 2$,  $\varphi(t+1,t) \ne 0$ and
additionally assume that $p(X,Y)$ satisfies the following properties:
\begin{itemize}
\item[(1)] $\varphi(t+1,t) \ne d_3$
\item[(2)] $\varphi(t+1,t) \ne d_4$
\item[(3)] $\varphi(t+1,t) \ne \varphi(t,t+1)$
\end{itemize}

\medskip

\item [({\bf 3})]
For some $r \ge 2$,  $\varphi(r,r) \ne 0$ and
additionally assume that
 the nc representation also satisfies the following properties:
\begin{itemize}
\item[(1)] $\varphi(0,2r) =\varphi(2r,0) = 0$
\item[(2)] $\varphi(r,r) \ne d_1$
\item[(3)] $\varphi(r,r) \ne d_2$
\item[(4)] $\varphi(s,t;X) \ne 0 \Longleftrightarrow \varphi(s,t;Y) \ne 0$ for all $s,t$ such that $s>0, t>0$, $s+t =2r$.
\item[(5)] $\varphi(X;s,t) \ne 0 \Longleftrightarrow \varphi(Y;s,t) \ne 0$ for all $s,t$ such that $s>0, t>0$, $s+t =2r$.
\end{itemize}

\end{itemize}

\noindent
Then the two letter algorithms developed in Section~\ref{sec:25f11}
determine an
nc representation $p(X,Y)$ for $\cP$ in $\cQ$  if and only if $\cP$ has a
 representation in the class $\cQ$.
\end{theorem}

\begin{proof}
Conditions (1), (2) and (3) are exactly what was needed to make the algorithms described
 \S \ref{sec1:29nov11},\S \ref{sec2:29nov11}, and \S \ref{sec3:29nov11} effective.
\end{proof}

We can now supply a proof that our algorithms work for a collection of families of commutative
polynomials that is in direct correspondence with a generic subset
of the space of nc polynomials.

\subsubsection{Proof of Theorem~\ref{genTheorem}}
Suppose that $\mathcal{U}$ is a subspace of the space $\cW$ of nc
polynomials of degree $d \ge 4$.  Then
$\mathcal{U}$ must contain an nc monomial of one of the following forms as a
basis element:
$$
m_{\alpha,\beta}(X,Y)\quad\textrm{with $\alpha=(\alpha_1,\cdots,\alpha_n)$,
$\beta=(\beta_1,\cdots,\beta_n)$ and $|\alpha|+|\beta|=d$},
$$
where $\alpha_1,\ldots, \alpha_n$ and $\beta_1,\ldots,\beta_n$ are positive
integers, except that $\alpha_1$ and $\beta_n$ are permitted to be equal to
zero.  Additionally, $\alpha$ and $\beta$ must satisfy one of the following
conditions:
\begin{enumerate}
\item[\rm(1)] $|\alpha|=d$ and $|\beta|=0$\quad or \quad$|\alpha|=0$ and
$|\beta|=d$,
 \smallskip
\item[\rm(2)] $|\alpha|>0$, $|\beta|>0$ and either $|\alpha|>|\beta|+1,$
\quad or \quad $|\beta|\ge |\alpha|+1,$
 \smallskip
\item[\rm(3)] $|\alpha|>0$, $|\beta|>0$ and $|\alpha|=|\beta|$,
 \smallskip
\item[\rm(4)] $|\alpha|>0$, $|\beta|>0$ and either $|\alpha|=1$\quad or\quad
$|\beta|=1$.
\end{enumerate}

If case (1) occurs, and $p \in \mathcal{U}$ is a polynomial only in the single
variable $X$ or $Y$,
then the single letter algorithms will determine $p$ and the set
$$\mathcal{S}_1 =
 (NC_{(\ref{eq:nov6a11})} \cup NC_{(\ref{eq:nov10a11})}\cap \mathcal{U}
$$ will be
open and dense in $\mathcal{U}$ since the indicated constraints are
inequalities.

Similarly, in case $(2)$ the set
$$\mathcal{S}_2 =( NC_{(\ref{eq:nov3a11})}\cup NC_{(\ref{eq:nov3c11})}
\cap \mathcal{U}$$
will be dense in $\mathcal{U}$

If case $(3)$ occurs and $\mathcal{U}$  contains
either the term $X^d$ or $Y^d$ as a basis element, then case $(1)$
applies.  If $\mathcal{U}$ does not contain this term as
a basis element, then the set
$$
\mathcal{S}_3 = NC_{(\ref{eq:nov17a11})}\cap \mathcal{U}
$$
is open and dense in $\mathcal{U}$.

If case $(4)$ occurs, then the set
$$
\mathcal{S}_4 = NC_{(\ref{eq:dec9a11})} \cap \mathcal{U}
$$
is an open dense set in $\mathcal{U}$ given that the defining constraints
are inequality constraints.
\qed

\subsection{Uniqueness results for two-letter algorithms}

The family $\mathcal{Q}$  defined in Theorem~\ref{thm:2varMain}
provides us with a large collection of nc polynomials for which our
two letter algorithms will be successful.  We now investigate
the uniqueness properties of families with a representation in
$\mathcal{Q}$.

\begin{theorem}\label{thm:2varUnique}
Suppose that $\cP$ is a family of $k^2$ polynomials in $2k^2$ commuting
variables that admits
two nc representations $p(X,Y)$ and $\wtilde{p}(\wtilde{X}, \wtilde{Y})$
in the family $\mathcal{Q}$.
Then the matrices $X,Y$ and $\wtilde{X},\wtilde{Y}$ are permutation equivalent
(as defined in \eqref{eq:feb23a12}).
\end{theorem}

\begin{proof} Suppose that (1) of \eqref{eq:feb23a12}) is satisfied
and let $\textup{diag}\{X\}$ denote the diagonal entries of the matrix $X$.
The assumption that both $p$ and $\wtilde{p}$ are
in $\mathcal{Q}$ implies that either

(a) $\textup{diag}\{X\}=\textup{diag}\{\wtilde{X}\}$ and
$\textup{diag}\{Y\}=\textup{diag}\{\wtilde{Y}\}$
\ \ \
or

(b) $\textup{diag}\{X\}=\textup{diag}\{\wtilde{Y}\}$ and
$\textup{diag}\{Y\}=\textup{diag}\{\wtilde{X}\}$.

In case (a), if
$n\varphi(s,n) \ne (s+1)\varphi(s+1,n-1)$, then
Remark~\ref{rem1:aug31:2010} implies that $D_{Y,I_k}^s p(X,Y)$
(resp., $D_{\wtilde{Y},I_k}^s \wtilde{p}(\wtilde{X},\wtilde{Y})$)
is an nc representation of the family $T_{Y}^sp_1, \cdots, T_{Y}^s p_{k^2}$
(resp., $T_{\wtilde{Y}}^s p_1, \cdots, T_{\wtilde{Y}}^s p_{k^2}$).  Furthermore,
given that case (a) holds, the diagonals of $Y$ and $\wtilde{Y}$ are the same, which
implies that the families $T_{Y}^sp_1, \cdots, T_{Y}^s p_{k^2}$ and
$T_{\wtilde{Y}}^s p_1, \cdots, T_{\wtilde{Y}}^s p_{k^2}$ are the same.

Therefore
$D_{Y,I_k}^s p(X,Y)$ and
$D_{\wtilde{Y},I_k}^s \wtilde{p}(\wtilde{X},\wtilde{Y})$ are
both nc representations of the family $T_{Y}p_1, \cdots, T_{Y}p_{k^2}$.
Moreover,
given that $p, \wtilde{p} \in \mathcal{Q}$, the nc polynomials
$D_{Y,I_k}^s p(X,Y)$ and
$D_{\wtilde{Y},I_k}^s \wtilde{p}(\wtilde{X},\wtilde{Y})$
are both in the set $\mathcal{W}$ that is defined in
Theorem~\ref{thm:nov25a11}.
Therefore, we may apply Theorem~\ref{cor:uniqIntro} to conclude that there
exists a permutation matrix $\Pi$ such that either
$X = \Pi^T\wtilde{X} \Pi$ and $Y = \Pi^T\wtilde{Y}\Pi$ or
$X = \Pi^T\wtilde{X}^T \Pi$ and $Y = \Pi^T\wtilde{Y}^T\Pi$.

In case (b), $X = \Pi^T\wtilde{Y} \Pi$ and $Y = \Pi^T\wtilde{X}\Pi$
or $X = \Pi^T\wtilde{Y}^T \Pi$ and $Y = \Pi^T\wtilde{X}^T\Pi$.
The other three cases in \eqref{eq:feb23a12} are handled in much the same way.
\end{proof}

\bigskip

\newpage
\centerline{NOT FOR PUBLICATION}
\tableofcontents

\end{document}